\documentclass[12pt]{article}
\usepackage{amsmath,amssymb,latexsym, amsfonts, amscd, amsthm, tikz, tikz-cd}
\usepackage{graphicx,epsfig}

\theoremstyle{plain}
\newtheorem{theorem}{Theorem}[section]

\newtheorem{proposition}[theorem]{Proposition}
\newtheorem{corollary}[theorem]{Corollary}

\newtheorem{assumption}[theorem]{Assumption}
\newtheorem{lemma}[theorem]{Lemma}

\theoremstyle{definition}
\newtheorem{definition}[theorem]{Definition}
\newtheorem{remark}[theorem]{Remark}

\newcommand{\bT}{{\mathbb T}}

\newcommand{\lra}{\longrightarrow}

\DeclareMathOperator{\Imm}{Im}
\DeclareMathOperator{\Sym}{Sym}

\newcommand{\tensor}{{\otimes}}

\newcommand{\Spec}{\mbox{Spec}}

\newcommand{\Spa}{\mbox{Spa}}

\newcommand{\Ker}{\mbox{Ker}}

\newcommand{\GL}{{\rm \mathbf{GL}}}

\newcommand{\Z}{{\mathbb Z}}
\newcommand{\Q}{{\mathbb Q}}
\newcommand{\C}{{\mathbb C}}
\newcommand{\R}{{\mathbb R}}

\newcommand{\N}{{\mathbb N}}

\newcommand{\XSinf}{{\rm (X/S)_{\rm inf}^{\rm log}}}

\newcommand{\XSinfx}{\left(X/S\right)_{\rm inf |_X}^{\rm log}}
\newcommand{\WW}{{\mathbb W}}

\newcommand{\Xinf}{{\rm Sh\bigl((X/S)_{\rm inf}^{\rm log}}\bigr)}
\newcommand{\Xinfx}{X_{\rm inf,X}}
\newcommand{\Uinf}{U_{\rm inf}^{\rm log}}

\newcommand{\uno}{{\bf{1}}}
\newcommand{\Eq}{{\rm Eq}}
\newcommand{\CoEq}{{\rm CoEq}}
\newcommand{\colim}{{\rm colim}}
\newcommand{\Psh}{{\rm Psheaves}}

\newcommand{\uS}{{\underline{S}}}

\newcommand{\Xbar}{{\overline{X}}}

\newcommand{\uX}{{\underline{X}}}

\newcommand{\uY}{{\underline{Y}}}

\newcommand{\cL}{{\mathbb L}}

\newcommand{\cM}{{\cal M}}
\newcommand{\cH}{{\cal H}}
\newcommand{\cC}{{\cal C}}
\newcommand{\cI}{{\cal I}}

\newcommand{\cU}{{\cal U}}
\newcommand{\cP}{{\cal P}}

\newcommand{\cT}{{\cal T}}
\newcommand{\cF}{{\cal F}}
\newcommand{\cG}{{\cal G}}
\newcommand{\cE}{{\cal E}}
\newcommand{\cV}{{\cal V}}
\newcommand{\cA}{{\cal A}}
\newcommand{\cB}{{\cal B}}
\newcommand{\PP}{{\mathbb P}}

\newcommand{\bV}{{\mathbb V}}
\newcommand{\bH}{{\mathbb H}}
\newcommand{\bA}{{\mathbb A}}
\newcommand{\bB}{{\mathbb B}}

\newcommand{\bE}{{\mathbb E}}

\newcommand{\cS}{{\cal S}}
\newcommand{\cD}{{\cal D}}
\newcommand{\cO}{{\cal O}}

\newcommand{\uU}{\underline{U}}

\newcommand{\bG}{{\mathbb G}}

\newcommand{\Sh}{{\mbox{Sh}}}

\newcommand{\cW}{{\cal W}}
\newcommand{\cZ}{{\cal Z}}
\newcommand{\cIG}{{\cal IG}}

\newcommand{\de}{{\rm d}}
\newcommand{\dual}{^\vee}
\newcommand{\oA}{\overline{A}}
\newcommand{\oE}{\overline{\cE}}

\newcommand{\fW}{{\mathfrak{W}}}

\newcommand{\fIG}{{\mathfrak{IG}}}
\newcommand{\aW}{{\mathfrak{W}_k^{\rm alg}}}
\newcommand{\fT}{{\mathfrak{T}}}
\newcommand{\fw}{{\mathfrak{w}}}

\newcommand{\fX}{{\mathfrak{X}}}

\newcommand{\fK}{{\mathfrak{K}}}
\newcommand{\fS}{{\mathfrak{S}}}

\newcommand{\fp}{{\mathfrak{p}}}

\newcommand{\fm}{{\mathfrak{m}}}

\newcommand{\fg}{{\mathfrak{g}}}
\newcommand{\frn}{{\mathfrak{n}}}
\newcommand{\fh}{{\mathfrak{h}}}
\newcommand{\fb}{{\mathfrak{b}}}
\newcommand{\fn}{{\mathfrak{n}}}

\newcommand{\fu}{{\mathfrak{u}}}

\newcommand{\ho}{\hat{\otimes}}

\newcommand{\cX}{\mathcal{X}}

\begin{document}

\bigskip
\title{BGG-decomposition for de Rham sheaves.}
\author{Fabrizio Andreatta, \\ Marco Baracchini, \\ Adrian Iovita.}
\maketitle

\tableofcontents \pagebreak

\section{Introduction}
\label{sec:intro}

The expression ``BGG-decomposition" in the title of this article refers to the work of Bernstein-Gelfand-Gelfand, \cite{bernstein_gelfand_gelfand}, during the 70's, where these authors realized and explained how to compute Lie-algebra cohomology of a finite, irreducible Lie-algebra representation.

Given a semi-simple Lie-algebra $\fg$ over $\C$, a Borel sub-Lie-algebra $\fb$ of it and a finite, complex, irreducible representation $V$ of $\fg$, one has the natural Koszul complex: $$K_\bullet:= U(\fg)\otimes_{U(\fb)}(V\vert_{U(\fb)}\otimes_{\C} \wedge^\bullet\fg/\fb),$$ where $U(\fg)$, $U(\fb)$ are the universal enveloping algebras of $\fg$, respectively $\fb$. See section \ref{sec:outline} for a description of the differentials of $K_\bullet$. This complex gives a resolution of $V$ by projective $U(\fg)$-modules.

The authors of the above mentioned article, understood that there is a way of cutting down the complex $K_\bullet$ by using the center 
$Z(\fg)$ of $U(\fg)$ as follows: the irreducible representation $V$ has a highest weight vector for a weight $\lambda\in \fh^\vee$, i.e. the action of the Cartan sub-algebra $\fh$ of $\fg$ is given via the character $\lambda$ of $\fh$. Let $\chi_\lambda$ be the character of $Z(\fg)$, giving the action of $Z(\fg)$ on the Verma module
$U(\fg)\otimes_{U(\fb)}\C(\lambda)$. If $W$ is a $U(\fg)$-module, we denote by $W_{\chi_\lambda}$ the subspace consisting of $w\in W$ such that for every $\sigma\in Z(\fg)$, there is $m\ge 0$ with $\bigl(\sigma-\chi_\lambda(\sigma)\bigr)^mw=0$, i.e., $W_{\chi_\lambda}$ is the generalized eigenspace of $\sigma\in Z(\fg)$ acting on $W$,  with eigenvalues $\chi_\lambda(\sigma)$, for all $\sigma\in Z(\fg)$. Then the sub-complex $(K_\bullet)_{\chi_\lambda}:=\bigl(U(\fg)\otimes_{U(\fb)}(V\otimes\wedge^\bullet \fg/\fb)\bigr)_{\chi_\lambda}$ of $K_\bullet$, is in fact a direct summand, and it is another resolution of $V$ by projective $U(\fg)$-modules, called now the BGG-complex of $V$. Therefore, the $BGG$-complex of $V$ is quasi-isomorphic to the Koszul complex.
The terms of the BGG complex have been completely determined in \cite{bernstein_gelfand_gelfand} in terms of the dot action of the Weyl group on the weight $\lambda$.  

It was remarked in \cite{bernstein_gelfand_gelfand}, (see Remark 2, section \S 9) that, if there are complex Lie-groups $G$, $B$, with $B$ a Borel subgroup of $G$, such that $\fg={\rm Lie}(G)$, $\fb={\rm Lie}(B)$, then for $V=\C$, the trivial $\fg$-representation, the dual of the complex $K_\bullet$ can be reinterpreted as the de Rham complex of the analytic flag variety $G/B$, at the point corresponding to the identity $e$ of $G/B$.

This observation was largely developed by Faltings in \cite{faltings} and Faltings-Chai in \cite{faltings_chai}, where they considered a 
Siegel variety $X$ for a symplectic algebraic group $G$ and a parabolic subgroup $P$ of it; $X$ is seen as a complex analytic variety. Associated to an irreducible, complex representation $V$ of $P$, one has a vector bundle  $\cV$ on $X$ associated to it. If now $V$ is a finite, complex $G$-representation, the vector bundle $\cV$ has a natural connection and therefore we have a de Rham complex $\cV\otimes_{\cO_X}\Omega^\bullet_{X/\C}$ attached to $V$. In \cite{faltings_chai}, the authors use BGG-techniques to simplify the de Rham complex and compute its de Rham cohomology as follows. Let $V^\vee$ be the dual of the representation $V$, and suppose that it is an irreducible representation of the Lie-algebra  $\fg:={\rm Lie}(G^o)$, where $G^o$ is the kernel of the similitude character of $G$. We consider its Koszul complex $K_\bullet$. Using the center $Z(\fg)$ of $U(\fg)$ and the character $\lambda^\vee$ of
$V^\vee$, we obtain, following the ideas of Bernstein-Glefand-Gelfand presented above, the BGG-complex of $K_\bullet$. This complex has the form
$(K_\bullet)_{\chi_{\lambda^\vee}}=U(\fg)\otimes_{U(\fp)}M_\bullet$, with $M_\bullet$ the sum of finite irreducible $U(\fp)$-modules, whose weights are explicitly determined in terms of the set of Kostant representatives of the quotient set $W_G/W_L$, for $L$  the Levi subgroup of $P$.

Now they use the following crucial observation (Proposition 5.1, \cite{faltings_chai}): let $E_1$, $E_2$  be complex representations of $P$ and denote by $\cE_1$, $\cE_2$ the associated vector bundles on $X$. Then there is a contravariant, functorial isomorphism 
$$
(\ast)\quad {\rm Diff}(\cE_1,\cE_2)\cong {\rm Hom}_{\fg}\bigl( U(\fg)\otimes_{U(\fp)}E_2^\vee, U(\fg)\otimes_{U(\fp)}E_1^\vee\bigr),
$$
where ${\rm Diff}(\cE_1,\cE_2)$ is the module of homogeneous differential operators of finite order from $\cE_1$ to $\cE_2$.

Therefore the $\fg$-morphism $K_{\bullet, \chi_{\lambda^\vee}}\to K_{\bullet}$ gives a morphism of complexes of vector bundles (they are not $\cO_X$-linear but differential operators) $\cV\otimes_{\cO_X}\Omega^\bullet_X\lra \cM^\bullet$, where 
$\cM^q$ is the vector bundle associated to the $P$-representation whose $\fp$-module is $M_q^\vee$. We'll denote ${\rm BGG}^\bullet$ the complex $\cM^\bullet$ and call it the BGG-complex associated to $V$.

Then \cite{faltings_chai} show that the above morphism of complexes of vector bundles on $X$ is a quasi-isomorphism and it stays a quasi-isomorphism for the natural extensions of these complexes to smooth toroidal compactifications $\overline{X}$ of $X$.   
Therefore the hypercohomology of the complex ${\rm BGG}^\bullet$ computes the de log Rham cohomology of the complex $\cV\otimes_{\cO_X}\omega^\bullet_{\overline{X}}$.

Faltings-Chai made the same constructions for algebraic Shimura varieties over number fields or local fields, or even rings of integers of local fields (under certain assumptions). No proof of why the BGG-complex is quasi-isomorphic to the de Rham complex is given in these cases, 

There have been a number of other articles \cite{tilouine}, \cite{polo_tilouine}, \cite{kwlan_polo} in which some of these ideas have been 
developed and improvements made.  For example in \cite{kwlan_polo} a rigid analytic version of the Proposition 5.1 in \cite{faltings_chai} is given, namely Proposition 4.21, and therefore using it, even in the rigid analytic setting, the BGG-complex exists. When one works over a field of characteristic zero, one may base change to $\C$ and use \cite{faltings_chai} to deduce the quasi-isomorphism of the de Rham complex and the BGG complex. 

\bigskip
\medskip
\noindent
In this article, in the so called $p$-adic case with $p>0$  a prime integer, we are concerned with computing de Rham cohomology of de Rham complexes of modules with connections which are not coherent sheaves, but much larger sheaves of $p$-adic Banach modules.  Moreover, these sheaves only live on certain open sub-spaces of the Shimura varieties, seen as $p$-adic rigid or adic analytic spaces over a $p$-adic field. 
In such cases one cannot base change to $\C$ and use the complex analytic theory, therefore a new strategy is needed.

Recently, there has been much activity around understanding the \'etale side of the story, that is to say the completed cohomology sheaves
and their Hodge-Tate structures by Lue Pan in \cite{lue_pan}, J.~E.~Rodriguez Camargo in \cite{camargo}, and 
G.~Boxer, F.~Calegari, T.~Gee, V.~Pilloni in \cite{boxer_calegari_gee_pilloni}. They develope a BGG-decomposition theory for the completed cohomology sheaf and use it in order to prove Hodge-Tate decomposition for Galois representations attached to modular forms in very general situations.

The de Rham side of this story is very different and new ideas and concepts are needed in order to understand how to apply the BGG theory to it.
The completed de Rham cohomology doesn't seem useful, even in the modular curve case and, instead, we use other $p$-adic families of de Rham sheaves. 
In the introduction we will explain our strategy of constructing BGG-complexes and proving their quasi-isomorphism to the de Rham complex in the simplest situation, the so called classical case. This is the case of a PEL type Shimura variety $X$ associated to a pair $P\subset G$, where 
$G$ is a semi-simple, reductive algebraic group and $P$ is a parabolic subgroup of it, all defined over a finite extension $K$ of $\Q_p$.
In this situation (see section \ref{sec:classical} for more details) let $V$ denote a finite dimensional $K$-vector space, which is a an irreducible representation of $G$. If we denote by $k$ the weight of $V$, we assume that $k$ is regular and dominant for the system of roots chosen.

On one hand, we can associate to $V$ an $\cO_X$-module with integrable connection $(\cV, \nabla)$, called ``automorphic bundle on $X$", by using the $G$-torsor of trivializations of the first relative de Rham cohomology sheaf $\cH$ of the universal Abelian scheme over $X$ with its Poincare pairing and extra endomorphisms. Call this torsor $\pi\colon \cE_G \to X$, see section \ref{sec:PEL}. Therefore we obtain a de Rham complex of sheaves on $X$, which extends canonically to a logarithmic de Rham complex on a smooth toroidal compactification $\Xbar$ of $X$, denoted $\cV\otimes_{\cO_{\Xbar}} \omega^\bullet_{\Xbar/S}$, by extending the above mentioned $G$-torsor to a torsor $\cE_G\to \Xbar$. Here $\omega^i_{\Xbar/S}$ is the sheaf of degree $i$-logarithmic differentials of $\Xbar$ over $S$. This is the de Rham complex we would like to study. In the following we will consider 
$\Xbar$ as a rigid analytic or adic space over $S={\rm Spa}(K, \cO_K)$.

On the other hand, let us denote by $\fp\subset \fg$ the Lie-algebras of $P\subset G$, respectively, over $K$, and by $V^\vee$ the $K$-dual of $V$. Then $V^\vee$ is a $(\fg, P)$-module and its Verma module, $U(\fg)\otimes_{U(\fp)}V^\vee$ is also a $(\fg, P)$-module (see section \ref{sec:PEL}).  Therefore we have the Koszul complex $$K_\bullet:= U(\fg)\otimes_{U(\fp)}\bigl(V^\vee\otimes_K\wedge^\bullet \fg/\fp\bigr).$$

As explained above, using the action of the centre of $U(\fg)$, denoted $Z(\fg)$, and the character $\chi_{k^\vee}$ of $Z(\fg)$ associated to the weight $k^\vee$ of the irreducible $G$-representation $V^\vee$, we obtain a sub-complex of $K^\bullet$, of the form
 $$(1)\quad U(\fg)\otimes_{U(\fp)}M_\bullet\subset K_\bullet:=U(\fg)\otimes_{U(\fp)}(V\otimes \wedge^\bullet \fg/\fp).$$
 What everybody does (e.g., \cite{faltings_chai}) is to use the isomorphism $(\ast)$ above, i.e.,  Proposition 5.1 \cite{faltings_chai} in the complex case, or the Proposition 4.21 \cite{kwlan_polo} in the rigid analytic case, and obtain from $(1)$ a morphism of complexes  of  sheaves associated to the follwoing morphism of complexes
 $$
 V\otimes_K\otimes \wedge^\bullet (\fg/\fp)^\vee\lra \bigl(M_\bullet  \bigr)^\vee. 
 $$
 Then they need to show that this is a quasi-isomorphism of complexes of sheaves.

 In this article we choose a different strategy, namely we dualize the entire formula $(1)$, associate complexes of sheaves to the dualized complexes and in the end extract from them what we need, namely the quasi-isomorphism between the BGG-complex and the de Rham complex.

Let us point out that this is delicate and so let us explain here what happens.  
Let us first notice that we can only sheafify duals of Verma modules locally, namely let us consider an open affinoid cover $(U_i)_{i\in I}$ of $\Xbar$ such that over each $U_i$ there is a section of the natural projection $\pi\colon \cE_G\lra \Xbar$. Let $\nabla_{\rm GM}\colon \cH\lra \cH\otimes_{\cO_{\Xbar}}\omega^1_{\Xbar/S}$ denote the Gauss-Manin connection on $\cH$.
 
 Let us recall that if we denote by $\cI$ the ideal of the log diagonal embedding 
 $\Delta\colon \Xbar\lra (\Xbar\times_S \Xbar)^{\rm ex}$, we let $\cP_{\Xbar/S}:=\lim_m \cP^m_{\Xbar/S}$ where $\cP^m_{\Xbar/S}:=(\cO_{\Xbar}\otimes_K\cO_{\Xbar})^{\rm ex}/\cI^{m+1}$. We recall that the exponent ${\rm ex}$ present in the above formula refers to the exactification functor applied to a log closed immersion of log or log adic spaces.

 Therefore  $\cP^1:=\cP^1_{\Xbar/S}$, with its two injections $\tau_j:\cO_{\Xbar}\to  \cP^1$, $j=1,2$,  denotes the structure sheaf of the first infinitessimal neighbourhood of the log diagonal of $\Xbar$, then the connection $\nabla_{\rm GM}$ can be seen as an isomorphism $\nabla_{\rm GM}: \cH\otimes_{\cO_{U_i},\tau_2}\cP^1_{U_i}\cong \cH\otimes_{\cO_{U_i},\tau_1}\cP^1_{U_i}$, which reduces to the identity modulo the ideal $\cI$. Using the section $s_i\colon U_i\lra \cE_G$ and noticing that $\cH|_{U_i}\cong \cO_{U_i}^{2g}$, we remark that
 $\cH|_{U_i}\otimes_{\cO_{U_i}, \tau_j}\cP^1_{U_i}\cong (\cP^1_{U_i})^{2g}$, for $j=1,2.$ Here we denoted $\cP^1_{U_i}:=\cP^1|_{U_i}\cong \cP^1_{U_i/S}$ and as $\cE_G$ is a $G$-torsor, $\nabla_{\rm GM}$ corresponds to an element $\gamma_{\rm GM,i}\in G(\cP^1_{U_i})$ such that the image of $\gamma_{\rm GM,i}$ under the map $G(\cP^1_{U_i})\lra G(\cO_{U_i})$ is the identity ${\rm Id}\in G(\cO_{U_i})$. Therefore we obtain
 a canonical element 
 $$
 (\gamma_{\rm GM,i}-{\rm Id})\in {\rm Ker}\bigl(G(\cP_{U_i})\lra G(\cO_{U_i})\bigr)\cong \fg\otimes_K\cI/\cI^2\cong \fg\otimes_K\omega^1_{U_i/K}
 $$
 This element gives two things.\smallskip
 
 a) Consider the image of $(\gamma_{\rm GM,i}-{\rm Id})$ in $\fg/\fp\otimes_K \omega^1_{U_i/S}$ and let $\partial\in (\fg/\fp)^\vee$.
 Then $(\partial\otimes 1)(\gamma_{\rm GM,i}-{\rm Id}\in \omega^1_{U_i/S}$, i.e. we obtain a map: ${\rm KS}\colon (\fg/\fp)^\vee\otimes_K\cO_{U_i}\lra \omega^1_{U_i/S}$. This is the Kodaira-Spencer map and it is an isomorphism of $\cO_{U_i}$-modules.\smallskip
 
 b) Let $\delta\in \fg/\fp\otimes_K \cO_{U_i}$. Then if denote ${\rm KS}^\vee$ the dual to the Kodaira-Spencer map above, we have ${\rm KS}^\vee(\delta)\in (\omega^1_{U_i/S})^\vee$ and so ${\rm KS}^\vee(\delta)(\gamma_{\rm GM,i}-{\rm Id})\in \fg\otimes_K\cO_{U_i}$.
 Therefore we obtain a map $\fg/\fp\otimes_K\cO_{U_i}\lra \fg\otimes_K\cO_{U_i}$, which is a section of the natural projection.
 We denote $\fn_i$ the image of this section, it is a sub-Lie algebra of $\fg\otimes_K\cO_{U_i}$, which is abelian (because $\nabla_{\rm GM}$ is integrable.) \smallskip

 Let us notice that we can write, for $Y$ a $U(\fp)$-module, $U(\fg)\otimes_{U(\fp)}Y\otimes_K\cO_{U_i}\cong U(\fn_i)\otimes_{\cO_{U_i}}(Y\otimes_K\cO_{U_i}$. Therefore formula $(1)$ base-changed to $U_i$ becomes
  
 $$
 (2)\quad U(\fn_i)\otimes_{\cO_{U_i}}(M_\bullet\otimes_K\cO_{U_i})\subset U(\fn_i)\otimes_{\cO_{U_i}} (V^\vee\otimes \wedge^\bullet \fg/\fp)\otimes_K\cO_{U_i}.
 $$
  We recall that if we denoted by $\cI$ the ideal of the diagonal embedding 
 $\Delta\colon \cX\lra (\cX\times_S \cX)^{\rm ex}$, we let $\cP_{\cX}:=\lim_m \cP^m_{\cX}$ where $\cP^m_{\cX}:=(\cO_{\cX}\otimes\cO_{\cX}))^{\rm ex}/\cI^{m+1}$. Then $\cP_{\cX}$ has a left and a right $\cO_{\cX}$-module structure, and we have local identifications  
 $U(\fn_i)^\vee\cong \cP_{U_i/S}$.    In other words, by dualizing formula $(2)$, we obtain the following commutative diagram of complexes of $\cO_{U_i}$-modules.
 $$
 \begin{array}{ccccccccc}
 (3)\quad U(\fn_i)^\vee\hat{\otimes}_{\cO_{U_i}}((M_\bullet)^\vee\otimes_K\cO_{U_i})&\leftarrow &U(\fn_i)^\vee\hat{\otimes}_{\cO_{U_i}} (V\otimes \wedge^\bullet (\fg/\fp)^\vee)\otimes_K\cO_{U_i}\\
 \downarrow \cong&&\downarrow\cong\\
 \cP_{U_i/S}\hat{\otimes}_{\cO_{U_i}}\cW^\bullet&\leftarrow &\cP_{U_i}\hat{\otimes}_{\cO_{U_i}}(\cV\otimes_{\cO_{U_i}}\omega^\bullet_{U_i/S})
 \end{array}
 $$

Now we ask ourselves: do the complexes above glue? Of course we have
$$
\cP_{U_i}\hat{\otimes}_{\cO_{U_i}}(\cV\otimes_{\cO_{U_i}}\omega^\bullet_{U_i/s})\cong \Bigl(\cP_{\Xbar/S}\hat{\otimes}_{\cO_{\Xbar}}(\cV\otimes_{\cO_{\Xbar}}\omega^\bullet_{\Xbar/S})\Bigr)|_{U_i}
$$ 
for every $i\in I$ and so, obviously, these complexes glue. Moreover, we know that for $i$, $j\in I$ the gluing data for these complexes are isomorphisms given by multiplication by an element $p_{i,j}\in P(U_{i,j})$, where $U_{i,j}:=U_i\cap U_j$.
It follows that, for these $i,j\in I$, the isomorphism induced by multiplication by $p_{i,j}$ between $\Bigl(\cP_{U_i}\hat{\otimes}_{\cO_{U_i}}(\cV\otimes_{\cO_{U_i}}\omega^\bullet_{U_i/S})\Bigr)|_{U_{i,j}}\cong \Bigl(\cP_{U_j}\hat{\otimes}_{\cO_{U_j}}(\cV\otimes_{\cO_{U_j}}\omega^\bullet_{U_j/S})\Bigr)|_{U_{i,j}}$, induces the isomorphism multiplication by $l_{i,j}$, which is the image of $p_{i,j}$ in $L(U_{i,j})$
($L$ the Levi quotient of $P$), between  
$$
\Bigl(\cP_{U_i/S}\hat{\otimes}_{\cO_{U_i}}\cW^\bullet\Bigr)|_{U_{i,j}}\cong \Bigl(\cP_{U_j/S}\hat{\otimes}_{\cO_{U_j}}\cW^\bullet\Bigr)|_{U_{i,j}}. 
$$
Therefore the family of complexes $\Bigl(\cP_{U_i/S}\hat{\otimes}_{\cO_{U_i}}\cW^\bullet\Bigr)_{i\in I}$ glue to give the complex
$\cP_{\Xbar/S}\hat{\otimes}_{\cO_{\Xbar}}\cW^\bullet$ with its morphism of complexes $\cP_{\Xbar/S}\hat{\otimes}_{\cO_{\Xbar}}(\cV\otimes_{\cO_{\Xbar}}\omega^\bullet_{\Xbar/S})\lra \cP_{\Xbar/S}\hat{\otimes}_{\cO_{\Xbar}}\cW^\bullet$.

 If $A$ is a coherent $\cO_{\Xbar}$-module, we denote by $L(A):= \cP_{\Xbar}\hat{\otimes}_{\cO_{\Xbar}}A=\lim_m (\cP^m_{\Xbar}\otimes_{\cO_{\Xbar}}A)$, seen as an $\cO_{\Xbar}$-module with the left structure and call it the linearization of $A$.

 The functor $L(\ )$ is called the linearization functor on $\Xbar$, and has the property that if $(A, \nabla)$ is a pair consisting of an
 $\cO_\Xbar$-module and an integrable connection on it, the linearized de Rham complex of the pair is exact in degrees larger than $0$.
 
 Therefore we can write the morphism of complexes above $$(2)\quad L\bigl(\cV\otimes_{\cO_{\Xbar}}\omega^\bullet_{\Xbar/S}\bigr)\lra L\bigl(\cW^\bullet\bigr).$$
 
 We point out that the projection $(2)$ of complexes has a section, i.e. $L(\cW)^\bullet)$ is a direct summand of $L(\cV\otimes_{\cO_{\Xbar}} \omega^\bullet_{\Xbar/S})$, but the two complexes are NOT quasi-isomorphic (see remark \ref{remark:elliptic}).
  
 We will pause for a moment to point out that in the literature (for example in \cite{kwlan_polo}) often the Lie algebra $\fn^-$ is used as natural 
  complement of $\fg/\fp$, i.e. $\fn^-\subset \fg$ and the composition $\fn^-\subset \fg\lra \fg/\fp$ is an isomorphism. Therefore we might want to identify locally $\cP_{U_i/S}\cong U(\fn^-)^\vee\otimes_K\cO_{U_i}$, for every $i\in I$ as above. But then the map
 $\cP_{U_i/S}\hat{\otimes}_{\cO_{U_i}}(\cV\otimes_{\cO_{U_i}}\omega^\bullet_{U_i/S})\lra \cP_{U_i/S}\otimes_{\cO_{U_i}}\cW^\bullet$
 is very complicated and it is not clear how to show that the complexes $\Bigl(\cP_{U_i/S}\otimes_{\cO_{U_i}}\cW^\bullet\Bigr)_{i\in I}$ glue. Therefore the use of the $\cO_{U_i}$-Lie algebras $\fn_i$ above is necessary if one wants to prove the gluing.

  \bigskip
  
The next step is to eliminate the linearization, for which we consider the log infinitesimal site $(\Xbar/S)_{\rm inf}^{\rm log}$ and the natural morphism $\pi$ of topoi from the log infinitesimal topos
$\Xbar_{\rm inf}$ to the analytic topos $\Xbar_{\rm an}$ of $\Xbar$, with  $\pi_\ast:\Xbar_{\rm inf}\lra \Xbar_{\rm an}$. It has the property that 
if $(A,\nabla)$ is an $\cO_{\Xbar}$-module with integrable connection, then $(L(A), L(\nabla))$ can be seen as a sheaf on $(\Xbar/S)_{\rm inf}^{\rm log}$, $\pi_\ast(L(A))\cong A$ and $R^i\pi_\ast(L(A))=0$ for $i>0$. See section \ref{sec:linedelin}. We show that applying $\pi_\ast$ to the diagram $(2)$ above, produces the desired quasi-isomorphism $\cV\otimes_{\cO_{\Xbar}}\omega^\bullet_{\Xbar/S}\lra \cW^\bullet$.

\bigskip

The plan of the article is the following:

In section 2, we introduce the classical setting of PEL type Shimura varieties and we define the relevant automorphic vector bundles and their connections.

In section 3 we provide an axiomatic setting for the BGG formalism that will apply both to the classical and the $p$-adic setting.

In section 4 we recover the BGG theory in the PEL type case, for classical coefficients, following the outline above.

In section 5 we discuss the $p$-adic case.

In section 6, we consider the case of modular curves and show that the hypothesis discussed in section 3 are satisfied. In this case, everything can be computed explicitly.

In section 7 we develop the theory of linearization and de-linearization of quasi-coherent crystals on the infinitesimal site of a smooth and separated log adic space. These results are heavily used in section 2 and 3 and 4.
We point out that recently there has been a lot of interest in the infinitesimal site and topos of a rigid space, see \cite{guo}, \cite{shiho}, \cite{chiarellotto_fornasiero}. We think that our treatment of the log infinitesimal neighbourhoods of the diagonal and linearization and de-linearization of quasi-coherent crystals, introduces some new ideas in the theory, which allow us to prove the results we need.

\bigskip

\bigskip

{\bf Acknowledgments:} We thank Martin Ortiz for pointing out a gap in one of our arguments. 

\bigskip

\section{The classical case for general PEL Shimura varieties.}
\label{sec:classical}

Our first example is the classical situation of \cite{faltings_chai}, presented more generally in \cite{kwlan_polo}, i.e., the situation where 
we start with a finite dimensional, irreducible representation of our reductive, algebraic group $G$.

We will briefly recall the geometric set-up, which is presented with all details in chapter 2 of \cite{kwlan_polo}.
Let $\cO$ be an order in a finite dimensional semi-simple $\Q$-algebra with a positive involution $\ast$.
We denote by $\Z(1)=2\pi i\Z={\rm Ker}\bigl(\exp:\C\to \C^\ast\bigr)$, where we denoted $i$ a complex number with $i^2=-1$.
If $M$ is a $\Z$-module, we denote $M(1):=M\otimes_{\Z}\Z(1)$.

\begin{definition}
A PEL type $\cO$-lattice is a triple $(L, \langle\ ,\ \rangle, h_0)$, where 

i) $L$ is an $\cO$-lattice

ii) An alternating pairing $\langle\ ,\ \rangle\colon L\times L\lra \Z(1)$ satisfying $\langle bx,y\rangle=\langle x, b^\ast y\rangle$, for all
$b\in \cO, x,y\in L$.

iii) $h_0$ is an $\R$-algebra homomorphism $h_0:\C\to {\rm End}_{\cO\otimes_{\Z}\R}(L\otimes_{\Z}\R)$ satisfying:

a) $\langle h_0(z)x, y\rangle=\langle x, h_0(z^c)y\rangle$, for all $z\in \C, x,y\in L$. where $z^c$ is the complex conjugate of $z$.

b) the $\R$-bilinear pairing $(2\pi i)^{-1}\langle \bullet, h_0(i)\bullet\rangle$ is symmetric and positive definite. 
\end{definition}

We recall that a tuple $(\cO, \ast, L, \langle\ ,\ \rangle, h_0)$ is called an integral PEL datum. It allows us to define the group functor $G$ on $\Z$-algebras: 
$$
G(R):=\{(g,r)\in {\rm Aut}_{\cO\otimes R}(L\otimes R)\times \mathbb{G}_m(R) \ | \ \langle(gx,gy\rangle=r\langle x, y \rangle, \mbox{ for all } x,y\in L\otimes R\},
$$
for $R$ a $\Z$-algebra.

We also recall that $h_0$ defines a Hodge structure of weight  $-1$ on $L$, i.e. $L\otimes_{\Z}\C\cong V_0\oplus V_0^c$ such that
$h_0(z)$ acts on $V_0$ by $1\otimes z$ and on $V_0^c$ by $1\otimes z^c$. As $V_0$ is maximal totally isotropic with respect to $\langle \ ,\ \rangle$, we have a canonical isomorphism $V_0^c\cong V_0^\vee(1)$.

We fix a prime integer $p>0$ called {\it good} if it satisfies:

1) $p$ is unramified in $\cO$

2) $p\neq 2$ if $\cO\otimes \Q$ involves simple factors of type D.

3) If denote $L^\#:=\{x\in L\otimes \Q\ | \ \langle x, y\rangle\in \Z(1), \mbox{ for all } y\in L\}$, i.e. the dual lattice to $L$ with respect to $\langle\ ,\ \rangle$, then $p\nmid [L^\#:L]$.

We fix a good choice of $p$. There is a finite field extension $F_0'$ of the reflex field $F_0$ of the $\cO\otimes \C$-module 
$V_0$, unramified at $p$ and an $\cO_{F_0'}$-torsion free $\cO\otimes_{\Z}\cO_{F_0'}$-module $L_0$, such that $L_0\otimes_{\cO_{F_0'}}\C\cong V_0$. We fix a choice of $F_0'$ and $L_0$ and denote by 
$$
\langle \ ,\ \rangle_{\rm can}:\bigl(L_0\oplus L_0^\vee(1)\bigr)\times \bigl(L_0\oplus L_0^\vee(1)\bigr)\lra \cO_{F_0'}(1)
$$
the pairing defined by: $\langle (x_1,f_1), (x_2,f_2)\rangle=f_2(x_1)-f_1(x_2)$, for
$(x_i,f_i)\in L_0\oplus L_0^\vee(1)$, $i=1,2.$

There exists a DVR $R_1$ over $\cO_{F_0'}$ satisfying the following:

a) There exists an isomorphism 
$$
\bigl(L\otimes_{\Z}R_1, \langle\ ,\ \rangle\bigr)\cong \bigl(L_0\oplus L_0^\vee(1), \langle\ ,\ \rangle_{\rm can}\bigr)\otimes_{\cO_{F_0'}}R_1,
$$
which allows us to write for all $R_1$-algebra $R$
$$
(G\otimes R_1)(R):=\{(g,r)\in {\rm Aut}_{\cO\otimes R}((L_0\oplus L_0^\vee(1))\otimes_{\cO_{F_0'}} R)\times \mathbb{G}_m(R) \ | \ \langle(gx,gy\rangle_{\rm can}=r\langle x, y \rangle_{\rm can}, 
$$
for all  $x,y\in (L_0\oplus L_0^\vee(1))\otimes R$. If we denote
$$
P(R):=\{(g,r)\in (G\otimes R_1)(R)\ | \ g\bigl(L_0^\vee(1)\otimes_{\cO_{F_0'}}R\bigr)=L_0^\vee(1)\otimes_{\cO_{F_0'}}R\}
$$
and 
$$
M(R):={\rm Aut}_{\cO\otimes R}\bigl(L_0^\vee(1)\otimes_{\cO_{F_0'}}R\bigr)\times \mathbb{G}_m(R),
$$
we may think of $M(R)$ as a quotient of $P(R)$, functorially for all $R_1$-algebra $R$.

The group functors $G\otimes R_1, P, M$ are representable by respectively a split, reductive algebraic group scheme over ${\rm Spec}(R_1)$, a parabolic subgroup scheme of it and respectively $M$ is represented by the Levi quotient of $P$. 

\begin{remark} We make the following remark on the notations. In the other sections of the article the Levi-quotient of the parabolic subgroup $P$ of $G$ is denoted $L$. In this section however, the notation $L$ is used for a lattice with action of $\cO$, part of the Shimura data, and therefore, in order to be consistent with the notations in \cite{kwlan_polo} we use the letter $M$ to denote the Levi quotient of $P$, in this section and the next. 

\end{remark}
\bigskip

\subsection{PEL type Shimura varieties and automorphic bundles.}\label{sec:PEL}

We fix $H$ a {\bf neat}  open compact subgroup of $G(\widehat{\Z}^p)$. The data $\bigl(L, \langle\ ,\ \rangle, h_0, \cH \bigr)$ defines a moduli problem $M_{\cH}$ over $S_0:={\rm Spec}\bigl(\cO_{F_0', (p)}\bigr)$, parameterizing tuples $(A,\lambda, \iota, \alpha_{H})$ over schemes $S$ over $S_0$, see \cite{kwlan_polo}, \S 2.2, where

1) $A\to S$ is an abelian scheme

2) $\lambda\colon A\to A^\vee$ is a polarization of degree prime to $p$.

3) $\iota:\cO\hookrightarrow {\rm End}_S(A)$ is an $\cO$-endomorphism structure on $A$.

4) ${\rm Lie}(A/S)$ with its $\cO\otimes_{\Z}\Z_{(p)}$-action satisfies the determinant condition given by $(L\otimes_{\Z}\cO_S, \langle\ ,\ \rangle, h_0)$.

5) $\alpha_{H}$ is an integral level $H$-structure on $(A, \lambda, \iota)$ of type $(L\otimes\widehat{\Z}^p, \langle\ ,\ \rangle)$.

Then $M_{H}$ is represented by a smooth, quasi-projective scheme over $S_0$, we denote by $\cM_{\cH}$ the adic analytic space over $S_0':={\rm Spa}(K, \cO_K)$, where $K$ is the $p$-adic completion of $F_0'$, associated to the formal completion of $M_{\cH}$ along its special fiber. 
We use the same notation $(A,\lambda, \iota, \alpha_{H})$ to denote the analytification of the universal algebraic tuple, over $\cM_{H}$. 
We consider the relative de Rham cohomology sheaf ${\rm H}^1_{\rm dR}(A/\cM_{H})$ and let ${\rm H}_1^{\rm dR}(A/\cM_{H})$ denote its dual.
We have a natural exact sequence defining the Hodge filtration of ${\rm H}_1^{\rm dR}(A/\cM_{H})$
$$
0\lra \bigl({\rm Lie}(A^\vee/\cM_{H})\bigr)^\vee(1) \lra {\rm H}_1^{\rm dR}(A/\cM_{\cH})\lra {\rm Lie}(A/\cM_{H})\lra 0,
$$
and a natural alternating pairing (the Poincar\'e pairing) $$\langle\ ,\ \rangle_{\lambda}:{\rm H}_1^{\rm dR}(A/\cM_{H})\times {\rm H}_1^{\rm dR}(A/\cM_{H})\lra \cO_{\cM_{H}}(1).$$

\medskip

We consider the algebraic groups $G$, $P$, $M$ over $R_1$-algebras and we base-change $\cM_{\cH}$ over $S:={\rm Spa}(R_1\otimes_{\Z_{(p)}} \Q_p, R_1\otimes_{\Z_{(p)}}\Z_p)$. We now fix $K$ a finite extension of $R_1\otimes_{\Z_{(p)}}\Q_p$) and we  define functors from the categories of algebraic representations of $G\otimes_{R_1} K$ (respectively $P\otimes_{R_1} K$, respectively $M\otimes_{R_1} K$) on finite $K$-vector spaces, to  the category of coherent sheaves on $\cM_{H, S}$, called sometimes automorphic vector bundles (see \cite{kwlan_polo}, \S 2.2.) We first define torsors $\cE_G, \cE_P, \cE_M$ over $\cM_{H,S}$ as follows:
$$
\cE_G:=\underline{\rm Isom}_{\cO\otimes\cO_{\cM_{\cH,S}}}\Bigl({\bigl(\rm H}_1^{\rm dR}(A/\cM_{H,S}, \langle\ ,\ \rangle_{\lambda}\bigr), \bigl((L_0\oplus L_0^\vee(1))\otimes_{R_1}\cO_{\cM_{H,S}}, \langle\ ,\ \rangle_{\rm can}\bigr)\Bigr).
$$

$$
\cE_P:=\underline{\rm Isom}_{\cO\otimes\cO_{\cM_{H,S}}}\Bigl({\bigl(\rm H}_1^{\rm dR}(A/M_{\cH,S}, \langle\ ,\ \rangle_{\lambda},
({\rm Lie}(A^\vee/\cM_{H,S})^\vee(1)\bigr), 
$$
$$
\bigl((L_0\oplus L_0^\vee(1))\otimes_{R_1}\cO_{\cM_{H,S}}, \langle\ ,\ \rangle_{\rm can}, L_0^\vee(1)\otimes_{R_1}\cO_{\cM_{H,S}}\bigr)\Bigr).
$$
and
$$
\cE_M:=\underline{\rm Isom}_{\cO\otimes\cO_{\cM_{H,S}}}\Bigl(\bigl({\rm Lie}(A^\vee/\cM_{H,S})^\vee(1)\bigr), \bigl(L_0^\vee(1)\otimes_{R_1}\cO_{\cM_{H,S}}\bigr)\Bigr).
$$

\begin{definition}
Let $Q\in \{G,P,M\}$ and $W$ an algebraic representation of $Q\otimes K$ on a finite dimensional $K$-vector space.
Then let $T:={\rm Spa}(R, R^+)$ be an affinoid adic space over $S$
$$
\cE_{Q,T}(W):=(\cE_Q\times_ST)\times^{Q_T}(W\otimes_{K}R),
$$
the adic push-out of the torsor by $W$. It is a coherent sheaf on $\cM_{H,S}$.
\end{definition}

These functors have the following properties (see Lemma 2.24 in \cite{kwlan_polo}).

1) The assignments $\cE_G(\ )$ (respectively $\cE_P(\ )$, respectively $\cE_M(\ )$) define exact functors from the category 
of algebraic representations of $G\otimes K$ on finite dimensional $K$-vector spaces (respectively $P\otimes K$, respectively $M\otimes K$), to the category of coherent sheaves on $\cM_{H,S}$.

2) If $W$ is a $G\otimes K$-representation and we denote $W|_P$ its restriction to representations of $P\otimes K$, then we have canonical isomorphisms $\cE_G(W)\cong \cE_P(W|_P)$.

3) If now $W$ is a representation of $M\otimes K$ and we consider it as a representation of $P\otimes K$ via the projection
$P\to M$, we have canonical isomorphisms $\cE_P(W)\cong \cE_M(W)$. 

4) If $W$ is a $P\otimes K$ representation and it has a finite separable and exhausting filtration $(F^a(W))_{a\in \Z}$ by sub-$P\otimes K$-representations such that the graded quotients are $M\otimes K$-representations, then $\cE_P(W)$ has a finite separable and exhausting filtration by coherent sub-sheaves $(\cE_P(F^a(W)))_{a\in \Z}$, with graded quotients $\cE_M(Gr^a(W))$.

The next step in \cite{kwlan_polo} is to compactify $\cM_{H}$, i.e. let $\cX$ be the adic space associated to a toroidal compactification $M_{H,S}^{\rm tor,\Sigma}$, of $M_{H,S}$, depending on a compatible collection $\Sigma$ of cone decompositions. Then $\cX\to S$ is a smooth, projective morphism of adic spaces. Moreover if denote
$D:=\cX\backslash \cM_{H,S}$, this is a relative Cartier divisor  with simple normal crossings. We denote from now on $\cX:=(\cX, M)$ the log adic space associated to the pair $(\cX, D)$. Similarly, the universal abelian scheme $A\to \cM_{H,S}$ extends to a semi-abelian scheme, denoted still $A\to \cX$ and this can be compactified to obtain a proper morphism $f:\overline{A}\to \cX$ such that $f^{-1}(D)\subset \overline{A}$ is a relative Cartied divisor with normal crossings, which defines a fine log structure on $\overline{A}$. We denote the log adic space thus obtained by $\overline{A}$. Then the morphism of log adic spaces $f:\overline{A}\to X$ is log smooth and proper and we denote ${\rm H}^1_{\rm dR}(\overline{A}/\cX)$ the first log de Rham cohomology space of $\overline{A}/\cX$. It is locally free $\cO_{\cX}$-module of finite rank and it canonically extends ${\rm H}^i(A/\cM_{\cH,S})$, with its Hodge filtration and Gauss-Manin connection, to $X$.
Moreover if denote by $\omega^1_{\oA/\cX}$, as usual, the sheaf of logarithmic differentials of $\oA/\cX$, then  the Hodge filtration of 
${\rm H}^1_{\rm dR}(\oA/\cX)$ is $\omega^1_{\oA/\cX}$.

Then one can re-define the torsors $\cE_G, \cE_P, \cE_M$ over $\cX$, instead of over $\cM_{H,S}$, denoted $\oE_G, \oE_P, \oE_M$, using 
${\rm H}_1^{\rm dR}(\oA/\cX):=\bigl({\rm H}^1_{\rm dR}(\oA/\cX)\bigr)^\vee$ and $\omega_{\oA^\vee/\cX}^1$ instead of 
${\rm H}_1^{\rm dR}(A/\cM_{\cH,S})$ and $\bigl({\rm Lie}(A^\vee/\cM_{\cH,S})\bigr)^\vee$. 


\bigskip

We arrived at the point where we consider Lie algebras and Verma modules. Let $\fg:={\rm Lie}(G\otimes K), \frak{p}:={\rm Lie}(P), \fm:={\rm Lie}(M)$ the Lie algebras of our groups, and let $\fn^+$, $\fn^-$ be the unipotent Lie algebras. We recall that by construction, and in general by the assumptions in the definition of Shimura varieties,  $\fn^+$ and $\fn^-$ are both abelian Lie algebras and the action of $P$ by conjugation on $\fg/\fp$ factors via $M$.

 \begin{definition}[see also \cite{kwlan_polo}, Definition 4.1] 
 \label{def:fgB} 
 By a $(U(\fg), P)$-module over a $K$-algebra $R$ we mean an $R$-module with actions of $U(\fg)$ and of $P$, which induce the same action of  $\fp$. By a morphism of $(U(\fg), P)$-modules we mean a morphism of $U(\fg)$-modules which induces by restriction a morphism of $U(\fp)$-modules coming from a morphism of $P$-modules. 
 \end{definition}
 
 In \cite{kwlan_polo}, Lemma 4.2 it is shown that if $W$ is a $P$-module over some $K$-algebra $R$, then 
 the Verma modules $V:=U(\fg)\otimes_{U(\fp)}W$, with action of $U(\fg)$ on the first factor of the tensor product and diagonal action of $P$ is a $(U(\fg), P)$-module.
 
 \bigskip

 \bigskip

If $W$ is an algebraic representation of $G\otimes K$, as above, then $W^\vee$ has a natural structure of a $(\fg, P)$-module and moreover
by the above, $U(\fg)\otimes_{U(\fp)}W^\vee$, with its natural $(\fg, P)$-module structure is a Verma module.

In the Introduction, i.e. in section \ref{sec:intro}, we explained how, using local sections of the $G$-torsor $\oE_G\lra \Xbar$, we obtain sheaves associated to duals of Verma modules. We present more details here. 
Let us denote by $\cH:={\rm H}^1_{\rm dR}(\oA/\cX)\bigr)$ and by $\nabla_{\rm GM}:\cH\lra \cH\otimes_{\cO_\cX}\omega^1_{\oA/cX}$ the Gauss-Manin connection. Let $\WW:=\oE_P(W)$. It is a locally free $\cO_\cX$-module of rank equal the dimension of $W$, with a natural filtration. We recalled above that $\WW$ has a natural logarithmic, integrable connection, $\nabla$, which satisfies the Griffith-transversality property with respect to the filtration. 

\noindent
We have the following way of relating $(U(\fg), P)$-modules to sheaves on $\cX$.
Let us consider an open affinoid $U={\rm Spa}(R,R^+)$ of $\cX$, over which there is a section $s:U\lra \oE_G$ of the natural projection
$\pi:\oE_G\lra \Xbar$. For any $\cO_\Xbar$-module $\cF$ on $X$ we denote $\cF_U:=\cF|_U$ the restriction of $\cF$ to $U$. We recall that we denoted $S:={\rm Spa}(K, \cO_K)$, seen as a log adic space with trivial log structure. 

i) By definition we have an isomorphism of filtered $\cO_U$-modules $\varphi_U: \WW_U\cong W\otimes_K\cO_U$.

ii) By \cite{kwlan_polo} Corollary 4.13, the Kodaira-Spencer isomorphism induces an isomorphism $\overline{\psi}_U: (\omega^1_{\oA/\cX})^\vee\cong \fg/\fp\otimes_K \cO_U$. 

We recall that the $\cO_U$-linear $G$-action on $\WW_U$ induced by $\varphi_U$, defines, by deriving, a $\fg$-action on $\WW_U$ such that
we can see $\fg\otimes_K\cO_U$ as a Lie-algebra of $\cO_U$-differential operators on $\WW_U$. 

\noindent
In particular, for $W$ the standard representation of $G$, we have a description of the $\cO_U$-differential operators on $\cH_U$ which respect the Poincar\'e pairing and the $\cO$-action, as $\fg\otimes_K\cO_U$.

iii) The important observation is that the Gauss-Manin connection $\nabla_{\rm GM}$ on $\cH_U$ allows us to define a unique splitting 
$$\eta_U\colon \fg/\fp\otimes_K \cO_U\lra \fg\otimes_K\cO_U$$of the quotient map $\fg\to \fg/\fp$ as follows. Let $\partial$ be a global section of $\fg/\fp\otimes_K \cO_U$. Using the identification $(\overline{\psi}_U)^{-1}\colon \fg/\fp\otimes_K \cO_U\stackrel{\sim}{\longrightarrow} (\omega^1_{\oA_U/U})^\vee$, we can view $(\overline{\psi}_U)^{-1}(\partial)$  as a $K$-derivation of $\cO_U$. Then $\nabla_{\rm GM}\bigl((\overline{\psi}_U)^{-1}(\partial)\bigr)\colon \cH_U\lra\cH_U$ is a 
$\cO_U$-differential operator on $\cH_U$, which respects the Poincar\'e pairing and the $\cO$-action. Therefore by the remark at ii) above, it defines a section $\eta_U\bigl(\psi_U^{-1}(\partial)\bigr)$ of $ \fg\otimes_K\cO_U$.

We denote $\fn_U:=\eta_U(\fg/\fp\otimes_K\cO_U)$, it is a Lie-sub-algebra of $\fg\otimes_K\cO_U$. As $\nabla_{\rm GM}$ is an integrable connection, $\fn_U$ is an abelian Lie-sub-algebra of $\fg\otimes_K\cO_U$, with an $\cO_U$-linear isomorphism $\fn_U\cong \fg/\fp\otimes_K\cO_U$ defined by composing with the projection $\fg\to \fg/\fp$.

Let us now consider the universal enveloping algebra $U(\fn_U)$, of the Lie-algbera $\fn_U$, which is a coherent $\cO_U$-algebra, in fact, as $\fn_U$ is an abelian Lie-algebra, 
$U(\fn_U)\cong {\rm Sym}_{\cO_U}(\fn_U)$.  As such it has a natural increasing filtration by degrees $\{U(\fn_U)_n\}_n$, which induces the orthogonal, decreasing filtration $\{\bigl(U(\fn_U)^\vee\bigr)^n\}_n$, on $U(\fn_U)^\vee$. 

We recall that we denote $P_{U/S}^n$ the $n+1$-infinitesimal neighbourhood of the log diagonal $\Delta_U$ and by $\bigl(\cP_{U/S}^n\bigr)_{U/S}$ its structure sheaf,

\begin{lemma}
\label{lemma:neidiag}
For every $n\ge 0$, we have a natural isomorphism of coherent $\cO_U$-algebras $\psi_U^n:\cP^n_{U/S}\cong \bigl(U(\fn_U)^\vee\bigr)^n$, compatible with the transition maps, so an isomorphism of $\psi_U\colon \cP_{U/K}:=\lim_n \cP_{U/K}^n\cong U(\fn_U)^\vee$. These isomorphisms are compatible with the Gauss-Manin connection.
\end{lemma}

\begin{proof}

We recall that we have $U\to S$ a log smooth morphism of log affinoid spaces, with underlying morphism $\uU:={\rm Spa}(A,A^+)\to \uS:={\rm Spa}(K, \cO_K)$. Therefore $(\fn_U)^\vee\cong \omega^1_{U/S}=\omega^1_{U/S}$ by ii) above,
and therefore it follows that $(\wedge^j\fn_U)^\vee\cong \omega^j_{U/S}$ for all $j\ge 0$.

We recall from Lemma \ref{lemma:logsmooth} (where all the notations are explained) that we have an exact closed immersion
$$
(1)\quad  U\stackrel{\Delta_U}{\lra} (U\times_S U)^{\rm ex},
$$
 with coherent ideal $J$, such that $J/J^2\cong \omega^1_{U_i/S}$. 
 We denoted $\cP^s_{U/S}:=\cO_{(U\times_SU)^{\rm ex}}/J_i^{s+1}.$ Using Corollary \ref{rmk:coordinates}, it follows that for every $s\ge 1$ we have natural isomorphisms as $\cO_{U}$-algebras
 $$
(2)\quad  \psi^s_U:\cP^s_{U_i/S}\cong \oplus_{i=0}^s {\rm Sym}_{\cO_U}^i\bigl((\fn_U)^\vee\bigr)\cong (U(\fn_U)^s)^\vee.
 $$
By the compatibility of this isomorphism with the Gauss-Manin connection we mean the following. On the one hand we can recover the Gauss-Manin connection, using the first infinitesimal neighbourhood of the log diagonal $\Delta_U$, as follows: let $\pi_1,\pi_2:P^1_{U/S}\lra U$ be the two projections. Then giving the connection $\nabla_{\rm GM}$ on $\cH_U$ is equivalent to giving an isomorphism of $\cP^1_{U/S}$-modules
$\epsilon_{\rm GM}: \pi_1^\ast(\cH_U)\lra \pi_2^\ast(\cH_U)$ such that $\Delta_U^\ast(\epsilon_{\rm GM})={\rm Id}_{\cH_U}$ and $\epsilon_{\rm GM}$ satisfies the cocycle condition.

Similarly, the splitting $\eta_U:\fg/\fp\otimes_K\cO_U\cong \fn_U\subset \fg\otimes_K\cO_U$ gives us back the Gauss-Manin connection on $\cH_U$ as follows. Let $\partial\in (\omega^1_{\oA_U/U})^\vee$, then $\overline{\psi}_U(\eta_U(\partial)) $ is a differential operator of $\WW_U$, and as we have seen above $\overline{\psi}_U(\eta_U(\partial))=\nabla_{\rm GM}(\partial)$.

These constructions are compatible via $\psi^1_U$, concluding the proof of the lemma. 

\end{proof}

\begin{remark}
Corollary 4.13 and Lemma 4.18 in \cite{kwlan_polo} imply that $\fn_U\cong \fn^-\otimes_K\cO_U$ for all $U$, and although there is such an isomorphism 
as Lie-algebras, we doubt that it is always compatible with the Gauss-Manin connection. Even in the $\GL_2$-case, if $U$ is not contained in the ordinary locus of the relevant modular curve $\cX$, then the image of the splitting defined by $\nabla_{\rm GM}$ is not $\fn^-\otimes_K\cO_U$..
\end{remark}

Therefore, if $W$ is a finite, algebraic representation of $G\otimes K$ we obtain an integrable connection on $\WW_U:=W\otimes_K\cO_U$, called $\nabla_U$ at the beginning of this paragraph, following the recipe in the proof above. Namely, for ever $\partial$ section of $(\omega^1_{\oA_U/U})^\vee$, let
$\delta:= \overline{\psi}_U(\eta_U(\partial))\in \fn_U\subset \fg\otimes_K\cO_U$. Using that $W$ is canonically a $\fg$-module, we define
$\nabla_U(\partial)(w):=\delta(w)$, for all $w\in W$.

In particular, for $W$ the standard representation of $G\otimes K$, we have $\cH_U\cong W\otimes_K\cO_U$ and the connection $\nabla_U$ defined above is $\nabla_{\rm GM}$.

\section{BGG-decomposition of de Rham sheaves.}

\subsection{The set-up and the strategy.}
\label{sec:outline}

Using the previous section as an example, in this section we set-up a general, axiomatic framework in which de Rham sheaves can be constructed and the BGG-decomposition performed.

 Let $p>0$ be a prime integer, let $K$ be a finite extension of $\Q_p$ and  $(B,B^+)$ a complete Huber pair over ${\rm Spa}(K, \cO_K)$. Basically, we have two disjoint cases in which we can work: the classical case and the non-classical case, which will be henceforth called, the $p$-adic case.
 
 Let $S$ denote the log adic space, with trivial log structure defined by $S:=\Spa(K, \cO_K)$ in the classical case and $S:=\Spa(B, B^+)$ in the $p$-adic case.
 We let $f\colon \cX\to S$ be a log smooth and separated morphism of log adic spaces over ${\rm Spa}(K, \cO_K)$. We denote $\omega^1_{\cX/S}$ the sheaf of log $1$-differential forms of $\cX$ over $S$.  We assume that $B$ is an integral domain and let $\fK$ denote its fraction field.\smallskip
 
 We consider a semi-simple, reductive, algebraic group ${\rm G}$ over $\Q_p$ and we fix a sequence ${\rm T}\subset {\rm Bo}\subset {\rm P}\subset {\rm G}$ consisting of a parabolic subgroup ${\rm P}$ with abelian unipotent radical, the respective Borel subgroup  ${\rm Bo}$ and torus ${\rm T}$. The $p$-adic case, as it will be clear farther, requires that 
 we are also given an analytic subgroup $\bG\subset {\rm G}^{\rm an}$, we denote by $\bT:=\cG\cap {\rm T}^{\rm an}\subset \bB:=\bG\cap {\rm Bo}^{\rm an}\subset \PP:=\bG\cap {\rm P}^{\rm an}\subset \cG$. We laso let $\cL:=\bG\cap L ^{\rm an}$ be the corresponding Levi subgroup. We assume that the inclusion $\bG\subset {\rm G}^{\rm an}$ defines an isomorphism of Lie-algebras $\fg:={\rm Lie}(\bG)\cong {\rm Lie}({\rm G}^{\rm an})={\rm Lie}({\rm G})$. We denote by $\fh\subset \fb\subset \fp$ the Lie-algebras of $\bT$ or ${\rm T}$, respectively
  of $\cB$ or ${\rm Bo}$, and respectively $\PP$ or ${\rm P}$. Then $\fg$ is a semi-simple Lie-algebra over $K$, $\fp$ is a parabolic, $\fb$ is a Borel and $\fh$ is a Cartan sub-algebra of $\fg$. We let $\fn^+$, respectively $\fn^-$ be the Lie algebra of the unipotent radical of ${\rm P}$,  respectively of the parabolic opposite to ${\rm P}$ and we let $\fm\subset \fp$ be the Lie algebra of the Levi subgroup of ${\rm P}$.  In the $p$-adic case we will base change all these Lie-algebras over $\fK$ and keep denoting them with the same symbols. We also denote by $\Phi_\fm\subset \Phi$ the root systems of $\fm$ and $\fg$ respectively and by $\Phi_\fm^+\subset \Phi^+$ the positive roots.

\smallskip
\begin{assumption}
\label{ass:basicassumption}

As we would like, in the sequel, to treat uniformly both  the ``classical setting"  and the ``$p$-adic setting", we make the following unifying notation. We denote by: $K^o$ one of $K, \fK$, by $\cO_{K^o}$ one of $K, B$, by $P^o$ one of $P, \PP$, by $L^o$ one of $L, \cL$, and by $T^o$ one of $T, \bT$ such that  in the $p$-adic setting we have $\{K^o, \cO_{K^o}, P^o, L^o, T^o\}=\{\fK, B, \PP, \cL, \bT\}$ and in the classical setting $\{K^o, \cO_{K^o}, P^o, L^o, T^o\}=\{K, K, P, L, T\}$.

a)  {\bf Weights}.In the classical case we consider a weight $k$ of $T$ which is integral, regular and dominant for $\Phi$. 

In the $p$-adic case we let $k$ be a $B$-valued analytic character of $\bT$.  By deriving we may consider $k\colon \fh\to B$ as a $B$-valued weight of $\fh$. We assume that $k$ is {\em generic}, i..e, for any co-root $\alpha\in \Phi$  we have $\langle k, \alpha\rangle \not \in \Z$.

We denote such a weight, in both cases, as $k:T^o(\Z_p)\lra \cO_{K^o}^\ast$ and respectively $k:\fh\lra \cO_{K^o}$.

\smallskip
 
b) {\bf $G$-representations}. In the classical case let $W=\aW$ be an algebraic, finite dimensional, irreducible representation of $G$ over $K$ of highest weight $k$.

In the $p$-adic case denote by ${\rm Ind}_{\cB\cap \cL}^\cL(k)^{\rm an}$ the analytic induction from $\bB\cap \cL$ to $\cL$ of the character $k$. We view it as a $\PP$-module where the action is via the quotient $\PP\to \cL$. Then this module has a natural filtration $(F^s)_{s\in \N}$, by finitely generated $B$-modules and we let $\aW={\rm co-lim}_s F^s$. It is a $(U(\fg), \PP)$-module over $\fK$ in the sense of Definition \ref{def:fgB}.  In particular, it  has a natural, increasing and exhausting filtration $(F^s)_{s\in \N}$, preserved by $\PP$, and such that $F^0={\rm Ind}_{\bB\cap \cL}^\cL(k)$ and $\fn^-\cdot F^s= F^{s+1}$ for all $s\in \N$.

\smallskip

 c) {\bf Modules with connection and filtration}. We assume that there is a triple $\bigl(\WW_k^{\rm alg}, \cF^\bullet, \nabla_k\bigr)$, where $\WW_k^{\rm alg}$ is an $\cO_{\cX}$-module, with an increasing, exhausting and separated filtration $\cF^\bullet$, indexed by $\Z$, where $\cF^s$ is a sheaf of locally free
 $\cO_{\cX}$-modules for  each $s\in \Z$ and $\cF^s=0$ for $s<0$. In the classical case $\WW_k^{\rm alg}$ is a finite, locally free $\cO_{\cX}$-module and therefore the filtration is finite. We have a logarithmic and integrable connection
 $\nabla\colon \WW_k^{\rm alg}\to \WW_k^{\rm alg}\otimes_{\cO_{\cX}}\omega^1_{\cX/S}$, satisfying Griffiths transversality, i.e. $\nabla_k(\cF^s)\subset \cF^{s+1}\otimes_{\cO_{\cX}}\omega^1_{\cX/S}$.
 
\smallskip

 d) {\bf Local properties}.  We assume that there is an open, affinoid covering $\bigl(U_i\bigr)_{i\in I}$  of $\cX$ and for each $i\in I$ there are 

\begin{itemize} 

\item[i.] $\cO_{U_i}$-linear isomorphisms $\varphi_i\colon \bigl(\WW_k^{\rm alg}\otimes_{\cO_{K^o}}K^o\bigr)\vert_{U_i}\to \aW\otimes_{\cO_{K^o}}\cO_{U_i}$,
 which induce isomorphisms $(\cF^s\otimes_{\cO_{K^o}}K^o)\vert _{U_i}\cong F^s\otimes_{\cO_{K^o}}\cO_{U_i}$, for every $s\in \Z$;

\item[ii.] abelian Lie subalgebras $\eta_i\colon \fn_i\to  \fg\otimes \cO_{U_i}$ whose composite to $\fg/\fp\otimes \cO_{U_i}$ defines an isomorphism $\fn_i\cong \fg/\fp\otimes_{K^o} \cO_{U_i}$

\item[iii.]   isomorphisms $\psi_i\colon \cP_{U_i} \to  \cO_{U_i} \otimes U(\fn_i)^\vee $  of filtered, left  $\cO_{U_i}$-algebras. Here the descending filtration on $U(\fn_i)^\vee$ is orthogonal to the ascending filtration on $U(\fn_i)$.  In particular, we get an  induced  isomorphism  $\overline{\psi}_i \colon \omega_{\cX/S}^1\vert _{U_i}\cong (\fg/\fp)^\vee\otimes \cO_{U_i}  $ of $\cO_{U_i}$-modules. We require that for every $\alpha\in \fn_i$ the map $(\varphi_i\otimes 1) \circ \bigl((1\otimes \alpha) \circ  (1\otimes \overline{\psi}_i) \circ \nabla_i\bigr) \circ (\varphi_i^{-1}\otimes 1) \colon \aW\otimes_{\cO_{K^o}}\cO_{U_i} \to \aW\otimes_{\cO_{K^o}}\cO_{U_i}  $ is the unique derivation induced by the map $\eta_i(\alpha) $ via the $\fg$-module structure on  $\aW$.

\end{itemize}

e) {\bf Gluing properties}. Moreover, we assume that for $i$, $j\in I$, denoting  $U_{i,j}$ the intersection $U_i\times_SU_j$,  there exists $p_{i,j}\in P^o(U_{i,j})$ such that the isomorphism $(\varphi_j\vert_{U_{i,j}})\circ (\varphi_i^{-1}\vert_{U_{i,j}})\colon \aW\otimes_{\cO_{K^o}}\cO_{U_{i,j}}\cong   
 \aW\otimes_{\cO_{K^o}} \cO_{U_{i,j}}$  is given by $p_{i,j} \otimes 1$,  multiplication by $p_{i,j}$,  using the structure of $P^o$-module of $\aW$. Moreover, for every $\alpha\in \fn_j$ we have $\eta_j(\alpha)=(c_{p_{i,j}} \otimes 1) \circ   \eta_i(\alpha)+p_{i,j}^{-1} \eta_i(\alpha)\bigl(\overline{\psi}_i (d p_{i,j})\bigr)  $.  In particular, using the identification $\fn_j\cong \fg/\fp\cong \fn_i$, the dual of $\overline{\psi}_i\vert_{U_{i,j}}\circ \overline{\psi}_j^{-1}\vert_{U_{i,j}}$ is given by conjugation $c_{p_{i,j}}$ by $p_{i,j} $ on $\fg/\fp$.


 \end{assumption}
 
 \bigskip

 Under these assumptions we will prove in the sequel that there is a canonical complex of locally free $\cO_{\cX}\otimes_{\cO_{K^o}}K^o$-modules, ${\rm BGG}^\bullet$, with   morphisms 
 $ \WW_{k, \fK}^{\rm alg}\otimes_{\cO_{\cX}}\omega_{\cX/S}^\bullet\to {\rm BGG}^\bullet $, which induce {\bf quasi-isomorphisms} of abelian complexes. Here the left-hand-side of the isomorphism is the log de Rham complex 
 of $(\WW_{k,\fK}^{\rm alg}, \nabla_k)$, and the morphisms between the two complexes are log differential operators, not $\cO_{\cX}$-linear maps.
 As a consequence, we obtain isomorphisms ${\rm H}^m_{\rm dR}\bigl(\cX, \WW_{k,\fK}^{\rm alg}\otimes_{\cO_{\cX}}\omega_{\cX/S}^\bullet\bigr)\cong \bH^m\bigl(\cX, {\rm BGG}^\bullet\bigr)$,
 for all $m\ge 0$. 
 
 
In section \S \ref{section:BGG} we showed that the Assumption \ref{ass:basicassumption} is satisfied in the PEL setting for automorphic vector bundles arising from algebraic representations of $G$.

In section \S \ref{sec:applyelliptic} we give the example of elliptic modular forms, where $\fW_k$ is infinite dimensional, $\fg:=\mathfrak{sl}_{2/\Q}$ and $\fp=\fb$ is a Borel sub-algebra and we will show that the Assumption \ref{ass:basicassumption} is satisfied in that case.
So far we think this set-up should work for Hilbert modular forms, i.e. $\fg=\mathfrak{sl}_{2/F}$, $F$ a totally real field, $\fp$ is a Borel sub-algebra, and in a future 
article we plan on explaining how this set-up can be adapted for Siegel modular forms
(i.e. $\fg=\mathfrak{sp}_{2g/\Q}$, where we consider a parabolic sub-algebra $\fp$, which could be a Siegel or Klingen parabolic).

 \bigskip
 
 \bigskip

  \subsection{Lie-algebra Lemmas on cutting the Koszul complex by the center of $U(\fg)$. }
  \label{sec:liealgebra}

As we work in a very general setting, having fixed a sequence $\fb\subset \fp\subset \fg$ of a parabolic Lie-sub-algebra of a semi-simple Lie-algebra $\fg$ over $\Q_p$ (see Assumption \ref{ass:basicassumption}), we need to understand the structure of the modules resulting in the cutting of a $U(\fg)$-module by $Z(\fg)$, the center of $U(\fg)$.  We recall the unifying notation from the Assumption \ref{ass:basicassumption}: we denoted by: $K^o$ one of $K, \fK$, by $\cO_{K^o}$ one of $K, B$, by $P^o$ one of $P, \PP$, by $L^o$ one of $L, \cL$, such that  in the $p$-adic setting we have $\{K^0, \cO_{K^o}, P^0, L^0\}=\{\fK, B, \PP, \cL\}$ and in the classical setting $\{K^o, \cO_{K^o}, P^o, L^o\}=\{K, K, P, L\}$.

 \begin{lemma}[see also Prop. 7.6.14, \cite{diximier}]
\label{lemma:diximier}
Let $F$ be a $P^o$-module and $E:=U(\fg)\otimes_{U(\fp)}F$ the corresponding $(U(\fg),P^o)$-module. Assume that $F$ is the direct sum $F:=F_1\oplus\cdots \oplus F_s$ of $L^o$-submodules  $F_1$,  $\ldots$, $F_s$  such that for every $i=1,\ldots,s$ the $U(\fm)$-module $F_i$ is a finite dmensional, irreducible,  highest weight $\fm$-module with  weight  $\mu_i$. For $1\le i,j\le s$, if $\mu_i-\mu_j$ is a positive combination of weights in $\Phi^+\backslash \Phi_\fm^+$, we write  $\mu_i>\mu_j$. Suppose the indexes $1$, $\ldots$, $s $ are totally ordered so that if $\mu_i>\mu_j$, then $i<j$. Define $E_i$ to be the $U(\fg)$-submodule of $E$ generated by $F_1+ \ldots + F_i$. Then 

\begin{itemize}

\item[i.] $E_0=(0)\subset E_1\subset E_2\subset \ldots \subset E_s=E$ is a filtration of $E$ by $(U(\fg), P^o)$-submodules such that $E_i/E_{i-1}\cong U(\fg)\otimes_{U(\fp)} F_i$;

\item[ii.] for every $i=1,\ldots,s$ the $\fg$-module $U(\fg)\otimes_{U(\fp)} F_i$ is a parabolic Verma module in the sense of \cite[\S 3, Def. on p. 500]{lepowsky} or \cite[\S 9.4]{humphreys}. In particular, it is a highest weight module with highest weight $\mu_i$;

\item[iii.] for every $j=1,\ldots,s$ the $\fp$-submodule $U(\fp) F_j\subset F$ is contained in $\oplus_{i\vert \mu_j \leq \mu_i} F_i$;

\item[iv.] assume that $k$ is a $p$-adic weight,   as in \S \ref{ass:basicassumption}(ii). For every $i<j$ such that $\mu_i=w_i \cdot k$ and $\mu_j=w_j\cdot k$ are (dot)-linked to  $k$, we have $\mu_j \nless \mu_i$.

\item[v.] assume $k$ is a classical weight, see \S \ref{ass:basicassumption}, and that $F$ is a simple $U(\fg)$-module of highest weight $k$. Then, $F_1$ has weight $k$ and for every $i=2,\ldots, s$ the weight  $\mu_i$ of $F_i$ is not  (dot)-linked to  $k$.

\end{itemize} 
\end{lemma}

\begin{proof} 

 (i) We first remark that $U(\fp)F_1+\ldots +U(\fp)F_i\subset F_1+\ldots +F_i$. Indeed, if $x\in \fn^+$ has weight $\rho$, $1\leq \ell \leq i$ and $v_-\in F_\ell$ is a highest weight vector with weight $\mu_\ell$, for every $\alpha\in h $ we have $\alpha(x v_-)=\bigl(\mu_\ell(\alpha)+\rho(\alpha)\bigr) x v_- \in \sum_{\mu'>\mu_\ell} F_{\mu'}$. Hence, $\fn^+ v_- \subseteq \sum_{\mu'>\mu_\ell} F_{\mu'} $. Similarly, for every $u\in \fm$ and $x\in \fn^+$ we have $ x(u v_-)=u (x v_-)+[x,u] v_-$. Since $\fn^+\subset \fp$ is an ideal,  $[x,u]\in \fn^+$  and we conclude that $x(u v_-)\in \sum_{\mu'>\mu_\ell} F_{\mu'} $. Continuing inductively  we  get that $x F_\ell\subseteq \sum_{\mu'>\mu_\ell} F_{\mu'}$. Therefore,  $x F_\ell\subseteq \sum_{j=1}^{\ell -1} F_j$.

Let us now look at the image $\overline{F}_i $ of $F_i$ in $E_i/E_{i-1}$. By the argument above if  $\alpha$ is a positive root of $\fg$ in $\Phi^+\backslash \Phi_\fm^+$  and $x_\alpha\in \fg^\alpha$, we have $x_\alpha F_i \subset F_1+\ldots +F_{i-1}= E_{i-1}$. Therefore $x_\alpha \overline{F}_i=0$.
Finally, let us notice that $E_i=U(\fg)\otimes_{U(\fp)}(F_1+\ldots+ F_i)\cong U(\frn^-)\otimes_{\fK}(F_1+\ldots+ F_i) $ is the direct sum of the $U(\frn^-)$-modules  $U(\frn^-)\otimes_{\fK}F_1$, $\ldots$, $U(\frn^-)\otimes_{\fK}F_i$.  
Therefore $E_i/E_{i-1}\cong U(\frn^-)\overline{F}_i \cong U(\frn^-)\otimes_{\fK} F_i$ and, hence, $E_i/E_{i-1}\cong U(\fg)\otimes_{U(\fp)} F_i$. 

(ii) follows from the definition and properties of parabolic Verma modules and (iii)  follows from the discussion in  (i).   

(iv) Assume that there exist $i<j$ such that $w_j \cdot k=\mu_j < \mu_i=w_i\cdot k$ and $F_i\subset U(\fp) F_j$.

Given an element $w=s_{\alpha_1}\cdots s_{\alpha_n}$ of the Weyl group written in reduced form, for every weight $\mu$ we have $$\mu- w \cdot \mu=\sum_{r=1}^n  s_{\alpha_1}\cdots s_{\alpha_{r-1}} (1-s_{\alpha_r})(\mu+\rho)= \sum_{r=1}^n \langle \mu+\rho, \alpha_r^\vee\rangle \beta_r;$$with $\beta_r:=  s_{\alpha_1}\cdots s_{\alpha_{r-1}} (\alpha_r)$. Recall that $\{\beta_1,\ldots,\beta_n\}$ is the inversion set ${\rm Inv}_\Phi(w)$ of $w$, i.e., the set $\{\alpha\in \Phi^+\vert w \alpha\in \Phi^-\}$. In our case, take $\mu:=\mu_i$ and $w=w_j \cdot w_i^{-1}$ . We conclude that
$$\mu_i-\mu_j= \sum_{r=1}^n \langle (w_i\cdot k) +\rho, \alpha_r^\vee\rangle \beta_r= \sum_{r=1}^n \langle k +\rho, (w_i^{-1}(\alpha_r))^\vee\rangle \beta_r .$$Since $k$ is generic and the $\beta_i$'s are positive roots, this latter expression can not be a positive combination of positive roots. 

(v) Consider the map of $U(\fg)$-modules $E\to F$. Since the restriction to $E_1$ is non-trivial, mapping $1\otimes v\mapsto v$ for every $v\in F_1$,  and since $F$ is simple, we conclude that the map $E_1\to F$ is surjective. In particular, $E_1$ and $F$ have the same character $\chi_k$. For every $2\leq i \leq s$, the highest weight $\mu_i$ of $F_i$ and of $E_i$ satisfies $\mu_i< k$. Assume that $\chi_{\mu_i}=\chi_k$. Then there exists a non-trivial element $w$ of the Weyl group $W$ such that $\mu_i=w \cdot k= w (k+\rho)-\rho$. Then $w^{-1}\mu_i=k+\rho-w^{-1} \rho=k+\beta_w$ is a weight appearing in the $\fg$-representation $F$, as this set is $W$-stable, and $k< w^{-1}\mu_i$, as $\beta_w=\rho-w^{-1}\rho\in \Phi^+$. This contradicts the fact that $k$ is the highest weight of $F$. 

\end{proof}

With the notation of the lemma, assume first that $k$ is a $p$-adic weight and  consider the $L^o$-submodule $\iota_k\colon F_{\chi_k}:=\oplus_{\chi_{\mu_i}=\chi_k} F_i \hookrightarrow F$ of $F$; the sum is taken over $i$'s such that the central character of $\mu_i$ is the same as the central character of $\chi_k$ or, equivalently by Harish-Chandra's theorem \cite[Thm. 1.10]{humphreys},  over the $i$'s such that $\mu_i$ is dot-linked to $k$.

Or, assume that $F$ is a $U(\fg)$-module and given a dominant weight  $k$, let $F_{\chi_k}\subset F$ be the sum of the irreducible $U(\fg)$-submodules of $F$ of central character $\chi_k$. We denote by $\iota_k$ this inclusion. In both cases, let $\pi_k\colon F \to F_{\chi_k}$ be the morphism of $\cL$-modules given by modding out the sum of the $F_j$'s such that $\chi_j\neq \chi_k$. Let
\begin{equation}\label{eq:jpichi} 
j_{\chi_k} \colon E_{\chi_k} \longrightarrow E, \quad \pi_{\chi_k}\colon E \longrightarrow E_{\chi_k}
\end{equation} be the $U(\fg)$-direct factor of $E$ on which the center $Z(\fg)$ acts via the character $\chi_k$. In particular, $j_{\chi_k}$ and $\pi_{\chi_k}$ are homomophisms of $U(\fg)$-modules,   $\pi_{\chi_k} \circ j_{\chi_k}={\rm Id}$  and $\mathfrak{e}_{\chi_k}:=j_{\chi_k}\circ \pi_{\chi_k}$ is an idempotent.

\begin{corollary}\label{cor:glueinfBGG} Suppose that the assumption on $k$ in Lemma \ref{lemma:diximier}(iv) holds for every $i$ and $j$ such that $\mu_i$ and $\mu_j$ are dot-linked to $k$. Then,

\begin{itemize}

\item[a.]  $j_{\chi_k}$ and $\pi_{\chi_k}$ are homomorphisms of $(U(\fg),P^o)$-modules, $F_{\chi_k}$ is a $U(\fp)$-submodule of $F$ on which the action of $U(\fp)$ factors via $U(\fm)$ and  the map $\pi_{\chi_k} \circ \iota_k\colon F_{\chi_k} \longrightarrow U(\fg)\otimes_{U(\fp)}F =E \longrightarrow E_{\chi_k} $ induces an isomorphism $\rho_k\colon U(\fg)\otimes_{U(\fp)}F_{\chi_k} \longrightarrow E_{\chi_k}$ of $(U(\fg),P^o)$-modules.

\item[b.]  for every $\gamma\in P^o$ with projection $\delta \in L^o$, we have $c_\delta=\pi_k \circ c_\gamma \circ \iota_k\colon F_{\chi_k} \rightarrow F_{\chi_k}$; here the action is by conjugation. Analogously, for every  $\alpha\in \fp$ with projection $\beta \in \fm$, we have $\beta=\pi_k \circ \alpha \circ \iota_k\colon F_{\chi_k} \rightarrow F_{\chi_k}$.  

\item[c.]   given $\gamma\in P^o$ with projection $\delta \in L^o$, we have $(1\otimes c_\gamma)\circ j_{\chi_k}=  j_{\chi_k} \circ \rho_k \circ (1 \otimes c_\delta)\circ \rho_k^{-1}   \colon  E_{\chi_k} \rightarrow E $. Similarly for $\alpha\in \fp$ with projection $\beta \in \fm$ we have $\alpha \circ j_{\chi_k}= j_{\chi_k}\circ  \rho_k \circ \beta\circ \rho_k^{-1}   \colon  E_{\chi_k} \rightarrow E $.
\end{itemize}

\end{corollary}
\begin{proof} Note that the action of $Z(\fg)$ preserves the filtration defined in  Lemma \ref{lemma:diximier} and, hence, also the operation of cutting via the character $\chi_k$. It follows from Lemma \ref{lemma:diximier} that $F_{\chi_k}\subset F$ is a $U(\fp)$-submodule on which the action of $U(\fp)$ factors via $U(\fm)$. This implies the second part of (b). Moreover, by loc. cit., $\iota_k$ induces an isomorphism of $U(\fg)$-modules $U(\fg)\otimes_{U(\fp)} F_i\cong (E_i/E_{i-1})_{\chi_k}$ if $\mu_i$ is dot-linked to $k$. Otherwise, $(U(\fp) F_i)_{\chi_k}=0$ and $(E_i/E_{i-1})_{\chi_k}=0$ so that $\pi_{\chi_k}\circ \iota_k(F_i)=0$. We conclude that $\pi_{\chi_k}\circ \iota_k$ defines the isomorphism of $U(\fg)$-modules $\rho_{\chi_k}$ as claimed in (a).

The other statements of (a) and (b) also follow from  Lemma \ref{lemma:diximier} except for the claim on the $P^o$-equivariance of $j_{\chi_k}$ and $\pi_{\chi_k}$  and the claim that $c_\delta=\pi_k \circ c_\gamma \circ \iota_k$ for $\gamma\in P^o$. The $P^o$-equivariance of $j_{\chi_k}$ and $\pi_{\chi_k}$ follow from the $U(\fp)$-equivariance using the exponential; see the proof of Lemma \ref{lemma:centrocommuta}. Also the statement that $c_\delta=\pi_k \circ c_\gamma \circ \iota_k$ for $\gamma\in P^o$  follows from the statement that for every $\alpha\in \fp$ with projection $\beta \in \fm$, we have $\beta=\pi_k \circ \alpha \circ \iota_k$, using the exponential.

Statement (c) follows from (b) and the fact that $\rho_k$ and $j_{\chi_k}$ are homomorphisms of $(U(\fg),P^o)$-modules.
\end{proof}

Therefore, if $F$ is as above and $\nu$ is a weight with character $\chi_{\nu}$ of $Z(\fg)$,  then $\bigl(U(\fg)\otimes_{U(\fp)}F\bigr)_{\chi_{\nu}}$ has a filtration by $U(\fg)$-sub-modules with graded quotients Verma modules $M(\mu_i)$, where $1\le i\le s$, with
$\nu=w\bullet \mu_i$ for some $w\in W_G$, the Weyl group of $G$.  Let $\Theta:=\Theta_{U(\fg)\otimes_{U(\fp)}F}$ be the finite set of characters $\chi$ of $Z(\fg)$ such that $\bigl(U(\fg)\otimes_{U(\fp)}F\bigr)_{\chi}\neq 0.$ Then $U(\fg)\otimes_{U(\fp)}F\cong \oplus_{\chi\in \Theta}\bigl(U(\fg)\otimes_{U(\fp)}F\bigr)_{\chi}$.

\begin{lemma}
\label{lemma:centrocommuta}
Suppose $Q$ is a $K^o$-vector space which is a $(U(\fg), P^o)$-module, with an increasing, exhausting and separated filtration $(Q^n)_{n\in \Z}$ such that $Q^n=0$ if $n<0$. We recall that this implies $Q=\colim_n Q^n$.
Suppose that the action  $P^o\times Q\lra Q$ of $P^o$ is analytic and preserves the filtration, i.e. there is a locally free $\cO_{K^o}$-module $Q^o$, with a filtration $(Q^{o,n})_n$, which is a $(U(\fg), P^o)$-module,  
with the $P^o$-action analytic and preserving the filtration and such that we have an isomorphism $Q\cong Q^o\otimes_{\cO_{K^o}}K^o$ as $\bigl(U(\fg), P^o\bigr)$-modules with filtrations.

Then, for every $q\in Q$, every $\sigma\in Z(\fg)$ and $p\in P^o$ we have $p^{-1}(\sigma(pq))=\sigma q$.
\end{lemma}

\begin{proof}
We fix $q\in Q$ and $\sigma\in Z(\fg)$, and consider the map: $\alpha_{q, \sigma}:P^o\lra Q$ defined by $\alpha_{q,\sigma}(p):=p^{-1}(\sigma(pq))-\sigma q$, for all $p\in P^o$.  Then $\alpha_{q, \sigma}$ has image in some 
$Q^n$, $n\in \N$ and it is an analytic map, in the sense above.
Let $\fp^o\subset \fp$ be an open neighbourhood of $0$ such that $\exp_\fg$ is defined on $\fp^o$ and suppose $p=\exp_{\fg}(p^o)$, with $p^o\in \fp^o$.
Then we have

$$
p^{-1}(\sigma(pq))=\exp_{\fg}\bigl({\rm ad}(p^o)\bigr)(\sigma) q.
$$ 
But, as $\sigma\in Z(\fg)=Z\bigl(U(\fg)\bigr)$, we have ${\rm ad}(p^o)(\sigma)=p^o\sigma-\sigma p^o=0$. Therefore we have
$$
\exp_{\fg}\bigl({\rm ad}(p^o)\bigr)(\sigma) q=\Bigl(\sum_{n=0}^\infty \frac{\bigl({\rm ad}(p^o)\bigr)^n}{n!}\Bigr)(\sigma) q=\sigma q,
$$
In other words $\alpha_{q,\sigma}(p)=0$ for all $p\in \exp_{\fg}(\fb^0)$ and $\exp_{\fg}(\fp^o)$ is a non-trivial open subset of $\cP$. By analyticity, $\alpha_{q,\sigma}=0$, which proves the Lemma. 

\end{proof}

\section{BGG-decomposition in the classical case.}\label{section:BGG}

We first work out the ``classical case", i.e.,  we suppose that we are in the setting of \S \ref{section:BGG} so that Assumption \ref{ass:basicassumption}  is satisfied. More precisely,  let $W$ be a finite, irreducible representation of $G\otimes K$ of weight $k$ giving rise to $(\WW, {\rm Fil}_\bullet, \nabla)$, the coherent sheaf on $\cX$ with a finite, separated and exhausting filtration by coherent sub-sheaves and an integrable logarithmic connection $\nabla:\WW \lra \WW\otimes_{\cO_{\cX}}\omega^1_{\oA/\cX}$. Therefore we have a de Rham complex on $\cX$, $\WW^\bullet: =\WW\otimes_{\cO_{\cX}}\omega^\bullet_{\oA/\cX}$.

On the other hand, if we denote $M:=W^\vee$, we have a natural structure of $(\fg, P)$-module on the Koszul complex of $M$:
$$
\cD^\bullet: = U(\fg)\otimes_{U(\mathfrak{p})}\bigl(M \otimes_K\wedge^\bullet (\fg/\mathfrak{p})\bigr).
$$
Because $M$ is a $U(\fg)$-module, Garland-Lepowski's theorem implies that each term of the complex is: $U(\fg)\otimes_{U(\fp)}(M\otimes_K\wedge^n(\fg/\fp))\cong M\otimes_K\bigl(U(\fg)\otimes_{U(\fp)}\wedge^n(\fg/\fp)\bigr)$. In other words $\cD^\bullet$ is the tensor product 
over $K$ of $M$ and the Koszul complex for the trivial $U(\fg)$-representation $K$. Therefore $\cD^\bullet$ is a resolution of $M$.

Let $Z(\fg)$ denote the centre of $U(\fg)$ and consider the sub-complex of $\cD^\bullet$, $\cV^\bullet$, defined as follows.
Let $k^\vee$ be the character of the Cartan sub-algebra $\fh$ of $\fg$ which gives the action of $\fh$ on $M$, and let
$\chi_{k^\vee}$ be the associated character of $Z(\fg)$. Then $\cV^\bullet$ is the generalized eigenspace of $\cD^\bullet$ for 
the elements $\sigma\in Z(\fg)$, with eigenvalues $\chi_{k^\vee}(\sigma)$. Arguing as in  Corollary \ref{cor:glueinfBGG} each term of the complex $\cV^\bullet$ is  a $(U(\fg), P)$-module, in fact a Verma module. Moreover, as cutting with the centre is an exact operation and
$M_{\chi_{k^\vee}}=M$, then $\cV^\bullet$ is another resolution of $M$, and a direct summand of $\cD^\bullet$.

If we denote by $\cU^\bullet:=\cD^\bullet/\cV^\bullet$, we have an exact sequence of complexes (which is canonically split)
$$
0\lra \cV^\bullet \stackrel{\alpha}{\lra} \cD^\bullet \stackrel{\gamma}{\lra} \cU^\bullet\lra 0
$$ 
such that $\alpha$ is a quasi-isomorphism, therefore $\gamma$ induces $0$ in homology.

Now we dualize this sequence of complexes, i.e. consider the sequence
$$
(\ast)\quad 0\lra (\cU^\bullet)^\vee \stackrel{\gamma^\vee}{\lra}(\cD^\bullet)^\vee\stackrel{\alpha^\vee}{\lra}(\cV^\bullet)^\vee\lra 0.
$$
The next step is to sheafify the sequence of complexes $(\ast)$. For this let us recall the open, affinoid covering  $(U_i)_{i\in I}$ of $\cX$. For every $i\in I$ we have (see Assumption \ref{ass:basicassumption} d)).
\begin{itemize}
\item[i)] isomorphisms $\varphi_i\colon \WW|_{U_i}\cong W\otimes_K \cO_{U_i}$;

\item[ii)] injective morphisms of Lie-algebras  $\eta_i\colon \fn_i\lra \fg\otimes_K\cO_{U_i}$;

\item[iii)] isomorphisms of filtered $\cO_{U_i}$-modules $\psi_i\colon \cP_{U_i/S}\cong \cO_{U_i}\otimes_K U(\fn_i)^\vee$, which induce
isomorphisms $\overline{\psi}_i\colon \omega^1_{U_i/S}\cong \cO_{U_i}\otimes_K(\fg/\fp)^\vee$. Moreover, from this data, one can recover the connection 
$\nabla_i:=\nabla\vert_{(\WW\vert_{U_i})}$.
\end{itemize}

Therefore, for each $i\in I$, if $Y$ is a $P$-module, we identify $U(\fg)\otimes_{U(\fp)}Y\cong U(\fn_i)\otimes_K Y$ and by dualizing, we have
$\bigl(U(\fg)\otimes_{U(\fp)}Y\bigr)^\vee\cong \bigl(U(\fn_i)\bigr)^\vee\otimes_K Y^\vee$, and moreover using iii) above, an identification
$\bigl(U(\fg)\otimes_{U(\fp)}Y\bigr)^\vee\otimes_K\cO_{U_i}\cong \cP_{U_i/S}\otimes_K Y^\vee\cong \cP_{U_i/S}\otimes_{\cO_{U_i}}(\cO_{U_i}\otimes_KY^\vee)$.

In other words, for every $i\in I$, we have complexes of $\cO_{U_i}$-modules, where the differentials are differential operators, not $\cO_{U_i}$-linear maps,  $Q_i^\bullet$, $E_i^\bullet$ such that we have isomorphisms of $\cO_{U_i}$-modules $Q_i^\bullet\cong \cO_{U_i}\otimes_K(\cU^\bullet)^\vee$ and $E_i^\bullet\cong \cO_{U_i}\otimes_K(\cV^\bullet)^\vee$. Reinterpreting this using linearization, see \S \ref{sec:linearization},  we get an exact sequence of linearized complexes of sheaves on $U_i$:
$$
(\ast\ast)\quad 0\lra L(Q_i^\bullet)\stackrel{\gamma_i^\vee}{\lra} \bigl(L(\WW\otimes_{\cO_{\cX}}\omega^\bullet_{\cX/S})\bigr)|_{U_i}\stackrel{\alpha_i^\vee}{\lra}L(E_i^\bullet)\lra 0. 
$$

Let $i$, $j\in I$ be as in Assumption \ref{ass:basicassumption} iv) and let $p_{i,j}\in P(U_{i,j})$ be an element such that
the identification $\cP_{U_{i,j}}\otimes_{ \cO_{U_{i,j}}}\bigl(\WW\otimes\omega^t_{U_i/S}\bigr)\vert_{U_{i,j}}= \cP_{U_{i,j}}\otimes_{ \cO_{U_{i,j}}}\bigl(\WW\otimes\omega^t_{U_j/S}\bigr)\vert_{U_{i,j}}$ is realized by multiplication by $p_{i,j}$ on $W$ using $\varphi_j\vert_{U_{i,j}}\circ \varphi_i^{-1}\vert_{U_{i,j}}$ and by conjugation by $p_{i,j}$ on $\fg/\fp\otimes_K \cO_{U_{i,j}}$ using the dual of $\overline{\psi}_i\vert_{U_{i,j}} \circ \overline{\psi}_j^{-1}\vert_{U_{i,j}}$. Recall that the latter action factors through the image of $p_{i,j}$ in the Levi quotient of $P$. Thus we can apply Corollary \ref{cor:glueinfBGG} and Lemma 
\ref{lemma:centrocommuta} and we deduce that right multiplication by the same $p_{i,j}$ gives an isomorphism:
$\cP_{U_{i,j}}\otimes_{\cO_{U_{i,j}}}\bigl(Q_i^t\bigr)\vert_{U_{i,j}}\cong \cP_{U_{i,j}}\otimes_{\cO_{U_{i,j}j}}\bigl(Q_j^t\bigr)\vert_{U_{i,j}}$, and also the isomorphism
$\cP_{U_{i,j}}\otimes_{\cO_{U_{i,j}}}\bigl(E_i^t\bigr)\vert_{U_{i,j}}\cong \cP_{U_{i,j}}\otimes_{\cO_{U_{i,j}}}\bigl(E_j^t\bigr)\vert_{U_{i,j}}$.
Therefore, the families of complexes $\{L(Q_i^\bullet)\}_{i\in I}$ and $\{L(E_i^\bullet)\}_{i\in I}$ glue to give complexes $L(Q^\bullet)$ and $L(E^\bullet)$ on $\cX$ such that for each $i\in I$ we have isomorphisms of complexes: $L(Q^\bullet)\vert_{U_i}\cong L(Q_i^\bullet)$ and
  and $L(E^\bullet)\vert_{U_i}\cong L(E_i^\bullet)$.  We remark that by Corollary \ref{cor:glueinfBGG} the families of morphisms of complexes $(\alpha_i)_{i\in I}$ and $(\gamma_i)_{i\in I}$ glue. 

In conclusion, gluing  the complexes $L(Q_i^\bullet)$ and $L(E_i^\bullet)$ and the families of morphisms $(\alpha_i)_{i\in I}$ and $(\gamma_i)_{i\in I}$ we get complexes $L(Q^\bullet)$ and $L(E^\bullet)$ on $\cX$ and an exact sequence of complexes of sheaves on $\cX$: $$
(\ast\ast)\quad 0\lra L(Q^\bullet)\stackrel{(\gamma^\bullet)^\vee}{\lra} \bigl(L(\WW\otimes_{\cO_{\cX}}\omega^\bullet_{\cX/S})\bigr)|_{U_i}\stackrel{(\alpha^\bullet)^\vee}{\lra}L(E^\bullet)\lra 0. 
$$

The first observation is that, as $\gamma^\bullet$ induces $0$ in cohomology, by the arguments in section \ref{sec:dual}, it follows that 
$(\gamma^\bullet)^\vee$ also induces $0$ in the cohomology of the complexes above.

\bigskip
 
The next step in the construction is  to remove the linearizations and show that the morphism of differential operators attached to $\alpha^\vee$ is a quasi-isomorphism of complexes. 

For this we consider the exact sequence of complexes $(\ast\ast)$ as an exact sequence of complexes of crystals on the log infinitesimal site of $\cX/S$, denoted $\XSinf$ (see section 5) and apply the functor $\pi_\ast$, defined in section  \ref{sec:delin}, to it. We obtain the exact sequence of complexes of sheaves on $\cX$
$$
0\lra Q^\bullet\stackrel{\pi_\ast(\gamma^\vee)}{\lra} \WW\otimes_{\cO_{\cX}}\omega^\bullet_{\cX/S}\stackrel{\pi_\ast(\alpha^\vee)}{\lra} E^\bullet\lra 0
$$
The sequence is either exact because $\pi_\ast$ is left exact and ${\rm R}^i\pi_\ast(L(\ -\ ))=0$ for $i\ge 1$, or because the sequence is split.

Now we apply the results of section \S \ref{sec:spectralseq} to the abelian categories, (the first with enough injectives) the category of abelian sheaves on the log infinitessimal site of $\cX$ and the category of abelian sheaf on the log analytic site of $\cX$, and the right exact functor $\pi_\ast$, which imply that $\pi_\ast(\gamma^\vee)$ induces $0$ in cohomology. It follows that
$\pi_\ast(\alpha^\vee)$ is a quasi-isomorphism of complexes. We conclude that the morphism
$$\pi_\ast(\alpha^\vee)\colon \WW\otimes_{\cO_{\cX}}\omega^\bullet_{\cX/S}\lra E^\bullet$$ is a quasi-isomorphism of complexes of abelian sheaves on $\cX$.

\subsection{The dual picture}\label{SubsectionDualTheorem}\label{sec:dual}

In this section we prove that $(\gamma^\bullet)^\vee$, of section \S \ref{section:BGG}, induces $0$ in cohomology, namely

\begin{proposition}\label{PropositionVanishingDualCohomology}
	With the notation of \S \ref{section:BGG} $(\ast\ast)$ we have ${\rm H}^i(\gamma^{\vee,\bullet})=0$ for each $i\in \N$.
\end{proposition}

We work in the slightly more general setting where $K^o$ is either the field $K$ or the field $\fK$ of  Assumption \ref{ass:basicassumption} as we will need to apply the results that follow also in the $p$-adic setting.  We consider the following natural pairing.
\begin{lemma}\label{LemmaPairings}
	Let $A_0\overset{\alpha}{\rightarrow} A_1\overset{\beta}{\rightarrow} A_2$ be a complex of $K^o$-vector spaces. We remark that we can choose a splitting, as $K^o$-vector spaces: $A_2\cong A_3\oplus \Imm(\beta)\cong A_3\oplus A_1/\Ker(\beta)$. Let $A_2\dual\overset{\beta\dual}{\rightarrow} A_1\dual\overset{\alpha\dual}{\rightarrow} A_0\dual$ be the dual complex, then there is a $K^o-$linear pairing
	\[
	\langle\ ,\ \rangle\, \colon {\rm H}(A)\times {\rm H}(A\dual)\longrightarrow S
	\]
	where $\displaystyle {\rm H}(A):=\frac{\Ker(\beta)}{{\rm Im}(\alpha)}$ and $\displaystyle {\rm H}(A\dual):=\frac{\Ker(\alpha\dual)}{\Imm(\beta\dual)}$. 
	If 
	\[
	\begin{tikzcd}
		A_0\ar[r,"\alpha"]\ar[d,"\varphi_0"]& A_1\ar[r,"\beta"]\ar[d,"\varphi_1"]& A_2\ar[d,"\varphi_2"]\\
		C_0\ar[r,"\epsilon"]& C_1\ar[r,"\delta"]& C_2
	\end{tikzcd}
	\]
	is a commutative diagram of $K^o$-vector spaces, where $\delta\circ\epsilon=0$; then the maps
	\[
	{\rm H}(\varphi)\colon {\rm H}(A)\rightarrow {\rm H}(C),\quad {\rm H}(\varphi\dual)\colon {\rm H}(C\dual)\rightarrow {\rm H}(A\dual)
	\]
	are adjoint w.r.t. the two pairings, i.e.
	\[
	\langle {\rm H}(\varphi)(a),b\rangle_C=\langle a, {\rm H}(\varphi\dual(b))\rangle_A,
	\]
	for each $a\in {\rm H}(A)$ and $b\in {\rm H}(C\dual)$.
	
	\[
	\text{fix }b\in {\rm H}(A\dual),\text{ if }\langle a,b\rangle=0\ \text{for each}\ a\in {\rm H}(A)\text{, then }b=0.
	\]
\end{lemma}
\begin{proof}
	For any $a\in {\rm H}(A),\ b\in {\rm H}(A\dual)$ we define $\langle a,b\rangle =:\phi(x)$ where $x\in A_1$ and $\phi\in A_1\dual$ are representative elements of $a$ and $b$. Observe that the pairing does not depend on the representative elements: if $\alpha(z)=y\in \Imm(\alpha)$ and $\beta\dual(\eta)=\psi\in \Imm(\beta\dual)$, then 
	\[
	(\phi+\psi)(x+y)=\phi(x)+ \phi(\alpha(z))+ \eta(\beta(x+y))=\phi(x)
	\]
	since $x,y\in \Ker(\beta)$ and $\phi\in \Ker(\alpha\dual)$.
	
	Let $a\in {\rm H}(A), \ b\in {\rm H}(C\dual)$ and choose some representatives $x\in A_1$ and $\phi\in C_1\dual$, then
	\[
	\langle {\rm H}(\varphi)(a),b\rangle_C=\phi(\varphi_1(x))= \left(\varphi_1\dual(\phi)\right)(x)=\langle a, {\rm H}(\varphi\dual(b))\rangle_A.
	\]
	
	Let $a\dual\in {\rm H}(A\dual)$ and $\phi\in A_1\dual$ be an element representing $a\dual$. If $\langle a,a\dual\rangle=0$ for any $a\in {\rm H}(A)$, then $\phi(x)=0$ for any $x\in \Ker(\beta)$, then $\Ker(\beta)\subset \Ker(\phi)$ and we can define a morphism $\delta': A_1/\Ker(\beta)\rightarrow S$ s.t. the following diagram commutes
	\[
	\begin{tikzcd}
		A_1\ar[r, "\phi"]\ar[d, "\beta"] & K^o\\
		A_1/\Ker(\beta)\ar[ur, dashrightarrow, "\delta'"]&
	\end{tikzcd}.
	\]
	
	   Due to the splitting chosen in the beginning, we can extend $\delta'$ to a $K^o$-linear function $\delta\colon A_2\rightarrow K^o$ and we get that $\phi=\beta\dual (\delta)\in \Imm(\beta\dual)$ and $a\dual=0$.
\end{proof}
We get as corollary the following result: 
\begin{corollary}\label{CorollaryDualVanishingCohomology}
	Let $f^\bullet\colon A^\bullet\rightarrow C^\bullet$ be a morphism between bounded complexes of $K^o$-modules with ${\rm H}^i(f^\bullet)=0$ for each $i\in \Z$. Suppose that for $A^\bullet$ we have splittings as in the lemma. Then ${\rm H}^i((f^\bullet)\dual)=0$ for each $i\in \Z$.
\end{corollary}
\begin{proof}
	We can suppose that the two complexes are concentrated on $0\leqslant i\leqslant n$, i.e., that $A^i=C^i=0$ for each $i<0$ and for each $n<i$. By the Lemma \ref{LemmaPairings} we get two pairings for each $i\in \N$ with $0\leqslant i\leqslant n$:
	\[
	\langle \ ,\ \rangle_A\, \colon {\rm H}^{i}(A^\bullet)\times {\rm H}^{n-i}((A^\bullet)\dual) \longrightarrow K^o,\quad \langle \ ,\ \rangle_C\, \colon {\rm H}^{i}(C^\bullet)\times H^{n-i}((C^\bullet)\dual)\longrightarrow K^o.
	\]
	One can conclude using properties of this pairing and the adjunction property. Indeed for any $b\in {\rm H}^{n-i}((C\dual)^\bullet)$ and $a\in {\rm H}^{i}(A^\bullet)$, we have
	\[
	\langle a, {\rm H}^{n-i}((f^\bullet)\dual)(b)\rangle_A=\langle {\rm H}^{i}(f)(a),b\rangle _C=\langle 0,b\rangle_C=0
	\]
	hence ${\rm H}^{n-i}((f^\bullet)\dual)(b)=0$.
\end{proof}
The proof of the Proposition \ref{PropositionVanishingDualCohomology} is an immediate consequence of the Corollary \ref{CorollaryDualVanishingCohomology}.

\subsection{A spectral sequence argument.}\label{sec:spectralseq}

Suppose we have two abelian categories with a left exact functor $\pi_\ast\colon \cA\lra \cB$, and a morphism of bounded below by $0$ complexes
$\phi^\bullet\colon T^\bullet \lra Q^\bullet$, such that for every $n\in \Z$, ${\rm H}^n(\phi^\bullet)=0$, where ${\rm H}^n(\phi^\bullet):{\rm H}^n(T^\bullet)\lra {\rm H}^n(Q^\bullet)$.

\begin{lemma}
\label{lemma:ssequence}
Suppose that $\cA$ has enough injective objects, for all $i>0, j\ge 0$ we have $R^i\pi_\ast(T^j)=R^i\pi_\ast(Q^j)=0$.
Then ${\rm H}^n(\pi_\ast(\phi^\bullet))=0$, where ${\rm H}^n\bigl(\pi_\ast(\phi^\bullet)\bigr): {\rm H}^n\bigl(\pi_\ast(T^\bullet)\bigr)\lra {\rm H}^n\bigl(\pi_\ast (Q^\bullet)\bigr)$, for all $n\ge 0$.
\end{lemma}

\begin{proof}
Let $I^{\bullet, \bullet}, J^{\bullet, \bullet}$ be Cartan-Eilenberg resolutions of $T^\bullet$ and $Q^\bullet$ respectively, with
$\varphi^{\bullet, \bullet}:T^{\bullet,\bullet}\lra J^{\bullet, \bullet}$ a morphism of double complexes compatible with $\phi^\bullet$ (see \cite{stack_proj}).

For each of the double complexes $\pi_\ast(I^{\bullet,\bullet})$ respectively $\pi_\ast(J^{\bullet, \bullet})$ consider the spectral sequences, with the natural morphisms induced by $\varphi^{\bullet, \bullet}$:

\[
\begin{tikzcd}[column sep=1, row sep=10]
	E_2^{p,q}\ar[d, "\alpha"]&:={\rm H}^p_{\rm vert}\bigl({\rm H}^q_{\rm hor}(\pi_\ast(I^{\bullet,\bullet})\bigr)&\Rightarrow&{\rm H}^{p+q}\bigl({\rm Total}(\pi_\ast(I^{\bullet,\bullet}))\bigr)\ar[d, "\beta"]\\
	E_2^{p,q}&:={\rm H}^p_{\rm vert}\bigl({\rm H}^q_{\rm hor}(\pi_\ast(J^{\bullet,\bullet})\bigr)&\Rightarrow &{\rm H}^{p+q}\bigl({\rm Total}(\pi_\ast(J^{\bullet,\bullet}))\bigr)
\end{tikzcd}
\]
where 
$\alpha:={\rm H}^p_{\rm vert}\bigl({\rm H}^q_{\rm hor}(\pi_\ast(\varphi^{\bullet,\bullet}))\bigr)$ and $\beta:={\rm H}^{p+q}\bigl({\rm Total}(\pi_\ast(\varphi^{\bullet,\bullet}))\bigr)$.

Let $M^\bullet$ be any one of the complexes $T^\bullet, Q^\bullet$ and $K^{\bullet,\bullet}$ its Cartan-Eilenberg resolution as above.
Then we have a natural morphism of complexes $M^\bullet\lra K^{0,\bullet}$ and we consider the morphism of complexes
obtained by applying $\pi_\ast$, i.e. $\pi_\ast(M^{\bullet})\lra \pi_\ast(K^{0,\bullet})$ and the double complex $\pi_\ast(K^{\bullet,\bullet})$.

The spectral sequence associated to this double complex, considered above is:
$$
(\ast)\quad E_2^{p,q}:={\rm H}^p_{\rm vert}\bigl({\rm H}^q_{\rm hor}(\pi_\ast(K^{\bullet,\bullet}))\bigr)\Rightarrow {\rm H}^{p+q}\bigl({\rm Total}(\pi_\ast(K^{\bullet,\bullet}))\bigr).
$$
Let us analyze this spectral sequence. We have that for every $q\ge 0$, the map ${\rm H}^q(M^{\bullet})\lra {\rm H}_{\rm hor}^q(K^{\bullet,\bullet})$ realizes the complex ${\rm H}_{\rm hor}^q(K^{\bullet,\bullet})$ as an injective resolution of ${\rm H}^q(M^{\bullet})$, moreover by the main properties of Cartan-Eilenberg resolutions, we have isomorphisms of complexes 
$$
\pi_{\ast}\bigl({\rm H}_{\rm hor}^q(K^{\bullet,\bullet})\bigr)\cong {\rm H}_{\rm hor}^q\bigl(\pi_\ast(K^{\bullet,\bullet})\bigr).
$$ 
Therefore 
$$
E_2^{p,q}={\rm H}_{\rm vert}^p\bigl({\rm H}_{\rm hor}^q(\pi_\ast(K^{\bullet,\bullet}))\bigr)\cong {\rm H}_{\rm vert}^p\bigl(\pi_\ast({\rm H}_{\rm hor}^q(K^{\bullet,\bullet}))\bigr)\cong R^p\pi_\ast\bigl({\rm H}^q(M^{\bullet})\bigr).
$$
On the other hand, for every $s\ge 0$, the map of complexes: $M^s\lra K^{\bullet,s}$ realizes $K^{\bullet,s}$ as an injective resolution of 
$M^s$. Therefore for every $i>0$, ${\rm H}^i\bigl(\pi_\ast(K^{\bullet,s})\bigr)\cong R^i\pi_\ast(M^s)=0$ by the assumptions of the lemma.
Therefore, thinking of $\pi_\ast(M^s)$ as a complex concentrated in degree $0$, the morphism  $\pi_\ast(M^s)\lra \pi_{\ast}(K^{\bullet,s})$ is a quasi-isomorphism of complexes for every $s\ge 0$. This implies that the natural morphism of complexes
$M^\bullet \lra {\rm Total}(K^{\bullet, \bullet})$ is a quasi-isomorphism, therefore for every $p,q\ge 0$ we have isomorphisms:
${\rm H}^{p+q}\bigl({\rm Total}(\pi_\ast(K^{\bullet,\bullet}))\bigr)\cong {\rm H}^{p+q}(\pi_\ast(M^{\bullet}))$.

We showed that the spectral sequence $(\ast)$ reduces to the spectral sequence:
$$
E_2^{p,q}:=R^p\pi_\ast\bigl({\rm H}^q(M^{\bullet})\bigr)\Rightarrow {\rm H}^{p+q}\bigl(\pi_\ast(M^\bullet))\bigr).
$$
Therefore, returning to our morphism of complexes $\phi^\bullet: T^\bullet\lra Q^\bullet$, the spectral sequences become
\[
\begin{tikzcd}[column sep=1, row sep=12]
	E_2^{p,q}\ar[d, "\alpha'"]&:=R^p\pi_\ast\bigl({\rm H}^q(T^{\bullet})\bigr)&\Rightarrow &{\rm H}^{p+q}\bigl((\pi_\ast(T^{\bullet}))\bigr)\ar[d, "\beta'"]\\
	E_2^{p,q}&:=R^p\pi_\ast\bigl({\rm H}^q(Q^{\bullet})\bigr)&\Rightarrow &{\rm H}^{p+q}\bigl((\pi_\ast(Q^{\bullet}))\bigr)
\end{tikzcd}
\]
where $\alpha'=R^p\pi_\ast({\rm H}^q(\phi^\bullet))=0$, therefore $\beta'={\rm H}^{p+q}(\pi_\ast(\phi^\bullet))=0$, which proves the lemma.

\end{proof}

\section{The $p$-adic case.}

  \subsection{The set-up and outline.}
  \label{sec:outline}

 We start by supposing the $p$-adic case of Assumption \ref{ass:basicassumption} granted. In other words we consider the sequence of algebraic groups $T,Bo,P,L,G$ over a finite extension $K$ of $\Q_p$, and the respective analytic groups $\bT,\bB,\PP, \bG, \cL$, we denote by $\cX$ a quasi-compact log-adic space over ${\rm Spa}(K, \cO_K)$ (not algebrizable, in general), we let $S:={\rm Spa}(B,B^+)\lra {\rm Spa}(K, \cO_K)$ be a connected affinoid of the weight space for this situation, and $k$ a $B$-valued, analytic character of $\bT$ which is generic. We suppose that $B$ is an integral domain and denote by $\fK$ its field of fractions. We base change $\cX$ to $S$ over ${\rm Spa}(K, \cO_K)$ and use the notation  $\cX\lra S$. Next we consider $\fW_k=\bigl({\rm Ind}_{\bB\cap \cL}^{\cL}\bigr)^{\rm an}(k)$, viewed as a $B$-module, with its natural, increasing filtration $(F^s)_{s\ge 0}$ by finitely generated locally free $B$-modules. We denote by $\aW:=\colim_s F^s\subset \fW_k$, with its increasing and exhausting filtration $(F^s)_{s\ge 0}$. We now base change both $\aW$ and $F^s$ for all $s$ to $\fK$ and keep denoting them 
 $\aW, F^s$. We will continue using  the notations of the Assumption. 
 We recall the basic definition in this setting, see Definition \ref{def:fgB}.
 
 \begin{definition} 
 
 By a $(U(\fg), \PP)$-module over a $B$-algebra $R$ we mean an $R$-module with actions of $U(\fg)$ and of $\PP$, which induce the same action of  $\fp$. By a morphism of $(U(\fg), \PP)$-modules we mean a morphism of $U(\fg)$-modules which induces by restriction a morphism of $U(\fp)$-modules coming from a morphism of $\PP$-modules. 
 \end{definition}
 
 If $W$ is a $\PP$-module over some $B$-algebra $R$, then 
 the Verma modules $V:=U(\fg)\otimes_{U(\fp)}W$, with action of $U(\fg)$ on the first factor of the tensor product and diagonal action of $\PP$ is a $(U(\fg), \PP)$-module.
 
 \bigskip  
 
\bigskip
 
In the ``$p$-adic case" we have infinite dimensional modules appearing, therefore the argument is more involved due to topological issues. Thus, for the sake of clarity, we first outline the arguments and then give the proofs in the next section.
   
 We start by denoting $M:=(\aW)^\vee:={\rm Hom}_{\fK}(\aW, \fK)$ and $M^s:=\bigl(F_s\bigr)^\vee$, for $s\ge 0$. Then $(M^s)_s$ is a projective system of $\fK$-vector spaces with $(U(\fp), \PP)$-action and $\displaystyle M\cong \lim_{\leftarrow,s}M^s$. On $M$ we consider the projective limit topology with discrete topology on each $M^s$ and the $(U(\fg), \PP)$-dual action. For example the $\fg$-action is  given by: if $g\in \fg$, $f\in M$, $x\in \aW$, we define $(gf)(x):=-f(gx)$, see \cite{diximier}. As $\aW$ is a Verma module of weight $k$, it follows that $M$ has a lowest weight vector of character $-k$ for the action of the Cartan subalgebra $\fh$ of $\fm$.

Let $\fn^-$ denote the nilpotent Lie sub-algebra of $\fg$, over $\fK$, opposite to $\fp$. For every $n,m,j\in \N$ we have natural maps (all tensor products are over $\fK$): $d\colon U(\fn^-)^{n}\otimes M^{m}\otimes \wedge ^j \fn^-\lra U(\fn^-)^{n+1}\otimes M^{m-1}\wedge^{j-1}\fn^-$ defined by 

$$
a\otimes b\otimes (X_1\wedge X_2\wedge ...\wedge X_j)\to \sum_{i=1}^j (-1)^i(X_ia\otimes b-a\otimes X_ib)\otimes X_1\wedge\ldots\hat{X}_i\wedge...\wedge X_i.
$$ 
For fixed $n,m$ we have complexes:
$$
U(\fn^-)^{n-\bullet}\otimes M^{m+\bullet}\otimes \wedge^\bullet\fn^-:=\ U(\fn^-)^n\otimes M^m\leftarrow U(\fn^-)^{n+1}\otimes M^{m-1}\otimes \fn^-\leftarrow U(\fn^-)^{n+2}\otimes M^{m-2}\otimes \wedge^2\fn^-\leftarrow...
$$

We have, naturally, two ways to assemble the complexes  $\Bigl(U(\frn^-)^{n-\bullet}\otimes_\fK M^{m+\bullet}\otimes_\fK\wedge^\bullet \frn^-\Bigr)_{m,n}$.
\bigskip

\noindent
(I). Let 
$
\cD^\bullet:=\colim_n\lim_m\Bigl(U(\frn^-)^{n-\bullet}\otimes M^{m+\bullet}\otimes \wedge^\bullet \frn^-\Bigr)$ so that $$\cD^\bullet=
U(\frn^-)\otimes M\otimes \wedge^\bullet \frn^-=U(\fg)\otimes _{U(\fp)} M\otimes \wedge^\bullet (\fg/\fp).$$
Looking also at the differentials of $\cD^\bullet$, we see immediately that $\cD^\bullet$ is the Koszul-complex of $M$.
Because $M$ is a $U(\fg)$-module, using the Garland-Lepowsky theorem, see \cite{garland_lepowsky}, we have an isomorphism: $U(\fg)\otimes_{U(\fp)}\bigl(M\otimes_K \wedge^\bullet (\fg/\fp)\bigr)\cong M\otimes_{K)} \bigl(U(\fg)\otimes_{U(\fp)}\wedge^\bullet (\fg/\fp)\bigr)$ and a natural 
$U(\fg)$-equivariant projection $\cD^0=M\otimes_K U(\fg)\lra M$ which gives a quasi-isomorphism of complexes: $\cD^\bullet\cong M^\bullet$, where $M^\bullet:=M\leftarrow 0\leftarrow 0\leftarrow\cdots$. In particular $\cD^\bullet$ is a resolution of $M$ and therefore the homology of $\cD^\bullet$ is given by: ${\rm H}_0(\cD^\bullet)=M$, ${\rm H}_i(\cD^\bullet)=0$ for $i>0$.

\bigskip
\noindent
(II) The second way to assemble the family of complexes $\Bigl(U(\frn^-)^{n-\bullet}\otimes_\fK M^{m+\bullet}\otimes_\fK\wedge^\bullet \frn^-\Bigr)_{m,n}$ is:
$$
\cC^\bullet:=\lim_{m}\colim_{n}\Bigl(U(\frn^-)^{n-\bullet}\otimes M^{m+\bullet}\otimes_K \wedge^\bullet \frn^-\Bigr)\cong
\lim_{m}\Bigl(U(\frn^-)\otimes M^{m+\bullet}\otimes \wedge^\bullet \frn^-\Bigr)=
$$
$$=\lim_{m}\Bigl(U(\fg)\otimes_{U(\fp)} M^{m+\bullet}\otimes \wedge^\bullet (\fg/\fp)\Bigr) $$
We have a natural morphism of complexes $\xi\colon\cD^\bullet \lra \cC^\bullet$ of $U(\fg)$-modules, which is not, in general, a quasi-isomorphism.

 \bigskip
\noindent
As mentioned before, $M$ has a lowest weight vector for the character $-k$ of the Cartan sub-algebra $\fh$ of $\fg$, and we look at the action of $Z(\fg):=Z(U(\fg))$, the center of $U(\fg)$, on the $U(\fg)$-modules $A_{s,j}:=
U(\fg)\otimes_{U(\fp)}\bigl(M^s\otimes \wedge^j(\fg/\fp)\bigr)$, for $s$, $j\ge 0$.

As pointed out in the Introduction (\S \ref{sec:intro}), if we denote by $\chi_k$ the character of $Z(\fg)$ giving the action of $Z(\fg)$ on the Verma module $U(\fg)\otimes_{U(\fp)}\fK(k)$, the action of $Z(\fg)$ on $M$ is given by a $\fK$-valued character, which we denote $\chi^\vee$, of $Z(\fg)$ (see the section \ref{sec:details}).  This means that for every $\tau\in Z(\fg)$ and $f\in M$ we have $\tau f=\chi^\vee(\tau)f$. We denote by
$\cV^\bullet:=\bigl(\cC^\bullet\bigr)_{\chi^\vee}\subset \cC^\bullet$ the generalized eigenspace in $\cC^\bullet$ of character $\chi^\vee$ for the action of $Z(\fg)$, and let $\cU^\bullet:=\cC^\bullet/\cV^\bullet$. 
We notice that for every $j\ge 0$ there is an integer $m_j$ such that we have $(\cC^j_m)_{\chi^\vee}\cong (\cC^j_{m+1})_{\chi^\vee}=:\cV^j_{m+1}$, for all $m\ge m_j$. Here we denoted $\cC^j_m:= U(\fg)\otimes_{U(\fp)}
(M^{m+j}\otimes \wedge^j \fg/\fp)$. This implies that $\cV^\bullet\subset \cC^\bullet$ is a direct summand, as complexes.

We have therefore the following commutative diagram $(\ast)$ of complexes, with exact, split rows (which defines $\gamma$):
\[
\begin{tikzcd}[row sep=small]
	0 \ar[r]&\cV^\bullet\ar[r]&\cC^\bullet\ar[r]&\cU^\bullet\ar[r]&0\\
				&						&\cD^\bullet\ar[u, "\xi"]\ar[r, equal]&\cD^\bullet\ar[u, "\gamma"]&
\end{tikzcd}
\]
1)  We will prove the following, see the next section:

\begin{theorem}
\label{thm:zerocoh}
The morphism of complexes $\gamma$ induces $0$ on homology, i.e. for all $i\ge 0$, ${\rm H}_i(\gamma):{\rm H}_i(\cD^\bullet)\lra
{\rm H}_i(\cU^\bullet)$ and we have ${\rm H}_i(\gamma)=0$.
\end{theorem}

\medskip

2) Now we continuously dualize the diagram $(\ast)$ (we use continuous duals for all modules, denoted $V\to V^\ast$, see the section \ref{SubsectionContinuousDual} for details) and the next step is to sheafify the complexes of the diagram:

 $$
\begin{array}{ccccccccc}
	0 \lra&(\cU^\bullet)^\ast&\lra&(\cC^\bullet)^\ast&\lra&(\cU^\bullet)^\ast&\lra&0\\
	&\downarrow\gamma^\ast&&\downarrow\xi^\ast\\
	&(\cD^\bullet)^\ast&=&(\cD^\bullet)^\ast
\end{array}
$$

 We recall the open affinoid cover $(U_i:={\rm Spa}(A_i, A_i^+))_{i\in I}$, in Assumption \ref{ass:basicassumption}, d), with its properties i), ii), iii) described there. We recall that for every $i\in I$ there is an abelian Lie-sub-algebra of  $\fg\otimes_BA_i$,  $\fn_i$, such that the composition $\fn_i\lra \fg\otimes_BA_i\lra \fg/\fp\otimes_BA_i$ is an isomorphism of Lie algebras over $A_i\otimes_B\fK$, for all $i\in I$.

\smallskip
a) We have an isomorphism as $A_i\otimes_B\fK$-modules $\bigl( \wedge^s \fn_i\bigr)^\vee\otimes_BA_i\cong \omega^s_{U_i/S}\otimes_B\fK$, for $s\ge 0$.
\smallskip

and
\smallskip

b) If we denote $U(\fn_i)^s$ the natural filtration of $U(\fn_i)$, then $\bigl( U(\fn_i)^s\bigr)^\ast=U(\fn_i)^\vee\cong \cP_{U_i/S}^s$, see Lemma \ref{lemma:logsmooth} and Remark \ref{rmk:coordinates} for the notations. 
We recall that $\cP^s_{U_i/S}$ can be seen as the log $s+1$-infinitesimal neighbourhood of the diagonal $\displaystyle U_i\stackrel{\Delta}{\lra}(U_i\times_S U_i)^{\rm ex}$.
\smallskip

Therefore we have for every $i\in I$, with notations as above (see section \ref{sec:linedelin} for the definition of linearization):
$$
\bigl(\cD^\bullet\bigr)^\ast_{U_i} \cong \lim_{n}\colim_{m}\bigl(\cP^{n-\bullet}_{U_i/S}\otimes F_{m+\bullet}\otimes \omega_{U_i/S}^\bullet\bigr)\cong
$$
$$
\cong \lim_{n}\bigl(\cP^{n-\bullet}_{U_i/S}\otimes\fW_k^{\rm alg}\otimes \omega_{U_i/S}^\bullet\bigr)=\cP_{U_i/S}\otimes\aW_k\otimes\omega_{U_i/S}^\bullet = L\bigl(\fW_k^{\rm alg}\otimes \omega_{U_i/S}^\bullet\bigr).
$$
We also have:
$$
\bigl(\cC^\bullet\bigr)_{U_i}^\ast\cong \Bigl(\lim_{m}\colim_{n} U(\fn_i)_{A_i}^{n-\bullet}\otimes _{A_i}M_{A_i}^{m+\bullet}\otimes_{A_i} \wedge^\bullet (\fn_i\otimes_BA_i)\Bigr)^\ast\cong \colim_{m}\lim_{n}\bigl(\cP^{n-\bullet}_{U_i/S}\otimes F_{m+\bullet}\otimes \omega_{U_i/S}^\bullet\bigr)\cong
$$
$$
\cong  \colim_{m}L(F_{m+\bullet}\otimes \omega_{U_i/S}^\bullet).
$$

\medskip

3) For every $i\in I$ we have the diagram $(\ast\ast)$ dual to $(\ast)$, with exact and split rows:
\[
\begin{tikzcd}[row sep=14]
	0\ar[r]& (\cU^\bullet)_{U_i}^\ast\ar[r]\ar[d,"\gamma^\ast" ]& (\cC)^\bullet)_{U_i}^\ast\ar[r]\ar[d, "\xi^\ast"]&(\cV^\bullet)_{U_i}^\ast\ar[r]&0\\
	&L(\aW\otimes \omega_{U_i/S}^\bullet)\ar[r, equal]&L(\aW\otimes \omega_{U_i/S}^\bullet)&&
\end{tikzcd}
\]

Where $(\cV^{\bullet})^\ast_{U_i}:=(\cV^{\bullet})^\ast\otimes_B A_i$ and similarly for $(\cU_{U_i}^\bullet)^\ast$. We claim that $\gamma^\ast$ induces $0$ in cohomology, i.e. the map ${\rm H}^i(\gamma^\ast):{\rm H}^i\bigl((\cU_{U_i}^\bullet)^\ast\bigr)\lra {\rm H}^i(L(\aW\otimes \omega_{U_i/S}^\bullet)))$ is the zero map. See section \ref{SubsectionContinuousDual} for the proofs.

\medskip

Of course, the family of complexes $\Bigl(\bigl(\cC^\bullet\bigr)_{U_i}^\ast\Bigr)_{i\in I}=\colim_mL\bigl(\cF_{m+\bullet}\otimes\omega^\bullet_{\cX/S}\bigr)|_{U_i}$ for $i\in I$, glue to give the complex of $\cP_{\cX/S}$-modules
$R^\bullet:=\colim_mL\bigl(\cF_{m+\bullet}\otimes\omega^\bullet_{\cX/S}\bigr)$, and we have

\begin{proposition}
\label{prop:gluing}
a) The complexes $\Bigl(\bigl(\cU^\bullet\bigr)_{U_i}^\ast\Bigr)_{i\in I}$ glue and the inclusions $$u_i\colon (\cU^\bullet)_{U_i}^\ast\hookrightarrow \colim_mL\bigl(\cF_{m+\bullet}\otimes\omega^\bullet_{\cX/S}\bigr)|_{U_i},$$ 
for $i\in I$, glue to give a complex of $\cP_{\cX/S}$-modules on $\cX$, with an $\cO_{\cX}$-linear inclusion $u\colon L(Q^\bullet)\hookrightarrow  R^\bullet$.

b) The complexes $\Bigl(\bigl(\cV^\bullet\bigr)_{U_i}^\ast\Bigr)_{i\in I}$ and the projections $\colim_mL\bigl(\cF_{m+\bullet}\otimes\omega^\bullet_{\cX/S}\bigr)|_{U_i}\to \bigl(\cV_{U_i}^\bullet\bigr)^\ast$, for $i\in I$, glue to give a
complex of $\cP_{\cX/S}$-modules on $\cX$, with $\cO_{\cX}$-linear projection $ R^\bullet\to L(E^\bullet)$ such that the sequence of complexes
$$
0\lra L(Q^\bullet)\lra   R^\bullet\lra L(E^\bullet)\lra 0
$$
is exact.
\end{proposition}

The proof will be given in the next section.

4) Now we consider the diagram $(\ast\ast\ast)$ of complexes of sheaves of $\cO_{\cX}$-modules obtained by gluing the complexes discussed above:
\[
\begin{tikzcd}[row sep=14]
	0\ar[r]& L(Q^\bullet)\ar[r]\ar[d,"\eta=\gamma^\ast" ]& R^\bullet \ar[r]\ar[d, "\rho=\xi^\ast"]&L(E^\bullet)\ar[r]&0\\
	&L\bigl(\WW_k^{\rm alg}\otimes_{\cO_{\cX}} \omega_{\cX/S}^\bullet\bigr)\ar[r, equal]&L\bigl(\WW_k^{\rm alg}\otimes_{\cO_{\cX}} \omega_{\cX/S}^\bullet\bigr)&&
\end{tikzcd}
\]

\medskip

5) De-linearization (see section \ref{sec:linedelin}).

We consider the log infinitesimal site attached to the log adic space $\cX$ over $S$ denoted $(\cX/S)_{\rm inf}^{\rm log}$. Then because all the sheaves in diagram 
5) are linearized or co-limits of linearized $\cO_{\cX}$-modules, we can see that diagram as a diagram of complexes of log crystals on
$(\cX/S)_{\rm inf}^{\rm log}$, denoted in the same way.

We recall from section  \ref{sec:linedelin} that if we denote by $\cX_{\rm inf}^{\rm log}$ respectively $\cX_{\rm an}$ the log infinitesimal topos, respectively the analytic topos of $\cX$, i.e. the category of sheaves of sets on the site $(\cX/S)_{\rm inf}^{\rm log}$, respectively the category of sheaves of sets on $\cX$, seen as a log adic space, then we have a morphism of sites $\pi_\ast: \cX_{\rm inf}^{\rm log}\lra \cX_{\rm an}$ with the properties listed there.

We claim the following:

a) the morphism 
$$
\pi_\ast(\rho)\colon \pi_\ast (R^\bullet)\lra \pi_{\ast}\bigl(L(\WW_k^{\rm alg}\otimes \omega_{\cX/S}^\bullet)\bigr)\cong \WW_k^{\rm alg}\otimes \omega_{\cX/S}^\bullet
$$
is an isomorphism. This follows from the fact that $\displaystyle R^\bullet \cong \colim_{ m}L(F_{m+\bullet}\otimes_B\omega^\bullet_{\cX/S})$ therefore we have 
$$\pi_\ast(R^\bullet)=\colim_m\pi_\ast L\bigl(\cF_{m+\bullet}\otimes \omega^\bullet_{\cX/S}\bigr)=
\WW_k^{\rm alg}\otimes\omega_{\cX/S}^\bullet.$$ See lemma \ref{lemma:uastcolim}.

b) ${\rm H}^i(\pi_\ast(\eta))=0$ for all $i\ge 0$. To see this we recall that on the one hand every term of the diagram $(\ast\ast\ast)$  is a linearized module or a co-limit of linearized ones, $R^a\pi_\ast$ vanishes on any of them, if $a>0$. Therefore we have spectral sequences with morphisms between them (see next section for proofs):
\[
\begin{tikzcd}[row sep=small, column sep=1]
	R^a\pi_\ast\bigl({\rm H}^b(R^\bullet)\bigr)&\Rightarrow&{\rm H}^{a+b}\bigl(\pi_\ast (R^\bullet) \bigr)\\
	R^a\pi_\ast\bigl({\rm H}^b(L(\WW^{\rm alg}\otimes \Omega^\bullet_{\cX} )  \bigr)\ar[u, "\sigma_{a,b}"]&\Rightarrow& {\rm H}^{a+b}(\WW_k^{\rm alg}\otimes \Omega_{\cX}^\bullet)\ar[u, "\tau_{a+b}"]
\end{tikzcd}
\]
where $\sigma_{a,b}:=R^a\pi_\ast\bigl({\rm H}^b(\eta)  \bigr)=0$ for all $a,b\ge 0$, therefore $\tau_{a+b}:={\rm H}^{a+b}(\pi_\ast(\eta)\bigr)=0$ for all $a,b\ge 0$. See section \ref{sec:spectralseq} for details.

c) Finally, from a) and b) above we get an exact sequence of complexes (split):
$$
0\lra Q^\bullet\stackrel{\pi_\ast(\eta)}{\lra}\WW_k^{\rm alg}\otimes \omega_{\cX/S}^\bullet\stackrel{\mu}{\lra} E^\bullet\lra 0
$$
such that ${\rm H}^a(\pi_\ast(\eta))=0$ for all $a\ge 0$. Therefore $\mu:\WW_k^{\rm alg}\otimes \omega_{\cX/S}^\bullet \lra E^\bullet$ 
is a quasi-isomorphism. We remark that $E^\bullet$ is the BGG-complex associated to 
the de Rham complex $\WW_k^{\rm alg}\otimes \omega_{\cX/S}^\bullet$, denoted $E^\bullet:=BGG^\bullet$ in the comment after Assumption \ref{ass:basicassumption}, and the argument above shows that the two complexes are quasi-isomorphic.
This proves the main claim of this article: granting Assumption \ref{ass:basicassumption}, we produce a complex of coherent $\cO_{\cX}$-modules, $E^\bullet=BGG^\bullet$, with a quasi-isomorphism of complexes  $\mu:\WW_k^{\rm alg}\otimes \omega_{\cX/S}^\bullet \lra E^\bullet$.

\bigskip

\bigskip

\subsection{The details of the above outline.}
\label{sec:details}

\subsubsection{The proof of theorem \ref{thm:zerocoh}.}

We work in the notations of section \ref{sec:outline}.
We first prove theorem \ref{thm:zerocoh} and to do so we need to understand the structures of the complexes $\cV^\bullet, \cU^\bullet$, respectively.

Let us recall that we denoted $\chi_k$ the character of $Z(\fg)$ giving the action of $Z(\fg)$ on the highest weight module $\aW$.

\begin{lemma}
\label{lemma:character}
 $Z(\fg)$ acts on $M:=(\aW)^\ast$ by a character, which we called $\chi^\vee$, i.e. for every $\tau\in Z(\fg)$ and $f\in M$ we have
 $\tau f=\chi^\vee(\tau)f$. If $\fg=\frak{sl}_2$, then $\chi^\vee=\chi_k$.

\end{lemma}

\begin{proof}
Let $v_+$ denote a highest weight vector of $\aW$, i.e. $0\neq v_+\in \aW$ and for all $H\in \fh$ we have $H v_+=k(H)v_+$. Let 
$\{v_+, v_i\}_{i}$ be a basis of $\aW$ of eigenvectors for $\fh$ with $H v_i=(k(H)-a_i(H))v_i$, where $a_i(H)\in \N\backslash \{0\}$ for all $i$.

Let $\{(v_+)^\vee, v_i^\vee\}_i$ be the dual basis of $M$. Then for every $H\in \fh$ we have $H(v_+)^\vee=-k(H)(v_+)^\vee$ and
$H(v_i)^\vee=(a_i(H)-k(H))v_i^\vee$ for all $i$.
Therefore if $\tau\in Z(\fg)$, $H\in \fh$ we have $H\bigl(\tau(v_+)^\vee\bigr)=\tau\bigl(H(v_+)^\vee)\bigr)=-k(H)\tau(v_+)^\vee$.
If $0\neq f\in M$ is such that $Hf=-k(H)f$ then we claim: $f(v_i)=0$ for all $i$. If not, there is $j$ such that $f(v_j)\neq 0$ and  we have for $0\neq H\in \fh$: 
$$
-k(H)f(v_j)=(Hf)(v_j)=-f(Hv_j)=\bigl(-k(H)+a_j(H)\bigr)f(v_j)
$$ 
so $a_j(H)=0$, which contradicts the above assumption. Therefore $a:=f(v_+)\neq 0$ and so $f=a(v_+)^\vee$ with $0\neq a\in \fK$.
 All in all we deduce $\tau (v_+)^\vee=\chi^\vee(\tau)(v_+)^\vee$, with $\chi^\vee:Z(\fg)\lra \fK$ a character, i.e. a $K$-algebra homomorphism.
 
 If $\fg=\frak{sl}_2$, $Z(\fg)$ is the polynomial algebra over $K$ in the variable $C$, where $C:=H^2+u^-u^++u^+u^-$ is the Casimir operator. Therefore if $0\neq f\in M$ and $x\in \aW$ we have $\chi^\vee(C)f(x)=(Cf)(x)=f(Cx)=\chi_k(C)f(x).$ Therefore $\chi^\vee(C)=\chi_k(C)$ and so $\chi^\vee=\chi_k$.
\end{proof}

The first question is, given a character $\chi$ of $Z(\fg)$, how do we describe the $U(\fg)$-modules $(A_{s,j})_{\chi}$, for varying $s,j$? These modules were defined at (II) by $A_{s,j}:=U(\fp)\otimes_{U(\fp)} M^s\otimes_K \wedge^j(\fg/\fp)$.
\bigskip

{\bf We claim:} {\it the natural projections $\bigl(A_{t,m}\bigr)_{\chi}\lra \bigl(A_{s,m}\bigr)_{\chi}$ are isomorphisms for $t>s>m$,  if $t-m$ and $s-m$ are large enough.  We denote 
$(A_m)_{\chi}:=\bigl(A_{t,m}\bigr)_{\chi}$ for $t$ large enough. Moreover, with notations as above, we have
\smallskip

a) $\cF_m=(A_m)_{\chi^\vee}$, for all $0\le m$.
\smallskip

b) $\displaystyle \cG_m=\lim_{\Delta}\Bigl(\oplus_{\chi\in \Delta}(A_m)_{\chi}\Bigr)$, where $\Theta$ is the set of 
characters of $Z(\fg)$, distinct from $\chi^\vee$, such that $(A_m)_{\chi}\neq 0$, and $\Delta\subset \Theta$ runs over the finite subsets. }
\smallskip

We will first prove the theorem granted the claim, and will prove the claim at the end.

\bigskip

We recall that if $\chi$ is any character of $Z(\fg)$, we denote by $(A_{s,m})_{\chi}:=\{x\in A_{s,m}\vert \mbox{ for all}\ \tau\in Z(\fg), (\tau-\chi(\tau))^nx=0, \mbox{for some}\ n\in \N\}$.
Then $(A_{s,m})_{\chi}$ is a direct summand of $A_{s,m}$ as $U(\fg)$-module.

Let now $\varphi$ be any character of $Z(\fg)$, then if denote $I_\varphi:={\rm Ker}(\varphi\colon Z(\fg)\lra \fK)$, we see that $I_\varphi$ is a maximal ideal of $Z(\fg)$. One the other hand, as $Z(\fg)\cong \bigl({\rm Sym}(\fh)\bigr)^{W_G}$, if we denote by $\alpha\colon\Spec({\rm Sym}(\fh))=\bA^m_{\fK}\lra \bA^m_{\fK}/W_G=\Spec(Z(\fg))$, as $\alpha$ is a finite morphism, we have $\alpha^{-1}(I_\varphi)=\{I_1, I_2,\ldots,I_s\}$, where
$I_j\subset {\rm Sym}(\fh)\cong \fK[X_1,\ldots,X_m]$ is a maximal ideal, for all $1\le j\le s$.

Let us denote $\displaystyle \Theta':=\cup_{\varphi\in \Theta}\pi^{-1}(I_\varphi)\subset {\rm Max}({\rm Sym}(\fh^-))$. Let $\alpha^{-1}(I_{\chi^\vee}):=\{ \fm_1, \fm_2,\ldots,\fm_u\}\subset {\rm Max}({\rm Sym}(\fh^-))$.

\begin{lemma}
\label{lemma:polyn}
To the expense of base changing $\alpha$ to a field extension $L$ of $\fK$, there exists an $f\in \fm_1$ such that $f\notin \fm$, for all
$\fm\in \Theta'$.
\end{lemma}

\begin{proof}
We have, for some field extension $L/\fK$, ${\rm Sym}_L(\fh^-)\cong L[X_1,\ldots,X_m]$. We may suppose WLOG that $\fK$ is algebraically closed, i.e.
$\fK=\overline{\fK}\subset L$ and so $\fm_1=(X_1,\ldots,X_m)$ and for $\fm\in \Theta'$ we have $\fm=(X_1-a_{\fm,1},\ldots, X_m-a_{\fm,m})$ with
$(a_{\fm,1},\ldots,a_{\fm,m})\in (\fK)^m\backslash \{(0,\ldots,0)\}$. 

We choose now $L:=\fK(Y_1,\ldots,Y_m)$ with $Y_1,\ldots,Y_m$ independent transcendentals over $\fK$. Then if $f\in \fm_1$, $f=Y_1X_1+Y_2X_2+\ldots+Y_mX_m$, for all $\fm\in \Theta'$ as above, $f(a_{\fm,1},\ldots,a_{\fm,m})=a_{\fm,1}Y_1+\ldots+a_{\fm,m}Y_m\neq 0$
for all $(a_{\fm, 1},\ldots,a_{\fm,m})\in \fK^n\backslash \{(0,\ldots,0)\}$.
\end{proof}

\begin{lemma}
\label{lemma:tau}
To the expense of base changing $\alpha$ to a field extension $L/\fK$, there is $\tau\in I_{\chi_k}$ such that $\tau\notin I_\varphi$ for all
$\varphi\in \Theta$.
\end{lemma}

\begin{proof}
Let $L$ be a field extension of $\fK$ and $f\in \fm_1$ such that $f\notin\fm$ for all $\fm\in \Theta'$, as in lemma \ref{lemma:polyn}. Let then $\tau:=\prod_{\sigma\in W_G}\sigma(f)\in {\rm Sym}(\fh)^{W_G}=Z(\fg)$. Then, for all $\varphi\in \Theta$, and $\sigma\in W_G$, we have
$\sigma(f)\notin \fm$ for all $\fm\in \pi^{-1}(I_\varphi)$, by lemma \ref{lemma:polyn}. Therefore, $\tau=\prod_{\sigma\in W_G}\sigma(f)\notin \fm$, for all $\fm\in \Theta'$.
\end{proof}

\begin{corollary}
\label{cor:tau}
Let $\tau\in I_{\chi_k}$ be as in lemma \ref{lemma:tau}, then $\tau$ acts invertibly on $\cU_i$, for all $i\ge 0$.
\end{corollary}

\begin{proof}
We recall that $\displaystyle \cU_i=\lim_{\Delta}\Bigl(\oplus_{\chi\in \Delta}(A_i)_{\chi}\Bigr)$, where $\Delta\subset \Theta$ are finite subsets. For every $\varphi\in \Delta\subset \Theta$, we have $I_\varphi^s$ annihilates $(A_i)_{\varphi}$, for some $s$.

We have: $\tau=(\tau-\varphi(\tau))+\varphi(\tau)\in I_{\varphi}+L^\ast$. As $\tau-\varphi(\tau)$ acts nilpotentely on 
$\Bigl((A_i)_L\bigr)_\varphi$ and $\varphi(\tau)\neq 0$, by lemma \ref{lemma:tau}, $\tau$ acts invertibly on $\bigl((A_i)_L\bigr)_{\varphi}$, for all $\varphi\in \Theta$.  
\end{proof}

\bigskip
\noindent
{\it Proof of theorem \ref{thm:zerocoh}.}

Let $L$ be a field extension of $\fK$ and $\tau\in Z(\fg)_L$ as in corollary  \ref{cor:tau}. We recall that $\gamma:\cD^\bullet\lra \cU^\bullet$ is a morphism of complexes, and that the cohomology of $\cD^\bullet$ is known. Moreover for every $i\ge 0$, and every $x\in {\rm H}^i(\cD^\bullet)$, there is $n_x\ge 0$ such that $\bigl(I_{\chi_k}\bigr)^{n_x}x=0$. Let us make the base change from $\fK$ to $L$ as in corollary \ref{cor:tau} and let $\tau\in I_{\chi_k}$ of that corollary. For every $i\ge 0$ and every $x\in {\rm H}^i(\cD^\bullet)$ we have: $0={\rm H}^i(\gamma)(\tau^{n_x}x)=
\tau^{n_x}{\rm H}^i(\gamma)(x)$. But by corollary \ref{cor:tau}, $\tau^{n_x}$ acts invertibly on $\cU^\bullet_L$ and so on ${\rm H}^i(\cU^\bullet)_L$, which implies ${\rm H}^i(\gamma_L)(x)=0$ for all $x\in {\rm H}^i(\cD^\bullet)$. Therefore ${\rm H}^i(\gamma)=0$ for all $i\ge 0$.

\bigskip
We next prove the claim.

\bigskip
{\bf Proof of the claim.}

We consider the $P$-module $F=M^s$, for every $s\in \N$. Recall that we are in the case of a $p$-adic weight. 
Then, thanks to Lemma \ref{lemma:diximier} and  Corollary \ref{cor:glueinfBGG}, if $\nu$ is a weight with character $\chi_{\nu}$ of $Z(\fg)$,  then $\bigl(U(\fg)\otimes_{U(\fp)}F\bigr)_{\chi_{\nu}}$ has a filtration by $U(\fg)$-sub-modules with graded quotients Verma modules $M(\mu_i)$ where $\nu=w\bullet \mu_i$ for some $w\in W_G$, the Weyl group of $G$.  Let $\Theta:=\Theta_{U(\fg)\otimes_{U(\fp)}F}$ be the finite set of characters $\chi$ of $Z(\fg)$ such that $\bigl(U(\fg)\otimes_{U(\fp)}F\bigr)_{\chi}\neq 0.$ Then $U(\fg)\otimes_{U(\fp)}F\cong \oplus_{\chi\in \Theta}\bigl(U(\fg)\otimes_{U(\fp)}F\bigr)_{\chi}$.

In particular, coming back to our modules $A_{s,m}$, for $m$ fixed, let us remark that for a character $\chi$ of $Z(\fg)$, the natural projections $\bigl(A_{t,m}\bigr)_{\chi}\lra \bigl(A_{s,m}\bigr)_{\chi}$ are isomorphisms for $t>s>m$ and $t-m$ and $s-m$ large enough.  We denote 
$(A_m)_{\chi}:=\bigl(A_{t,m}\bigr)_{\chi}$ for $t$ large enough. Coming back to the problem of describing the structures of our complexes $\cF^\bullet, \cG^\bullet$, we remark that:

a) $\cF_m=(A_m)_{\chi_k}$, for all $0\le m\le d$.

b) $\displaystyle \cG_m=\lim_{\leftarrow, \Delta}\Bigl(\oplus_{\chi\in \Delta}(A_m)_{\chi}\Bigr)$, where $\Theta$ is the set of 
characters of $Z(\fg)$, distinct from $\chi_k$, such that $(A_m)_{\chi}\neq 0$, and $\Delta\subset \Theta$ runs over the finite subsets. 

\bigskip

Finally, Proposition \ref{prop:gluing} follows from Corollary \ref{cor:glueinfBGG} and Lemma \ref{lemma:centrocommuta}.

\subsubsection{The continuous dual picture}\label{SubsectionContinuousDual}

In this section we continue to use the notations of section \ref{sec:dual}. Given a topological $\fK$-vector space $T$, we denote by $T^\ast$ its continuous $\fK$-dual of $T$, i.e. $T^\ast:={\rm Hom}_{\fK, cont}(T,\fK)\subset T\dual={\rm Hom}_{\fK}(T,\fK)$, where on $\fK$ we set the discrete topology. Of course, if $T$ is a finite dimansional $\fK$-vector space $T^\ast=T^\vee$.

 The main reason we consider these continuous duals is the following: suppose
 $T$ is a $\fK$-vector space with a separated and exhausting filtration $\{T_i\}_{i\in \N}$ by finite dimensional $\fK$-vector spaces $T_i\subset T$, then we have:
$T=\colim_n T_n$, $T^\ast \cong T^\vee\cong \lim_n T_n^\vee$ and finally $\bigl(T^\ast\bigr)^\ast\cong \Bigl(\lim_n T_n^\vee\Bigr)^\ast\cong \colim_n \bigl(T_n^\vee\bigr)^\ast\cong \colim_n T_n=T$. 

Coming back to our problem, for any $s, n\in \N$ we have defined the complexes
\[
B_{n,s}^\bullet:=U(\fg)^{ n-\bullet}\tensor_{U(\fp)}\bigl(M^{s+\bullet}\tensor_S \wedge_{\fK}^\bullet (\fg/\fp)\bigr).
\]
By definition and by the previous results we have that 
\[\cD^\bullet=\colim_{n\in \N}\lim_{s\in \N} B_{n,s}^\bullet\quad\text{and}\quad \cG^\bullet=\lim_{s\in \N} \oplus_{\chi\in \Theta} \left(\colim_{n\in \N} {\fK}_{n,s}^\bullet\right)_{\chi}.
\]
In the previous section we proved that the map between the algebraic duals of the two complexes
$
\cG^{\bullet,\vee}\rightarrow \cD^{\bullet,\vee}
$ induces the zero map between the cohomology groups. We will be interested in understanding the natural maps between the continuous duals of the two complexes. More precisely the topologies we use are the limit and co-limit topologies, with discrete topology on each term and the discrete topology on K, respectively $B$. From the comment at the beginning of this section it follows that
$$
(\cG^\bullet)^\ast=\Bigl(\lim_{s\in \N} \left(\oplus_{\chi\in \Theta} \left(\colim_{n\in \N} B_{n,s}^\bullet\right)_{\chi}\right)\Bigr)^\ast\cong \colim_{s\in \N} \left(\oplus_{\chi\in \Theta} \left(\colim_{n\in \N} B_{n,s}^\bullet\right)_{\chi}\right)\dual
$$
and 
$$
(\cD^\bullet)^\ast=\Bigl(\colim_{n\in \N}\lim_{s\in \N} B_{n,s}^\bullet\Bigr)^\ast\cong \lim_{n\in \N}\colim_{s\in \N} \left((B_{n,s}^\bullet)\dual\right).
$$
All in all, we look at the maps
\[
(\gamma^{\bullet})^\ast:\colim_{s\in \N} \left(\oplus_{\chi\in \Theta} \left(\colim_{n\in \N} B_{n,s}^\bullet\right)_{\chi}\right)\dual\longrightarrow  \lim_{n\in \N}\colim_{s\in \N} \left((B_{n,s}^\bullet)\dual\right).
\]
Observe that there is a diagram of complexes
\small
\begin{equation}\label{DiagramContinuousDualTheorem}
\begin{tikzcd}
	\ar[d]\colim_{s\in \N} \left(\oplus_{\chi\in \Theta} \left(\colim_{n\in \N} B_{n,s}^\bullet\right)_{\chi}\right)\dual\ar[r, "	(\gamma^{\bullet})^\ast"]&\lim_{n\in \N}\colim_{s\in \N} \left((B_{n,s}^{\bullet})^\vee\right)\ar[d, "c^\bullet"]\\
	\cG^{\bullet,\vee}=\left(\lim_{s\in \N} \oplus_{\chi\in \Theta} \left(\colim_{n\in \N} B_{n,s}^\bullet\right)_{\chi}\right)\dual\ar[r, "\gamma^{\bullet,\vee}"]& \cD^{\bullet,\vee}=\lim_{n\in \N}\left(\colim_{s\in \N} B_{n,s}^\bullet\right)\dual
\end{tikzcd}
\end{equation}
\normalsize
In this section we prove that the maps induced on the cohomology groups by $(\gamma^{\bullet})^\ast$ vanish. We need only to prove that the canonical map $c^\bullet$ induces injective morphisms between the cohomology groups; indeed the vanishing of ${\rm H}^k\left((\gamma^{\bullet})^\ast\right)$ will follow by the diagram above and the vanishing  in cohomology of $\gamma^{\bullet,\vee}$ arguing as in the proof of  Proposition \ref{PropositionVanishingDualCohomology}.

\begin{lemma}\label{LemmaInjectivityOnsCohomology}
	Let $\{A_s^\bullet,\ d_s^\bullet\}_{s\in \N}$ be a projective system of complexes of $\fK$-vector spaces, i.e. $A_s^\bullet$ is a complex of $\fK$-vector spaces with differentials $d_s^\bullet$ and for any $s\in \N$ there is a morphism of complexes $f_{s+1}^\bullet: A_{s+1}^\bullet\rightarrow A_s^\bullet$.
	
	Suppose that for any $k\in \Z$, $s\in \N$ the map $f_{s+1}^k: A_{s+1}^k\twoheadrightarrow A_s^k$ is surjective and that there is a section $g_s^k: A_s^k\hookrightarrow \lim_{s\in \N} A_s^k$ of the projection map $\lim_{s\in \N} A_s^k\rightarrow A_s^k$. Suppose that the sections $g_s^k$ are compatible with the morphisms $d_s^k$, i.e. that the diagrams
	\[
		\begin{tikzcd}[column sep= large]
			A^k_{s}\ar[r,"d^k_s"]\ar[d, "g_s^k", hookrightarrow]& A^{k+1}_{s}\ar[d, "g_s^{k+1}", hookrightarrow]\\
			\lim_{s\in \N} A_s^k\ar[r, "\lim_{s\in \N} d^k_s"]& 	\lim_{s\in \N} A_s^{k+1}
		\end{tikzcd}
	\]
	are commutative.
	Then the map of complexes
	$
		\colim_{s\in \N} (A_s^\bullet)\dual\rightarrow \left( \lim_{s\in \N} A_s\right)\dual
	$
	induces a morphism in cohomology 
	\[
		{\rm H}^k\left( 	\colim_{s\in \N} (A_s^\bullet)\dual\right)\hookrightarrow {\rm H}^k\left(\left( \lim_{s\in \N} A_s^\bullet\right)\dual\right)
	\]that is injective for any $k\in \Z$.
\end{lemma}
\begin{proof}
	Let $k\in \Z$, $s\in \N$ and $\varphi\in \Ker\left(\colim_{s\in \N} d_s^{k-1,\vee} \right)$. Then there is an $s_*\in \N$ such that $\varphi\in \Ker\left(d_{s_*}^{k-1,\vee} \right)$. Suppose that the image of $\varphi$ is zero in cohomology, i.e. that there is a $\delta\in \left(\lim_{s\in \N} A_s^{k+1}\right)^\vee$ such that the diagram
	\[
		\begin{tikzcd}
			\lim_{s\in \N} A_s^k \ar[d, "p_{s_*}^k"]\ar[r, "\lim_s d_s^k"]& \lim_{s\in \N} A_s^{k+1}\ar[rd, "\delta"]\ar[d, "p_{s_*}^{k+1}"]&\\
			A_{s_*}^k\ar[r, "d_{s_*}^k"]\ar[rr,"\varphi",bend right=20,swap]\ar[u,"g_{s_*}^k",bend left=50]
			&A_{s_*}^{k+1} \ar[u,"g_{s_*}^{k+1}",bend left=50]&S\\
		\end{tikzcd}
	\]
	 commutes. Observe that
	 \[
	 	\varphi= \delta\circ \left(\lim_{s\in \N} d_s^k\right) \circ g_{s_*}^k= \left( \delta\circ g_{s_*}^{k+1}\right)\circ d_{s_*}^k= d_{s_*}^{k,\vee}\left(\delta\circ g_{s_*}^{k+1}\right).
	 \]
	 Hence $\varphi$ is zero on cohomology and the Lemma is proved.
\end{proof}
\begin{corollary}\label{CorollaryHcnInjective}
	For any $n\in \N$ the map
	\[
	c_n^\bullet \colon\ \colim_{s\in \N} (B_{n,s}^{\bullet})^\vee\rightarrow\left( \lim_{s\in \N} B_{n,s}^{\bullet}\right)\dual
	\]
	induces an injective morphism on the cohomology groups.
\end{corollary}
\begin{proof}
	Clearly, as we work with $\fK$-vector spaces, the filtration of $\aW$ admits a system of $\fK$-linear projections $g_i:\,\fW_k^{\rm alg}\rightarrow F_i$, for all $i\ge 0$, with the desired properties.

	We recall that for any $Y\in\fn^-$ the diagram
		\[
	\begin{tikzcd}
		Y\colon \,\fW^{\rm alg}_k\ar[r]\ar[d, "g_i"]& \fW^{\rm alg}_k\ar[d, "g_{i+1}"]\\
		Y\colon \,F_i\ar[r]&F_{i+1}
	\end{tikzcd}
	\]
	commutes by hypothesis. The projections induce sections $M^s\rightarrow M=\lim_{s\in \N}M^s$, hence sections
	\small
	\[
		g_s^k\colon \  U(\fn^-)^{\leqslant n-k}\tensor_S M^{s+k}\tensor_S \wedge_S^k \fn^-\rightarrow \lim_{s\in \N}\  U(\fn^-)^{\leqslant n-k}\tensor_S M^{s+k}\tensor_S \wedge_S^k \fn^-.
	\]
	\normalsize
	Since the Koszul complex $d_{n,s}^{B,k}\colon B^k_{n,s}\rightarrow B^{k+1}_{n,s}$ is defined in terms of the $\fn^-$action and the projections commute with this action we get that the sections are compatible with the morphisms $d_{n,s}^{B,k}$ and we apply the Lemma \ref{LemmaInjectivityOnsCohomology}.
	
\end{proof}
\begin{lemma}\label{LemmaInjectivityOnnCohomology}
	Let $\{A_n^\bullet,\ d_n^\bullet\}_{n\in \N}$ be a projective system of complexes of $\fK$-vector spaces. Suppose that the complex $\lim_{n\in\N}d_n^\bullet$ splits, i.e. for each $k\in\Z$ the short exact sequence
	\[
		0\rightarrow\Ker\left(\lim_{n\in \N}d_n^{k}\right)\rightarrow \lim_{n\in \N}A_n^k\rightarrow \frac{\lim_{n\in \N}A^k_n}{\Ker\left( \lim_{n\in \N}d_n^{k}\right)}\rightarrow 0
	\] 
	splits, i.e. there is a projection $h^k: \lim_{n\in \N}A_n^k\rightarrow \Ker\left(\lim_{n\in \N}d_n^{k}\right) $. Suppose that for each $n\in \N$ and $k\in \Z$ there is a section $g_n^k: A_n^k\rightarrow \lim_{n\in \N} A_n^k$ of the canonical projection.
	Then 
	\[
		{\rm H}^k\left(\lim_n A^\bullet_n\right)\hookrightarrow \lim_{n\in \N} {\rm H}^k(A_n^\bullet)
	\]
	is an injective morphism.
\end{lemma}
\begin{proof}
	Let $x=(x_n)_{n\in \N}\in \Ker\left(\lim_{n} (d_n^k)\right)$ with $x_n\in \Ker\left(d_n^k\right)\subset A_n^k$ the $n$-th projection of $x$. Suppose that $[x_n]_{n\in \N}\in \lim_{n\in \N} {\rm H}^k\left( A^\bullet_n\right)$ vanishes, i.e. that for each $n\in \N$ there is a $z_n\in A^{k-1}_n$ such that $d_n^{k-1}(z_n)=x_n$.  We want to show that also $[x]\in {\rm H}^k\left(\lim_{n\in \N} A_n^\bullet\right)$ is zero. A priori the system $z_n$ is not compatible, but we know that the map $\lim_{n}d_n^{k-1}$ split, then we can define 
	\[
		w_n:= p_n^k \left(g_n^k(z_n)- h^k g_n^k(z_n)\right),
	\]
	where $p_n^k: \lim_{n\in \N} A_n^k\rightarrow A_n^k$ is the canonical projection. Observe that for each $n\in \N$ we get that $d_n^k(w_n)=x_n$ and the system $w_n$ is a compatible system. Then the element $(w_n)_{n\in \N}\in \lim_{n\in \N}A_n^{k-1}$ is well defined and his image is $x$.
\end{proof}

\begin{corollary}\label{CorollaryVanishingContinuousDualCohomology}
	The map
\[
(\gamma^{\bullet})^\ast:\colim_{s\in \N} \left(\oplus_{\chi\in \Theta} \left(\colim_{n\in \N} B_{n,s}^\bullet\right)_{\chi}\right)\dual\longrightarrow  \lim_{n\in \N}\colim_{s\in \N} \left(B_{n,s}\dual\right).
\]
vanishes on the cohomology groups.
\end{corollary}
\begin{proof}
	
	We apply the Lemma \ref{LemmaInjectivityOnnCohomology} to the projective system 
	\[
	\colim_{s\in \N} \left(B^\bullet_{n,s}\right)\dual= \left(U(\fn^-)^{\leqslant n-\bullet}\right)\dual \tensor_{K} \left(\fW_\kappa^{alg}\tensor_K\left(\wedge_S^\bullet \fn^-\right)\dual\right).
	\] The discussion  about splittings in the proof of Lemma \ref{LemmaInjectivityOnsCohomology} applies here as well, and by the isomorphism $U(\fn^-)\cong \Sym_B\left( \fn^-\right)$ we get the sections $\left(U(\fn^-)^{\leqslant n-k}\right)\dual\rightarrow \left(U(\fn^-)\right)\dual$. Then 
\[
{\rm H}^k\left(\lim_{n\in \N}\colim_{s\in \N}\left(B_{n,s}^{\bullet,\vee}\right) \right) \hookrightarrow \lim_{n\in \N} {\rm H}^k\left(\colim_{s\in \N} \left(B_{n,s}^{\bullet,\vee}\right)\right).
\]Using that the limit commutes with kernels and the fact that ${\rm H}^k(c_n^\bullet)$ is injective all $n,k\in \N$ by the Corollary \ref{CorollaryHcnInjective}, one gets that
	\small
	\[
	 \lim_{n\in \N} {\rm H}^k\left(\colim_{s\in \N} (B_{n,s}^{\bullet})^\vee\right)\hookrightarrow \lim_{n\in \N} \left(\colim_{s\in \N } B_{n,s}^\bullet\right)\dual =\left(\cD^{\bullet}\right)^\ast
	\]
	\normalsize
	is injective. We conclude as was explained at the beginning of this section: we use the diagram (\ref{DiagramContinuousDualTheorem}), the vanishing of ${\rm H}((\gamma^{\bullet})^\vee)$ and the injectivity of ${\rm H}(c^\bullet)$.
\end{proof}

\section{Application to the elliptic case, i.e. for $G:=\GL_{2/\Q}$.}\label{sec:applyelliptic}

In this section we follow the outline of section \ref{sec:outline} and give all the details of Assumption \ref{ass:basicassumption}, the construction of sheaves, the BGG-decomposition and the quasi-isomorphism of the BGG-complex to the de Rham complex.
In this simple case, all computations can be done explicitly and the conclusions follow easily, confirming our proofs in section \ref{sec:details}.

\subsection{de Rham sheaves for $\GL_{2/\Q}$.}

\subsubsection{Review of the VBMS construction and some variants.}
\label{sec:VBMS}

Let $p>2$ be a prime integer, and $G:=\GL_{2/\Z_p}$. We fix a Borel subgroup $Bo$, and a torus $T$, such that $T\subset Bo\subset G$. Let $\cX\lra {\rm Spa}(\Q_p, \Z_p)$ be an adic analytic space and let $(\cE, \cE^+)$ denote a pair consisting of a locally free $\cO_{\cX}$-module of rank $2$ and a sub-sheaf $\cE^+$ of $\cE$ which is a locally free $\cO_{\cX}^+$-module of rank $2$ such that $\cE^+\otimes_{\cO_{\cX}^+}\cO_{\cX}=\cE$.
 
 Let $\cI\subset \cO_{\cX}^+$ be an invertible ideal, which gives the topology of $\cO_{\cX}^+$ and let $r\ge 0$ be such that $\cI\cap \Z_p=p^r\Z_p$. We suppose that there is a section $s\in {\rm H}^0(\cX, \cE^+/\cI\cE^+)$ such that, locally on $\cX$, $s$ is an element of an $\cO_{\cX}^+/\cI$-basis $\{s, s'\}$ of $\cE^+/\cI\cE^+$. We have the following (see \cite{andreatta_iovita}) 
 
 \begin{theorem}
 The functor associating to every adic space $\gamma:\cZ\lra \cX$, such that $\gamma^\ast(\cI)$ is an invertible $\cO_{\cZ}^+$-ideal, the set $\bV_0(\cE^+, s)(\gamma:\cZ\lra \cX)$ of elements $h\in {\rm Hom}_{\cO_{\cZ}^+}\bigl(\gamma^\ast(\cE^+), \cO_{\cZ}^+ \bigr)=
 {\rm H}^0\bigl(\cZ, (\cE^+)^\vee\bigr)$ such that:
 
 a) the set $\{h\}$ can be completed, locally on $\cZ$ to a basis $\{h, h'\}$ of $\gamma^\ast\bigl(\cE^+\bigr)^\vee$;
 
 b) $h^\vee\ {\rm mod}\bigl(\gamma^\ast(\cI\cE^+)\bigr)=\gamma^\ast(s)$;\\
 is representable by the adic space $\bV_0(\cE^+,s)$.
 
 \end{theorem} 
 
 \subsubsection{The sheaf $\WW_k$.}

 We work in the notations and hypothesis of section \S \ref{sec:VBMS}. Let $\cW$ be the weight space, i.e. the adic space over 
 ${\rm Spa}(\Q_p, \Z_p)$ associated to the complete noetherian $\Z_p$-algebra $\Lambda:=\Z_p[[\Z_p^\ast]]\cong \Z_p[\Z/(p-1)\Z][[T]]$, i.e. $\cW$ is the space of locally analytic characters on $\Z_p^\ast$. We recall the universal weight $k^{\rm univ}:\Z_p^\ast \lra \Lambda^\ast$ sending $a\in \Z_p^\ast$ to the group-like element $[a]\in\Lambda^\ast$. Suppose we have a morphism of adic spaces $ \cX\lra \cW$, such that the composition $\displaystyle \Z_p^\ast\stackrel{k^{\rm univ}}{\lra}\Lambda^\ast\subset \cO_{\cW}(\cW)^\ast\lra \bigl(\cO_{\cX}^+(\cX)\big)^\ast$ is an analytic function when restricted to $1+p^r\Z_p=(1+\cI)\cap\Z_p^\ast$. By analyticity, the inverse image of $k^{\rm univ}$ is defined on $1+\cI$.
 
 We denote by $\cT$ the adic torus, which represents the functor associating to every adic space $\gamma:\cZ\lra \cX$, such that
 $\gamma^\ast (\cI)$ is invertible, the set: $\cT(\gamma:\cZ\lra \cX):=1+\gamma^\ast(\cI)$.
 We have a natural action of $\cT$ on $\bV_0(\cE^+, s)$, defined on $\gamma:\cZ\lra \cX$-points by: $u\ast h:=uh$ with $u\in \cT(\gamma:\cZ\lra \cX)$ and $h\in \bV_0(\cE^+, s)(\gamma:\cZ\lra \cX)$.
 
 \begin{definition}
 Let $f:\bV_0(\cE^+, s)\lra \cX$  the natural morphism of adic spaces. We define: $\WW_{k^{\rm univ}}(\cE^+,s):=f_\ast\bigl(\cO_{\bV_0(\cE^+,s)}\bigr)[k^{\rm univ}]$, where if $\cG$ is an $\cO_{\cX}^+$-module with an action of $\cT$, we denote $\cG[k^{\rm univ}]$ the sub-sheaf  of $\cG$ of sections $x$ such that for all corresponding section $u$ of $\cT$, we have 
 $u\ast x=k^{\rm univ}(u)x$.
 \end{definition}  
 
 \begin{remark}
 If $\cZ:={\rm Spa}(S,S^+)\lra \cW$ is a morphism of adic spaces and $k:\Z_p^\ast\lra S^{+,\ast}$ is a $\cZ$-valued weight, as such a weight factors through the universal one, the above definition gives also the definition of $\WW_k$.  
 \end{remark}
 
 \subsubsection{An example: analytic induction.}
 \label{sec:exampleind}

  We continue working in the notations of section \ref{sec:VBMS}. Let us recall our sequence of group-schemes over $\Z_p$: $T\subset Q\subset G=\GL_2$. Let us fix a weight
  $k:\Z_p^\ast\lra (B^+)^\ast$, where $B^+$ is a $p$-adically complete, topologically of finite type $\Z_p$-algebra. We denote
  $B:=B^+\hat{\otimes}_{\Z_p}\Q_p$ and suppose $S:={\rm Spa}(B,B^+)\lra {\rm Spa}(\Q_p, \Z_p)$ is smooth (in fact we may suppose 
  $S$ is a closed polydisk in $\cW$) and moreover, we suppose that $k$ is $n$-analytic, i.e. analytic when restricted to 
  $1+p^n\Z_p$. In particular this implies that there is $u_k\in B^+$ such that $k(t)=\exp(u_k\log(t))$ for all $t\in 1+p^n\Z_p$.

  Let  $\bG$ denote the adic analytic group which represents the functor associating to ${\rm Spa}(R,R^+)\lra \cX$ the group
 $\bG(R, R^+):=\bigl\{\alpha\in G(R^+)\quad |\quad \bigl(\alpha\ {\rm mod}p^n\bigr)\in N(R^+/p^nR^+)\subset G(R^+/p^nR^+)\bigr\}$, where we denoted $N\subset Q$ the unipotent radical, and $T_n(R^+):={\rm Ker}\bigl(T(R^+)\lra T(R^+/p^nR^+)\bigr)$. We also define the adic analytic group $\bB$, which represents the functor ${\rm Spa}(R, R^+)\to \bB(R,R^+):=T_n(R^+)\cdot N(R^+)$.
 
 \begin{definition}
 We denote
 $$
 \fW_k:=\bigl({\rm Ind}_{\bB}^{\bG}\bigr)^{\rm an}(k):=\bigl\{f:\bG\lra \bA^{\rm 1, an} \quad | \mbox{ such that } f  \mbox{ is a morphism of adic analytic groups and }
 $$
 $$
 \mbox{ for every } {\rm Spa}(R, R^+)\to \cX, f(bg)=k(b)f(g), \ \mbox{for all } b\in \bB(R,R^+), g\in \bG(R,R^+)\bigr\}.
 $$
 Moreover, if $\displaystyle b= tn\in \bB(R,R^+)=T_n(R^+)\cdot N(R^+))$, with $t:=\left( \begin{array}{cc} {a} & {0} \\ {0} & {d}\end{array} \right)\in T_n(R^+)$ then $k(b):=k(ad^{-1})$.
 \end{definition}
 
We notice that we have a right, continuous action of $\bG$ on $\fW_k$.

\bigskip
\noindent
On the other hand, let us now choose $\cX=S:={\rm Spa}(B, B^+)$ seen as an adic space, and let $\cF^+:=V:=(B^+)^2:=e_1B^+\oplus e_2B^+$, seen as a free $B^+$-module of rank $2$. Let $s':=e_1\bigl({\rm mod}\ p^n\bigr)\in \cF^+/\cI$, where $\cI=p^n\cO_{\cX}^+$. Let also $\cE^+:=V^\vee$, the $B^+$-dual of $V$, and let $s:=e_1^\vee({\rm mod} p^n)$, then $(\cF^+, s')$ is a free sheaf with a marked section and let $\WW_k$ be the $B^+$-module associated to the pair $(\cF^+, s')$, constructed in the previous section

\begin{proposition}
\label{prop:fWk=WWk}
There is a natural $\bG(B, B^+)$ action on $\WW_k$, and a canonical isomorphism as $\bG(B, B^+)$-modules $\fW_k\cong \WW_k$.
\end{proposition}

\begin{proof} We consider the adic space $\cS:=\cS_{(\cE+,s)}$ which represents the functor associating to every morphism of adic spaces ${\rm Spa}(R, R^+)\lra \cX$,
the set 
$$
\cS_{(\cE^+,s)}(R,R^+):=
\bigl\{(e_1', e_2'), \ R^+-\mbox{ basis of } V^\vee\otimes_{\cO_K}R^+ \mbox{ such that } 
$$
$$
 e'_2=e_2^\vee\otimes 1 \bigl({\rm mod}\ p^n(V\otimes R^+)\bigr), \ e'_1\in \bigl((e^\vee_1\otimes 1)+R^+(e_2^\vee\otimes 1)\Bigr)\ {\rm mod}\ p^n         \bigr\}.
$$
We remark that there is a natural left $\bG$-action on $\cS$. 
There is also a natural morphism of adic spaces $\varphi: \cS=\cS_{(\cE^+, s)}\lra \bV_0(\cF^+, s')$, defined on points by: let $\cZ:={\rm Spa}(R,R^+)$ be as above, then $\varphi_\cZ(e_1', e_2')=e_1'\in \bV_0(\cF^+,s')(\cZ)=\{h\in {\rm Hom}_{R^+}(V\otimes R^+, R^+)=V^\vee\otimes R^+, \quad | \quad h(e_1\otimes 1)\ {\rm mod}\  p^n=1\}$. It is easy to see that the morphism $\varphi$ realizes $\cS_{(\cE^+, s)}$ as an $N$-torsor over $\bV_0(\cF^+, s')$. Therefore we have, denoting $\pi:\bV_0(\cF^+, s')\lra X$ and $\pi_1:\cS_{(\cE^+,s)}\lra X$,
$$
\pi_\ast\bigl(\cO_{\bV_0(\cF^+,s')}\bigr)[k]\cong \Bigl(\pi_{1, \ast}\bigl(\cO_{\cS_{(\cE^+, s)}}\big)^{N}\Bigr)[k]=\bigl({\rm Ind}_{\bB}^{\bG}\bigr)^{\rm an}(k),
$$
considering the right $\cB\subset \bG$ action on $\pi_{1,\ast}\bigl(\cO_{\cS_{(\cE^+,s)}}\bigr)$. Therefore we have 


$$\fW_k\cong \Gamma\bigl(\cS_{(\cE^+,s)}, \cO_{\cS_{(\cE^+,s)}}\bigr)[k]\cong\Gamma\bigl(\bV_0(\cF^+,s'), \cO_{\bV_0(\cF^+,s')}\bigr)[k]\cong\WW_k.$$

\end{proof}

\bigskip

\subsubsection{de Rham sheaves on modular curves.}
\label{sec:derhammodular}

Let $N\ge 5$ be an integer and $p\geq 3$ a prime integer such that $(N,p)=1$ and consider the 
 modular curve $X_1(N)$ over $\Z_p$, which classifies generalized elliptic curves with $\Gamma_1(N)$-level
structures. We denote $\hat{X}_1(N)$ the formal completion of this curve along its special fiber and denote by $\cX_1(N)$ the adic analytic generic fiber associated to the above formal scheme. We denote $\bE$ the universal generalized elliptic curve over this formal scheme or adic space.

Consider the ideal ${\rm Hdg}$, called the Hodge ideal, defined as the ideal of $\cO_{\hat{X}_1(N)}$, locally (on open affines ${\rm Spf}(R)\subset \hat{X}_1(N)$
which trivialize the sheaf $\omega_\bE$), generated by $p$ and  a local lift, ${\rm Ha}^o(E/R,  \omega)$, of the Hasse invariant ${\rm Ha}$, where $\omega$ is a basis of $\omega_E$. For every integer $r\geq 2$ we denote by
$\fX_r$ the formal open sub-scheme of the formal admissible blow-up of  $\hat{X}_1(N)$ with respect to the sheaf of ideals $(p, {\rm Hdg}^r)$, where this ideal is generated by ${\rm Hdg}^r$. Let $\mathcal{X}_r$ denote the adic generic fiber of $\fX_r$. By construction ${\rm Hdg}$ is an invertible ideal in $\fX_r$.  Recall from \cite[App. A]{halo_spectral}  that the universal generalized elliptic curve $\bE\to \fX_r$ has a canonical subgroup  ${\rm H_m}\subset \bE[p^m]$ of order $p^m$, where $m$ depends on $r$. For example, if $r\ge 2$ then $m\ge 1$ and if $r\ge p+2$ then $m\ge 2$.

\medskip
\noindent
Let us denote by $\pi\colon \mathcal{IG}_{m,r}:={\rm Isom}\bigl(\underline{\Z/p^m\Z}, {\rm H}_m^D\bigr)\to \mathcal{X}_r$ the $m$-th layer of the adic analytic Igusa tower over $\mathcal{X}_r$, where ${\rm H}_m^D$ denotes the Cartier dual of ${\rm H_m}$. Then $\mathcal{IG}_{m,r}$ is a finite, \'etale, Galois cover
of $\mathcal{X}_r$, with Galois group $(\Z/p^m\Z)^\ast$ and  we denote by $\fIG_{m,r}$ the formal scheme which is the normalization of
$\fX_r$ in $\mathcal{IG}_{m,r}$.  Let $r\ge 2$.

\begin{proposition}
\label{prop:cansgr}

We have

\begin{itemize}

\item[i.] the canonical subgroup ${\rm H}_1$ of the universal elliptic curve $\bE$ over $\fIG_{1,r}$ is a lifting of the kernel of Frobenius modulo $p/\mathrm{Hdg}$;

\item[ii.] we have an isomorphism between $C:=\bE[p]/{\rm H}_1$ and ${\rm H}_1^D$ defined by the Weil pairing on $\bE[p]$, which gives a canonical section $P^{\rm univ}$ of $C$;

\item[iii. ] the map of invariant differentials associated to the inclusion ${\rm H}_1\subset \bE$ induces a map $\omega_{{\rm H}_1} \to \omega_\bE/(p\mathrm{Hdg}^{-1}) \omega_\bE$ so
that via ${\rm dlog}_{{\rm H}_1}\colon L={\rm H}_1^D \to \omega_{{\rm H}_1}$ we get a section $\displaystyle s':={\rm dlog}_{{\rm H}_1}(P^{\rm univ})\in \mathrm{H}^0\bigl(\fIG_{1,r}, \omega_{\bE}/(p\mathrm{Hdg}^{-1})\omega_{\bE}\bigr)$;

\end{itemize}
\end{proposition}

\begin{proof}:  These statements are proved, for example, in \cite[Appendix A]{halo_spectral}.
 \end{proof}

\noindent
In the notations of Proposition \ref{prop:cansgr} iii), let $\displaystyle s':={\rm dlog}_{{\rm H}_1}(P^{\rm univ})\in \mathrm{H}^0\bigl(\fIG_{1,r}, \omega_{\bE}/\bigl(p\mathrm{Hdg}^{-1}\bigr)\omega_{\bE}\bigr)$ be the section defined there.

\begin{lemma}

Any local lift of the section $s':={\rm dlog}_{{\rm H}_1}(P^{\rm univ})$, as above, to a local section  $\tilde{s}$ of $\omega_{\bE}$, spans the $\cO_{\fIG_{1,r}}$-submodule
$\mathrm{Hdg}^{\frac{1}{p-1}}\omega_{\bE}\subset \omega_{\bE}.$

\end{lemma}

\begin{proof}:   
The statement of the Lemma strengthens \cite{halo_spectral}, where it was stated for $r\geq p^2$. It follows in this stronger form by the explicit computation of ${\rm dlog}_{\rm H}\colon L={\rm H}^D \to \omega_{\rm H}$ using Oort-Tate group schemes in \cite[Prop. 6.2]{andreatta_iovita_stevens}.  
\end{proof}

\medskip
\noindent
We'll use the trivialized canonical subgroups on $\fIG_{1,r}$ in order to define locally free sheaves with marked sections on this formal scheme, to which we will apply the VBMS-machine presented in section \S \ref{sec:VBMS}. More precisely, in the notations above, we define the invertible $\cO_{\fIG_{1,r}}$-submodule
$\Omega_\bE$ of $\omega_{\bE}$ as the span of any lift
$\tilde{s}$ of $s$ such that, if we set $\underline{\delta}:=\Omega_\bE\omega_\bE^{-1}$, then $\underline{\delta}$ is an invertible $\cO_{\fIG_{1,r}}$-ideal with
$\underline{\delta}^{p-1}= \pi^\ast({\rm Hdg})$ (recall that $\pi\colon \fIG_{1,r} \to \fX_r$ is the natural projection). From $\tilde{s}$ we also get a canonical section $s$
of $\mathrm{H}^0\bigl(\fIG_{1,r},
\Omega_{\bE}/p\mathrm{Hdg}^{-\frac{p}{p-1}}\Omega_{\bE}\bigr)$ so
that $s$ defines a basis of
$\Omega_{\bE}/p\mathrm{Hdg}^{-\frac{p}{p-1}}\Omega_{\bE}$ as
$\cO_{\fIG_{1,r}}/ p\mathrm{Hdg}^{-\frac{p}{p-1}}\cO_{\fIG_{1,r}}$-module. Therefore, our first locally free sheaf with marked section 
on $\bigl(\fIG_{1,r}, \cI:=p\mathrm{Hdg}^{-\frac{p}{p-1}}\bigr)$ is the pair $\bigl(\Omega_{\bE}, s\bigr)$.

 \medskip
 \noindent
 We now define a second locally free sheaf with marked section on $\bigl(\fIG_{1,r}, \cI:=p\mathrm{Hdg}^{-\frac{p}{p-1}}\bigr)$.
We denote by ${\rm H}_\bE^\#$ the push-out in the category of coherent sheaves on $\fIG_{1,r}$
of the diagram
\[
\begin{tikzcd}[row sep=10]
	\underline{\delta}^p\omega_E\ar[r]\arrow[d, phantom, sloped, "\subset"] &\underline{\delta}^p{\rm H}_{\bE}\\
	\Omega_{\bE}&
\end{tikzcd}.
\]
We then have the following commutative diagram of sheaves with exact rows:
\[
\begin{tikzcd}[row sep=10]
	0\ar[r]&\underline{\delta}^p\omega_\bE\ar[r]\arrow[d, phantom, sloped, "\subset"]&\underline{\delta}^p{\rm H}_{\bE}\ar[r]\ar[d]&\underline{\delta}^p\omega_{\bE}^{-1}\ar[r]\ar[d, equal]&0\\
	0\ar[r]&\Omega_{\bE}\ar[r]\arrow[d, phantom, sloped, "\subset"]&{\rm H}_{\bE}^\#\ar[r]\arrow[d, phantom, sloped, "\subset"]&\underline{\delta}^p\omega_{\bE}^{-1}\ar[r]\arrow[d, phantom, sloped, "\subset"]&0\\
	0\ar[r]&\omega_{\bE}\ar[r]&{\rm H}_{\bE}\ar[r]&\omega_\bE^{-1}\ar[r]&0
\end{tikzcd}
\]
It follows that ${\rm H}_\bE^\#$ is a locally free $\cO_{\fIG_{1,r}}$-module of rank two and 
 $(\Omega_\bE, s)=(\underline{\delta}\omega_\bE, s)\subset ({\rm H}_\bE^\#, s)$ is a compatible inclusion of locally free sheaves with marked sections.   We also have the following commutative diagram with exact rows:
 \[
 \begin{tikzcd}[row sep=10]
 	0\ar[r]&\Omega_{\bE}\ar[r]\ar[d, equal]&{\rm H}_{\bE}^\#\ar[r]\arrow[d, phantom, sloped, "\subset"]&\underline{\delta}^p\omega_{\bE}^{-1}\ar[r]\arrow[d, phantom, sloped, "\subset"]&0\\
 	0\ar[r]&\Omega_{\bE}\ar[r]&\underline{\delta}{\rm H}_{\bE}\ar[r]&\underline{\delta}\omega_\bE^{-1}\ar[r]&0
 \end{tikzcd}
 \]
in which the right square is cartesian, defining ${\rm H}_\bE^\#$ as pull-back.

We have natural actions of $\fT^{\rm ext}:=  \Z_p^\ast \bigl(1+\pi_\ast\bigl(p\mathrm{Hdg}^{-\frac{p}{p-1}} \cO_{\fIG_{1,r}}\bigr)\bigr)$
 on the morphisms of formal schemes:
 $$
  u\colon \bV_0({\rm H}_\bE^\#,s)\lra  \fX_r,  \mbox{ and on } v\colon \bV_0(\Omega_\bE,s)\lra \fX_r,
 $$ 
 with trivial action on $\fX_r$.

\begin{definition}\label{def:analweight} Given a ring $R$ which is $p$-adically complete and separated, 
we say that a homomorphism $\nu \colon\Z_p^\ast\lra R^\ast$ is an analytic weight if there exists $u\in R$ with the property that $\nu(t)= \exp (u \log t)$, for every $t\in 1+ p\Z_p$. 
\end{definition}

Assume now that $r\ge 2$ for $p\geq 5$ and $r\geq 4$ for $p=3$. Let $\nu$ be an $R$-valued analytic weight.

\begin{definition}\label{def:bW} We define $\fw^\nu:=v_\ast(\cO_{\bV_0(\Omega_\bE,s)}\widehat{\otimes}_{\Z_p} R)[\nu]$
and $\WW_\nu:=u_\ast(\cO_{\bV_0({\rm H}_\bE^\#,s)}\widehat{\otimes}_{\Z_p} R)[\nu]$ as the sub-sheaves of sections of the respective sheaves on which $\fT^{\rm ext}$ acts via $\nu$ (see \cite[\S 3.1\& \S 3.3]{andreatta_iovita}).
\end{definition}

The definition makes sense  if $r\geq p^2$ for any prime $p$ as explained in \cite[\S 3.2 \& \S 3.3]{andreatta_iovita}.

\begin{remark}\label{rmk:functVBMS} {\bf Specialization.}  In the notations of definition \ref{def:bW}, assume that $R'$ is the ring of integers of a finite extension of $\Q_p$ and that we have an algebra homomorphism $R\to R'$. Let $x$ be an $R'$-valued point of $\fX_r$ defined by an elliptic curve $E$ over $R'$. Set $\Omega_E:=x^\ast(\Omega_\bE)$,  ${\rm H}_E^\#:=x^{\ast}\bigl({\rm H}_\bE^\#\bigr)$ with induced section $s_x$ obtained from $s$ by pulling-back via $x$. Then, applying the construction above to $(\Omega_E,s_x)$ and $({\rm H}_E^\#,s_x)) $ and the $R'$-valued weight $\nu':\Z_p^\ast\lra (R')^\ast$ which is the composition 
$\displaystyle \Z_p^\ast\stackrel{\nu}{\lra}R^\ast \lra (R')^\ast$, we get $R'$-modules $\fw^{\nu'}_{R'}\subset  \WW_{\nu',R'}$ that coincide with $x^\ast(\fw^\nu)$ and $x^\ast(\WW_\nu)$ respectively by the description provided in Section \ref{section:local}.  We sometimes write  $\fw^{\nu'}_{R'}=\fw^{\nu'}_{R'}(\Omega_E,s_x)$ and $ \WW_{\nu',R'}=\WW_{\nu',R'}({\rm H}_E^\#,s_x)$ if needed.

\end{remark}

As in \cite[\S 3.2 \& \S 3.3]{andreatta_iovita} one has the following:

\begin{proposition} The sheaf $\fw^{\nu}$ is an invertible
$\cO_{\fX_r}\widehat{\otimes} R$-module and $\WW_\nu$ has a natural, increasing filtration $\bigl({\rm Fil}\bigr)_{n\ge 0}$ by $\cO_{\fX_r}\widehat{\otimes} R$-submodules such that $\fw^\nu$ is identified with ${\rm Fil}_0$. 

The Gauss-Manin connection $\nabla\colon {\rm H}_\bE\to {\rm H}_\bE\otimes \Omega^1_{\fX_r/R}\bigl({\rm log}({\rm cusps})\bigr)$ induces a connection $\nabla_{\nu}$ (with poles) on $\WW_\nu$.
\end{proposition}



\subsubsection{Assumption \ref{ass:basicassumption} in the elliptic case. The torsor of trivializations of the pair $(\Omega_{\bE}, {\rm H}_\bE^\#)$.}
\label{sec:GMperiod}

Let $r\ge 2$ be an integer and $k:\Z_p^\ast\lra (B^+)^\ast$ be an $n$-analytic weight, for $n\ge 1$, see section \ref{sec:exampleind}. Let $\fIG_{n,r}\lra \fX_r$ be the formal Igusa curve of level $n$ over the formal neighbourhood of the ordinary locus $\fX_r$ in $\hat{X}_1(N)$, base-changed to $\fS:={\rm Spf}(B^+)$. Over $\fIG_{n,r}$ we have the locally free $\cO_{\fIG_{n,r}}$-module of rank $2$, ${\rm H}_{\bE}^\#$, with its filtration $\Omega_{\bE}\subset {\rm H}_{\bE}^\#$. We denote by ${\rm Std}$ the standard representation of the analytic group $\bG$ over $B$,  by ${\rm Fil}_1$ its natural filtration, i.e. if $e_1,e_2$ is an ordered basis of ${\rm Std}$, ${\rm Fil}_1$ is generated by $e_1$. In the elliptic case, the parabolic subgroup $P$ in Assumption \ref{ass:basicassumption} is a Borel subgroup, denoted there ${\rm Bo}$, therefore the analytic sub-group $\PP$ is $\bB:={\rm Bo}\cap \bG$. The analytic group $\bG$ acts naturally on ${\rm H}^\#_{\bE}$ and $\bB$ preserves its Hodge filtration.

We consider the formal scheme over $\fS$, 
$$
\fT_{({\rm H}_{\bE}^\#)^\vee}:=\{\varphi\in {\rm Isom}\bigl(({\rm H}_{\bE}^\#)^\vee, {\rm Std}\otimes \cO_{\fIG_{n,r}}\bigr)\quad \mbox{such that } \varphi\Bigl({\rm Ker}\bigl({\rm H}_{\bE}^\#)^\vee\lra (\Omega_{\bE})^\vee\Bigr)={\rm Fil}_1\otimes \cO_{\fIG_{n,r}}\}.
$$
Then $\fT_{({\rm H}_{\bE}^\#)^\vee}\lra \fIG_{n,r}$ is a formal $\bB$-torsor. 

 We denote $\cT_{({\rm H}_{\bE}^\#)^\vee}\lra \cIG_{n,r}\lra \cX_r$ the log adic spaces associated to the formal schemes above, i.e. $\cX_r$ and $\cIG_{n,r}$ have the log structures induced by the divisors at the cusps, while $\cT_{({\rm H}_{\bE}^\#)^\vee}$ has the inverse image log structure. We denote $\pi:\cT_{({\rm H}_{\bE}^\#)^\vee}\lra \cX_r$ the composition.
 
 Let $U\subset \cX_r$ be an open affinoid such that there is a section $s:U\lra \cT_{({\rm H}_{\bE}^\#)^\vee}$ of $\pi$. We remark that if 
 $\varphi\in \cT_{({\rm H}_{\bE}^\#)^\vee}(\pi^{-1}(U))$ then $\varphi\otimes 1_K: \Bigl(\bigl({\rm H}_{\bE}^\#\bigr)^\vee\otimes_{\cO_K}K\Bigr)|_U=\bigl(({\rm H}^1_{\rm dR}{\bE_U/U}\bigr)^\vee\cong \bigl(\cO_U\bigr)^2$, compatibly with the Hodge filtrations.
 
Then, following section \ref{sec:classical}, using $\varphi\otimes_{\cO_K}K$ and the Gauss-Manin connection $\nabla_{\rm GM}$, we define the splitting $\eta_U\colon \fg/\fb\otimes_{\fK}\cO_U\lra \fg\otimes_{\fK}\cO_U$ and denote by
$\fn_U$ its image. It is an abelian Lie-sub-algebra of $\fg\otimes_{\fK}\cO_U$.

\begin{lemma}
\label{lemma:torsor}
 Then we have

a) $s^\ast\bigl(\aW\otimes \cO_{\cT_{({\rm H}_{\bE}^\#)^\vee}}\bigr)\cong \WW_k^{\rm alg}|_U$, where the isomorphism preserves the filtrations.

b) $(\fg/\fb)^\vee\otimes \cO_U \cong \omega^1_{\cX_r/S}|_U$. 

c) We have natural isomorphisms of $\cO_U$-modules, compatible with the Gauss-Manin connection $\psi_U:\cP_{U/S}\cong U(\fn_U)^\vee$.

\end{lemma}
 
 \begin{proof}
 
 a) We have $\pi^\ast\bigl(\WW_{k,\fK}^{\rm alg}\bigr)\cong \aW\otimes \cO_{\cT_{({\rm H}_{\bE}^\#)^\vee}}$, see Proposition \ref{prop:fWk=WWk}. This implies a).
 
 b) Let ${\rm KS}$ denote the Kodaira-Spencer isomorphism defined as the composition:
 $$
 \Omega_{\bE}\subset {\rm H}_{\bE}^\#\stackrel{\nabla}{\lra}{\rm H}_{\bE}^\#\otimes\omega^1_{\cX_r/S}\lra (\Omega_{\bE})^\vee\otimes\omega^1_{\cX_r/S}.
 $$
 We have 
 $$
 \begin{array}{cccccc}
 \pi^\ast\bigl(\Omega_{\bE})&\stackrel{\pi^\ast({\rm KS})}{\lra}&\pi^\ast\bigl((\Omega_{\bE})^\vee\bigr)\otimes\pi^\ast\bigl(\omega^1_{\cX_r/S}\bigr)\\
 \downarrow\cong&&\downarrow\cong\\
 {\rm Fil}_1\otimes\cO_{\cT_{({\rm H}_{\bE}^\#)^\vee}}&\cong&{\rm Gr}_1({\rm Std})\otimes\pi^\ast\bigl(\omega^1_{\cX_r/S}\bigr))
 \end{array}
 $$
 We have $\pi^\ast({\rm KS}):\pi^\ast\bigl(\omega^1_{\cX_r/S}\bigr)^\vee\cong {\rm Hom}\bigl({\rm Fil}_1, {\rm Gr}_1\bigr)\otimes \cO_{\cT_{({\rm H}_{\bE}^\#)^\vee}}$. On the other hand
 we have a commutative diagram
 $$
 \begin{array}{cccccccc}
 {\rm End}({\rm Std})&\stackrel{\rho}{\lra}&{\rm Hom}\bigl({\rm Fil}_1, {\rm Gr}_1\bigr)\\
 \cup&&\uparrow\cong\\
 \fg&\lra&\fg/\fb
 \end{array}
 $$
 Therefore if $s:U\to \cT_{({\rm H}_{\bE}^\#)^\vee}$ is a section of $\pi$, with $U\subset \cX_r$ an open affinoid, we have: $${\rm KS}: \bigl(\omega_{\cX_r/S}^1\bigr)^\vee|_U\cong \fg/\fb\otimes \cO_U,$$ 
 which is the dual of b).
 
 c) is proved in Lemma \ref{lemma:neidiag}.
 
 \end{proof}

\bigskip

\begin{remark}
Let us observe that Lemma \ref{lemma:torsor} implies Assumption \ref{ass:basicassumption} in the elliptic case.
\end{remark}

\subsection{{\rm BGG}-decomposition in the elliptic case.}

Let $k$ be a universal family weight with values in the $\Z_p$-complete algebra $B^+$ and $\fW_k$ our $\cG$-representation, $\fW_k:=\left({\rm Ind}_\cB^{\cG}\right)^{\rm an}$, see section \ref{sec:GMperiod}. Suppose that $k$ is $m$-analytic, i.e. there is $\alpha\in B:=B^+\otimes_{\Z_p}\Q_p$ such that $k(t)=\exp(\alpha\log(t))$ for all $t\in 1+p^{m}\Z_p\subset \Z_p^\ast$. We suppose that $B$ is an integral domain and denote, as in section \ref{sec:outline}, $\fK:={\rm Frac}(B)$.

Let us recall that if we denote by $V:=(B^+)^2=e_1B^+\oplus e_2B^+$ the standard representation of $\GL_2(B^+)$, and we let
$V_0:=e_1(1+p^{m}B^+)\times e_2B^+\subset V$, then we can write 
$$
\fW_k:=\{f:V_0\lra B\quad |\quad f(a(x,y))=k(a)f(x,y)\mbox{ for all } a\in 1+p^{m}\Z_p, (x,y)\in V_0, \\
$$
$$
\mbox{ such that } f|_{e_1\times e_2B^+}
\mbox{ is analytic }\}.
$$
By choosing appropriately variables $X,Z$, the elements of $\fW_k$ are written: $\sum_{n=0}^\infty a_nX^{k-n}Z^n$, with $a_n\in B$, $a_n\to 0$. As we have an analytic action of $\bG$ on $\fW_k$, its Lie-algebra acts naturally on the module, but ${\rm Lie}(\bG)\otimes_{\Z_p}\Q_p\cong \mathfrak{gl}_2$. In what follows we will be interested in the action of the semi-simple Lie-sub-algebra $\fg:=\mathfrak{sl}_2\subset \mathfrak{gl}_2$ on $\fW_k$. 

Let $$u^-:={\left( \begin{array}{cc} {0} & {0} \\ {1} & {0}
\end{array} \right)},u^+:={\left( \begin{array}{cc} {0} & {1} \\ {0} & {0}
\end{array} \right)}, H:={\left( \begin{array}{cc} {1} & {0} \\ {0} & {-1}  
\end{array} \right)}
$$ 
be the standard generators of $\mathfrak{g}:=\mathfrak{sl}_2$, then they act on $\fW_k$ by:
$$
u^-=Z\frac{\partial}{\partial X},\quad u^+=X\frac{\partial}{\partial  Z},\quad H=X\frac{\partial}{\partial X}- Z\frac{\partial}{\partial Z}.
$$
 
In particular, if set $Y:=Z/X$ can write $\fW_k=X^kB\langle Y\rangle$ and so an element can be uniquely written $X^k\sum_{n=0}^\infty a_nY^n$, with $a_n\to 0$ in $B$.
In this shape, the action of $\mathfrak{sl}_2$ is given by:
$$
u^-(X^kY^n)=(k-n)X^kY^{n+1}, \  H(X^kY^n)=(k-2n)X^kY^n, \ u^+(X^kY^n)=nX^kY^{n-1}.
$$

We denote $F_n:=X^kB[Y]^{\rm deg\le n}\subset \fW_k$ and let $\displaystyle \aW:=\colim_{n}(F_n\otimes_B\fK)=X^k\fK [Y]$.
We denote from now on ${\rm Fil}_n:=F_n\otimes_B\fK\cong X^k\fK [Y]\subset \aW$, which has a natural action of $\fb=\Q_pu^++\Q_p H$ and $u^-:{\rm Fil}_n\to {\rm Fil}_{n+1}$.

For every $N\ge 0$ let $V_N:=\bigl({\rm Fil}_N\bigr)^\vee={\rm Hom}_{\fK}({\rm Fil}_N, \fK)$ with the dual $\mathfrak{b}$-action.
Moreover if denote $(e_{N,i})_{i=0,N}$ the dual basis to the $\fK$-basis of ${\rm Fil}_N$, $(X^k, X^kY,...,X^kY^N)$, the action of 
$H$ is given by: $H(e_{N,i})=(-k+2i)e_{N,i}$ and $u^-:V_{N+1} \lra V_N$ is defined on basis by:
for every $1\le i\le N+1$, we have $u^-(e_{N+1, i})=(-k+i)e_{N, i-1}$ and $u^-(e_{N+1,0})=0$. Finally, we denote $V:=(\fW_k^{\rm alg})^\vee=\lim_N V_N$, with its projective limit topology and discrete topology on every $V_N$.

We recall   the dual Koszul complex associated to this situation: for every $N\ge 0$ we have
$$
(\ast)_N:\quad U(\fg)\otimes_{U(\fb)}\bigl(V_{N+1}\otimes_{\Q_p} \fg/\fb\bigr)\stackrel{\Xi_N}{\lra} U(\fg)\otimes_{U(\fb)}V_N. 
$$

We write $\fu^-\cong \fg/\fb$ and by using the Poincar\'e-Birkhof-Witt theorem, we write  for any $U(\fb)$-module $M$: $U(\fg)\otimes_{U(\fb)}M\cong U(\fu^-)\otimes_{\fK}M$, where $U(\fu^-)={\rm Sym}(\fu^-)=\fK[u^-]\cong \fK[t]$, as $\fK$-algebras. Here $t$ is a variable and the isomorphism sends $u^-\to t$.

Therefore the complex $(\ast)$ can be re-written:
$$
(\ast)_N:\quad \fK[t]\otimes_{\fK}\bigl(V_{N+1}\otimes_{\fK} \fu^-\bigr)\stackrel{\Xi_N}{\lra} \fK[t]\otimes_{\fK} V_N,
$$
where $t$ is a variable and the differential is defined by: for any $P(t)\in \fK[t], 0\le i\le N+1$ we have 
$$
\Xi(P(t)\otimes e_{N+1,i}\otimes u^-)=tP(t)\otimes e_{N, i}+(k-i)P(t)\otimes e_{N, i-1},$$ 
where we set $e_{N, N+1}=e_{N, -1}=0$.

We recall from section \ref{sec:outline} the two ways to assemble the complexes: $(\ast)_N$: 

 The complex $\cD^\bullet$ defined by
 $$
 \cD^\bullet: \quad [U(\fu^-)\otimes_{\fK}\bigl(\lim_N V_{N+1}\otimes_{\fK} \fu^-\bigr)\stackrel{\Xi_N}{\lra} U(\fu^-)\otimes_{\fK} \lim_N V_N]\cong  
 $$
 $$
 \cong [U(\fu^-)\otimes_{\fK}\bigl(V\otimes_{\fK} \fu^-\bigr)\lra U(\fu^-)\otimes_{\fK} V].
 $$ 
 This is the Koszul complex of $V$. 
 
Another way to assemble the complexes $(\ast)_N$ is the complex $\cC^\bullet$ defined by:
 $$
 \cC^\bullet:\quad[\lim_N\Bigl(U(\fu^-)\otimes_{\fK}\bigl(V_{N+1}\otimes_{\fK} \fu^-\bigr)\Bigr)\lra \lim_N\Bigl( U(\fu^-)\otimes_{\fK} V_N\Bigr)].
 $$
 
 There is a morphism of complexes $\xi:\cD^\bullet\lra \cC^\bullet$, which, is proved in section \ref{sec:details}, induces $0$ as a map on the cohomology of the two complexes. It will be seen at the end of this section that in our simple situation, this follows easily from the following explicit computations.
 
 \medskip
 
We suppose that for any $n\in \Z$, $k-n\in \fK^\times$, i.e.  that $k$ is a generic $p$-adic weight (see Assumption \ref{ass:basicassumption}),
and we explicitly compute the cohomology of the complex $(\ast)_N$.

\bigskip
{\bf The Cokernel of $\Xi_N$}.

Let $P(t)\in \fK[t]$ and $0\le i\le N$. For all $0\le j\le i$, we set $Q_j(t)=0$, moreover we define
$$
Q_{i+1}(t):=\frac{1}{(k-i-1)}P(t),
\mbox{ and  for } i+1\le j\le N, \ Q_{j+1}(t):=-\frac{1}{(k-j-1)}t Q_j(t).
$$ 
A simple calculation shows
$$
\Xi_N\bigl(\sum_{j=0}^{N+1}Q_j(t)\otimes e_{N+1, j}\otimes u^-\bigr)=P(t)\otimes e_{N,i}.
$$
In other words $\Xi_N$ is surjective and so ${\rm Coker}(\Xi_N)=0$.

\bigskip

{\bf The Kernel of $\Xi_N$.}

\bigskip

Let $\sum_{i=0}^{N+1} P_i(t)\otimes e_{N+1,i}\otimes u^-\in \fK[t]\otimes V_{N+1}\otimes \fu^-$ be an element in ${\rm Ker}(\Xi_N)$.
So $0=\Xi_N\bigl(\sum_{i=0}^{N+1}P_i(t)\otimes e_{N+1,i}\otimes u^-)$ which implies that for every $0\le i\le N$ we have
$P_{i+1}(t)=-(k-i-1)^{-1}tP_i(t)$, in other words if we denote by:
$$
e:=\sum_{i=0}^{N+1} \frac{(-1)^i}{(k-1)(k-2)...(k-i)}t^i\otimes e_{N+1, i}\otimes u^-
$$
then $M:={\rm Ker}(\Xi_N)=\fK[t]e\subset \fK[t]\otimes V_{N+1}\otimes \fu^-$.

The question is what kind of $\fh:=H\fK$-module is $M$?
We compute the action of $H$ on $e$, recalling that the variable $t$ stands for $u^-$. For $0\le i\le N+1$ we have
$$
H((u^-)^i\otimes e_{N+1,i}\otimes u^-)=H(u^-)^i\otimes e_{N+1,i}\otimes u^-=([H, u^-]-u^-H)(u^-)^{i-1}\otimes e_{N+1,i}\otimes u^-=
$$

$$
=(-2i(u^-)^i+u^iH)\otimes e_{N+1,i}\otimes u^-=
$$
$$
=-2i(u^-)^i\otimes e_{N+1,i}\otimes u^-+(u^-)^i\otimes He_{N+1,i}\otimes u^-+(u^-)^i\otimes e_{N+1,i}\otimes [H, u^-]=
$$
$$
=-2i(u^-)^i\otimes e_{N+1,i}\otimes u^-+(2i-k)(u^-)^i\otimes e_{N+1,i}\otimes u^-+(u^-)^i\otimes e_{N+1,i}\otimes (-2u^-)=
$$
$$
=(-k-2)\Bigl((u^-)^i\otimes e_{N+1,i}\otimes u^-\Bigr).
$$
We notice that every term of the sum is an eigenvector for $H$ of eigenvalue $-k-2$.
In other words $He=(-k-2)e$ and moreover, $u^+(e)=0$.  
 It follows that ${\rm Ker}(\Xi_N)$ is the generalized  Verma module of weight $-k-2$, i.e. ${\rm Ker}(\Xi_N)=U(\fg)\otimes_{U(\fb)}\fK (-k-2)$. 

\bigskip

{\bf The BGG complex for $(\ast)_N$.}

\bigskip

We look again at our complex
$$
(\ast)_N:\quad U(\fu^-)\otimes_{\Q_p}\bigl(V_{N+1}\otimes_{\fK}{\Q_p}\fu^-\bigr)\stackrel{\Xi_N}{\lra} U(\fu^-)\otimes_{\Q_p} V_N,
$$
and look for the generalized eigenspace subcomplex for the action of $Z(\fg)$, the center of $U(\fg)$, with eigenvalues given by the values of the character $\chi_k$ (see lemma \ref{lemma:character}).
According to lemma \ref{lemma:diximier} we use the basis $e_{N+1, i}$ $i=0,1,2..., N+1$ for $V_{N+1}$ and the basis $e_{N,j}$ with $j=0, 1, ..., N$ 
for $V_N$. 
$H$ acts on the first basis by the weights: $-k-2, -k, -k+2, ..., -k+2N$ and on the second by $-k, -k+2,..., -k+2N$.
Now we write filtrations of the two terms of the complex with graded quotients Verma modules with those weights.
Let $\chi_k:Z(\fg)\lra C$ be the central character associated to the Verma module $V_{k}=\aW$. We know that if $V_\lambda$ is a generalized Verma module with character $\lambda$, then $(V_\lambda)_{\chi_k}$ is not zero if and only if $\lambda=k$ or $\lambda=-k-2$.

Therefore $\Bigl(U(\fu^-)\otimes_{\Q_p}\bigl(V_{N+1}\otimes_{\Q_p} \fu^-\Bigr)_{\chi_k}=U(\fg)\otimes_{U(\fb)}\Q_p(-k-2)$ and
$\Bigl(U(\fu^-)\otimes_{\Q_p} V_N\Bigr)_{\chi_k}=0$.

Thus the BGG complex associated to $(\ast)_N$ is: $({\rm BGG})_N^\bullet: U(\fg)\otimes_{U(\fb)}\fK(-k-2)\lra 0$, and indeed it is independent of $N$. By the previous subsection we have that the two complexes 
\[
	\begin{tikzcd}[row sep=10]
		(\ast)_N^\bullet:&U(\fg)\otimes_{U(\fb)}\bigl(V_{N+1}\otimes_{\Q_p} \fg/\fb\bigr)\ar[r, "\Xi_N"] &U(\fg)\otimes_{U(\fb)}V_N\\
		({\rm BGG})_N^\bullet:&U(\fg)\otimes_{U(\fb)}\fK (-k-2)\ar[r]\arrow[u, phantom, sloped, "\subset"]& 0 \arrow[u, phantom, sloped, "\subset"]
	\end{tikzcd}
\]
are quasi-isomorphic. The first inclusion in the above diagram identifies $U(\fg)\otimes_{U(\fb)}\fK (-k-2)$ with ${\rm Ker}(\Xi_N)\subset U(\fg)\otimes_{U(\fb)}\bigl(V_{N+1}\otimes_{\Q_p} \fg/\fb\bigr)$.

\bigskip

{\bf The BGG complex for $(\ast)_N$.}

\bigskip

We look again at our complex
$$
(\ast)_N:\quad R[t]\otimes_{R}\bigl(V_{N+1}\otimes_{R}\fn_U\bigr)\stackrel{\Xi_N}{\lra} R[t]\otimes_{R} V_N,
$$
and look for the generalized eigenspace subcomplex for the action of $Z(\fg)$, the center of $U(\fg)$, with eigenvalues given by the values of the character $\chi_k$ (see lemma \ref{lemma:character}).
According to lemma \ref{lemma:diximier} we use the basis $e_{N+1, i}$ $i=0,1,2..., N+1$ for $V_{N+1}$ and the basis $e_{N,j}$ with $j=0, 1, ..., N$ 
for $V_N$. 
$H$ acts on the first basis by the weights: $-k-2, -k, -k+2, ..., -k+2N$ and on the second by $-k, -k+2,..., -k+2N$.
Now we write filtrations of the two terms of the complex with graded quotients Verma modules with those weights.
Let $\chi_k:Z(\fg)\lra C$ be the central character associated to the Verma module $V_{k}=\aW$. We know that if $V_\lambda$ is a generalized Verma module with character $\lambda$, then $(V_\lambda)_{\chi_k}$ is not zero if and only if $\lambda=k$ or $\lambda=-k-2$.

Therefore $\Bigl(R[t]\otimes_{R}\bigl(V_{N+1}\otimes_{R} \fn_U\Bigr)_{\chi_k}=U(\fg)\otimes_{U(\fb)}R(-k-2)$ and
$\Bigl(R[t]\otimes_{R} V_N\Bigr)_{\chi_k}=0$.

Thus the BGG complex associated to $(\ast)_N$ is: $({\rm BGG})_N^\bullet: U(\fg)\otimes_{U(\fb)}R(-k-2)\lra 0$, and indeed it is independent of $N$. By the previous subsection we have that the two complexes 
\[
	\begin{tikzcd}[row sep=10]
		(\ast)_N^\bullet:&U(\fg)\otimes_{U(\fb)}\bigl(V_{N+1}\otimes_{R} \fg/\fb\bigr)\ar[r, "\Xi_N"] &U(\fg)\otimes_{U(\fb)}V_N\\
		({\rm BGG})_N^\bullet:&U(\fg)\otimes_{U(\fb)}R(-k-2)\ar[r]\arrow[u, phantom, sloped, "\subset"]& 0 \arrow[u, phantom, sloped, "\subset"]
	\end{tikzcd}
\]
are quasi-isomorphic. The first inclusion in the above diagram identifies $U(\fg)\otimes_{U(\fb)}R(-k-2)$ with ${\rm Ker}(\Xi_N)\subset U(\fg)\otimes_{U(\fb)}\bigl(V_{N+1}\otimes_{R} \fg/\fb\bigr)$.

\subsection{The complexes $\cC^\bullet$ and $\cD^\bullet$.}

For any $N\ge 0$ we have natural commutative diagrams with exact rows
\[
\begin{tikzcd}[row sep=10]
	(\ast)_{N+1}:& 0\ar[r]& {\rm Ker}(\Xi_{N+1})\ar[r] \ar[d, "\alpha_N"]&R[t]\otimes_R V_{N+2}\otimes \fn_U\ar[r, "\Xi_{N+1}"] \ar[d, "\beta_N"]&R[t]\otimes_R V_{N+1}\ar[r]\ar[d, "\gamma_N"]&0\\
	(\ast)_{N}:&0\ar[r]& {\rm Ker}(\Xi_{N})\ar[r] & R[t]\otimes_R V_{N+1}\otimes \fn_U\ar[r, "\Xi_{N}"] &R[t]\otimes_R V_{N}\ar[r]&0
\end{tikzcd}
\]
where $\alpha_N$ is induced by $\beta_N$ and $\beta_N, \gamma_N$ are induced by the natural projections $V_{i+1}\lra V_i$, $i\ge 0$.

We observe that the maps $\alpha_N$ are isomorphisms of $U(\fg)$-modules. By taking projective limits over the $N$'s and using the Mittag-Lefler property, we obtain the exact sequence:
$$
0\to U(\fg)\otimes_{U(\fb)}R(-k-2)\to \cC^1_R:=\lim_{\leftarrow,N}\bigl(U(\fg)\otimes_{U(\fb)}(V_N\otimes \fg/\fb)\bigr)\to \cC^0_R:=\lim_{\leftarrow,N}\bigl(U(\fg)\otimes_{U(\fb)}V_N\bigr)\to 0.
$$
 
 In other words, for every $U:={\rm Spa}(R,R^+)\subset \cX_r$ open affinoid, over which there is a section of the natural projection of the torsor $\cT_{({\rm H}^\#_{\bE})^\vee}\lra \cX_r$, we have a quasi-isomorphism of complexes, in the notations of Theorem \ref{thm:zerocoh}: 
 $$
\cF^\bullet_R: [U(\fg)\otimes_{U(\fb)}R(-k-2)\to 0]\lra \cC^\bullet_R:=[\cC^1_R\to \cC^0_R].
$$
The diagram of that theorem, denote it $(1)$:
\[
\begin{tikzcd}[row sep=12]
	0\ar[r]&\cF^\bullet_R\ar[r]&\cC^\bullet_R\ar[r]&\cG^\bullet_R\ar[r]&0\\
	&&\cD^\bullet_R\ar[r, equal]\ar[u, "\xi"]&\cD^\bullet_R\ar[u, "\gamma"]&
\end{tikzcd}
\]
and the fact that $\cF^\bullet_R\lra \cC^\bullet_R$ is a quasi-isomorphism, implies that $\cG^\bullet_R$ is an exact complex, therefore $\gamma$ induces the map $0$ in cohomology.
Therefore in the elliptic case, the conclusion of Theorem \ref{thm:zerocoh} follows very easily.

\bigskip

To follow our program, now we continuously $R$-dualize  diagram (1),  and obtain the diagram which we call (2)
\[
\begin{tikzcd}[row sep=12]
	0\ar[r]&(\cG^\bullet_R)^\ast\ar[r]\ar[d, "\gamma^\ast"]&(\cC^\bullet_R)^\ast\ar[r]\ar[d, "\xi^\ast"]&(\cF^\bullet_R)^\ast\ar[r]&0\\
	&(\cD^\bullet_R)^\ast\ar[r, equal]&(\cD^\bullet_R)^\ast&&
\end{tikzcd}
\]

For every open affinoid $U\subset \cX_r$ over wich we have a section $s:U\lra \cT_{({\rm H}^\#_{\bE})^\vee}$, we denote $R:=\cO_U(U)$ and base change the complex $(\ast)_N$ to $R$ over $\fK$. We recall that we have a natural abelian Lie-subalgebra $\fn_U\subset \fg\otimes_{\fK}R$.

We write $\fn_U\cong \fg/\fb$ and by using the Poincar\'e-Birkhof-Witt theorem, we write  for any $U(\fb)$-module $M$: $U(\fg)\otimes_{U(\fb)}M\cong U(\fn_U)\otimes_{R}M$, where $U(\fn_U)={\rm Sym}_R(\fn_U)\cong R[t]$, as $R$-algebras, where $t$ is a variable and the isomorphism is given by sending $u=u^-+au^++bH\to t$. Here $u$ is a generator of $\fn_U$ over $R$, where $a,b\in R$..

From Lemma \ref{lemma:torsor}, we consider diagram (2) for every $U$ as above, and glue the sheaves obtained locally on every $U$, using the results of section \ref{sec:liealgebra}. In fact, in this case everything is clear by inspection. For example, we see immediately that the local pieces of the complex $\cF^\bullet_U: [U(\fn_U)\otimes_{\cO_U}\cO_U(-k-2)\lra 0]$, where $U\subset \cX_r$ is an open affinoid over which there is a splitting of the torsor, are constant with respect to $U$. By continuously dualizing it, we obtain $\bigl(\cF_U^\bullet\bigr)^\ast: [0\lra \cP_{U/S}\otimes_{\cO_U}\omega^{k+2}|_U]$, which clearly glue to give 
$(\cF^\bullet)^\ast: [0\lra L(\omega^{k+2}))]$. Similarly the local pieces of the complex $(\cC^\bullet_R)^\ast$ obviously glue and therefore the local pieces of the complex $(\cG^\bullet_R)^\ast$ glue. 
We obtain the diagram of $\cO_{\cX_r}$-modules with connections:called (3):

\[
\begin{tikzcd}[row sep=12]
	0\ar[r]&\psi^\ast\bigl((\cG^\bullet)^\ast\bigr)\ar[r]\ar[d, "\psi^\ast(\gamma^\ast) "]&\colim_N L({\rm Fil}_{N+\bullet}\otimes \omega^\bullet_{\cX/S})\ar[r]\ar[d, "\xi^\ast"]&L({\rm BGG}^\bullet)\ar[r]&0\\
	&L\bigl(\WW^{\rm alg}_k\otimes\omega^\bullet_{\cX/S}\bigr)\ar[r, equal]&L\bigl(\WW^{\rm alg}_k\otimes\omega^\bullet_{\cX/S}\bigr)&&
\end{tikzcd}
\]
where now the functor $L(\ )$ denotes linearization.
We remark the following:

$\bullet$ the complex which we denoted ${\rm BGG}^\bullet$ in the above diagram is ${\rm BGG}^\bullet:=\psi^\ast\bigl((\cF^\bullet)^\ast\bigr)=[0\lra \omega^{k+2}].$ 

$\bullet$ the connection defining the differentials of the complexes $\colim_N L({\rm Fil}_{N+\bullet}\otimes \omega^\bullet_{\cX/S})$ and respectively 
$L\bigl(\WW^{\rm alg}_k\otimes\omega^\bullet_{\cX/S}\bigr)$ is $L(\nabla_k)$, where $\nabla_k$ is the connection of $\WW_k^{\rm alg}$ induced by the Gauss-Manin connection. 

$\bullet$ We recall that $\psi^\ast(\gamma^\ast)$ induces $0$ in cohomomology.
 
 \medskip

The final step in this comparison is to consider the log infinitesimal site $(\cX/S)_{\rm inf}^{\rm log}$ and the infinitesimal topos $(\cX)_{\rm inf}^{\rm log}$ on it and the morphism of topoi: 
$\pi_\ast: (\cX)_{\rm inf}^{\rm log}\lra \cX^{\rm an}$, see section \ref{sec:outline} and section \ref{sec:inf}. We consider the linearized sheaves with integrable connection on $\cX$, think of them as crystals in $(\cX)_{\rm inf}^{\rm log}$ and apply $\pi_\ast$ to the diagram $(3)$. Several things happen, namely:

By applying lemma \ref{lemma:uastcolim}, we obtain
$$
\bullet\ \pi_\ast\bigl(\colim_{N+\bullet} L\bigl({\rm Fil}_{N+\bullet}\otimes \omega^\bullet_{\cX/S})\bigr)\cong \colim_N {\rm Fil}_{N+\bullet}\otimes \omega^\bullet_{\cX/S}\cong  \WW_k\otimes\omega^\bullet_{\cX/S} \cong\pi_\ast\bigl(L(\WW_k^{\rm alg}\otimes \omega^\bullet_{\cX/S})\bigr)
$$

$\bullet$ $\pi_\ast\bigl(\gamma^\ast\bigr)$ induces $0$ in cohomology, as follows from lemma \ref{lemma:ssequence}.

Therefore, we obtain an exact sequence of complexes on $\cX$ (we recall that the top row in diagram $(3)$ is a split exact sequence; alternatively, one can use that $R^i\pi_\ast\bigl(L(\cM)\bigr)=0$ for all $i>0$ and $\cM$ a sheaf of $\cO_{\cX}$-modules in $(\cX)_{\rm inf}^{\rm log}$):
$$
0\lra \cM^\bullet\stackrel{\rho}{\lra} \WW_k^{\rm alg}\otimes\omega^{\bullet}_{\cX/S}\stackrel{\theta}{\lra} {\rm BGG}^\bullet\lra 0
$$ 
where $\cM^\bullet$ is the complex $\pi_\ast\bigl(\pi_{GM}^\ast((\cG^\bullet)^\ast)\bigr)$ and $\rho:=\pi_\ast\bigl(\gamma^\ast\bigr)$, and we know it induces $0$ on cohomology. It follows that $\theta$ is a quasi-isomorphism of complexes and we are done.

\begin{remark} 
\label{remark:elliptic}
We have a few comments to the above result:

1) The map $\theta: \WW_k^{\rm alg}\otimes\omega^\bullet_{\cX/S}\lra {\rm BGG}^\bullet$ is a differential operator, so only 
$B$-linear, not $\cO_{\cX}$-linear. But this projection is split naturally as follows:
\[
\begin{tikzcd}[row sep=10]
	\WW_k^{\rm alg}\ar[r, "\nabla_k"]&\WW_k^{\rm alg}\otimes\omega^1_{\cX/S}\\
	0\arrow[u, phantom, sloped, "\subset"]\ar[r]&\omega^{k+2}\arrow[u, phantom, sloped, "\subset"]
\end{tikzcd}
\]
and the inclusion $\omega^{k+2}\hookrightarrow \WW^{\rm alg}_k\otimes\omega^1_{\cX/S}$ is of course $\cO_{\cX}$-linear.

2) In order for the reader to appreciate the subtlety of the set-up, let us remark that before applying $\pi_\ast$, the situation is complicated: the natural inclusion of complexes:
\[
\begin{tikzcd}[row sep=10]
L(\WW_k^{\rm alg})\ar[r, "L(\nabla_k)"]&L(\WW_k^{\rm alg}\otimes\omega^1_{\cX/S})\\
	0\arrow[u, phantom, sloped, "\subset"]\ar[r]&L(\omega^{k+2})\arrow[u, phantom, sloped, "\subset"]
\end{tikzcd}
\]
cannot be a quasi-isomorphism as ${\rm Ker}(L(\nabla_k))=\WW_k^{\rm alg}\neq 0$.
Therefore, only after applying $\pi_\ast$ the complexes $(\cC^\bullet)^\ast$ and $(\cD^\bullet)^\ast$ become isomorphic and everything works out. In fact, for this simple situation, there is another, direct proof of the quasi-isomorphism of the complexes in the diagram of remark 2) above, see \cite{andreatta_iovita}.
\end{remark}

\bigskip
 
\section{The log infinitesimal site; linearization and de-linearization of complexes.}
\label{sec:inf}

In this section we work with locally noetherian fine and saturated (fs for short) log adic spaces, as introduced in Definition 2.2.2 of \cite{diao_lan_liu_zhu}. We fix a log smooth and separated morphism $f: X\lra S$, as in Definition 3.1.1 of \cite{diao_lan_liu_zhu} and we present the basics of the theory of the log infinitesimal site associated to it. We use the following convention: if $X$ is a log adic space, $\uX$ denotes the underlying adic space. Moreover we fix for the rest of the article $\uS={\rm Spa}(B, B^+)$ an affinoid over 
${\rm Spa}(K, \cO_K)$ and $S$ has trivial log structure. For simplicity of notation, we will denote by $\cO_X$ the structure sheaf $\cO_{\uX}$.
 This site is denoted $(X/S)_{\rm inf}^{\rm log}$. Moreover we need a theory of crystals on this site and of linearization and de-lineraization of such crystals.

The site (over a different base) and with trivial  log structure is defined in \cite{guo} but in that article the author works with coherent crystals and avoids linearization of such crystals. We need a theory of crystals which are more general than coherent ones, but they are {\bf not complete}. We recall that in \cite{conrad} such a category of $\cO_X$-modules is defined and called the category of quasi-coherent $\cO_X$-modules. It is the category of $\cO_X$-modules which are locally filtered colimits of coherent sheaves. We recall that now there are other, more sophisticated definitions of ``quasi-coherent" $\cO_X$-modules on an adic space $\uX$, (see \cite{clausen_scholze}, \cite{andreychev}, \cite{snoor}), but for us the above definition suffices. We do not need sheaves which are acyclic on affinoids, only a category of sheaves which are preserved by pullback of morphisms between adic spaces, pull-back defined by the use of the {\bf algebraic tensor product}, versus the completed one. 

This definition of quasi-coherent modules on adic spaces allows us to define the category of 
quasi-coherent crystals on the log infinitesimal site of a log adic space $X$ over $S$, their linearization and de-linearization.

 \subsection{The log infinitesimal site.}
 \label{sec:infinitesimal}
 
  We start by recalling the main definitions.
   
 
 We define the (small) log infinitesimal site of $f:X\lra S$ following \cite{chiarellotto_fornasiero} and \cite{guo}, Definition 2.1.2.: .
 \begin{definition}
 	The \textbf{log infinitesimal site of $f:X\lra S$} is the category, denoted by $\XSinf$, having the following objects and arrows:
 	\begin{itemize}
 		\item The objects are pairs $(U,T)$ where $U\subset X$ is an open immersion of log adic spaces over $S$, (as $X\to S$ is log smooth, this is equivalent to saying that $\uU$ is an open of $\uX$ with log structure induced from the one on $X$), equipped with a nilpotent, exact closed immersion over $S$, $g:U\rightarrow T$;
 		\item The arrows from $(U_1,T_1)$ to $(U_2,T_2)$ are pairs of $S$-morphisms $(\alpha,\beta):(U_1,T_1)\rightarrow (U_2,T_2)$ s.t. $\alpha$ is an open immersion and the following diagram commutes
 		\begin{tikzcd}
 			& X\\
 			U_1\ar[d, "g_1"]\ar[ru, "i_1"]\ar[r, "\alpha"] & U_2 \ar[u, "i_2"]\ar[d, "g_2"]\\
 			T_1\ar[r, "\beta"]& T_2
 		\end{tikzcd}.
 	\end{itemize}
 \end{definition}
 
 We have 
 
 \begin{lemma}[\cite{guo}, Lemma 2.1.4, \cite{diao_lan_liu_zhu}, Proposition 2.3.27] \label{InfCathegoriesHaveFiberProductsAndEqualizers}
 	The site $\XSinf$ has equalizers and fiber products.
 \end{lemma}

 Now we define coverings on the category $\XSinf$ and the associated topos.
 
 \begin{definition}
 	A \textbf{covering in the log infinitesimal site} is a collection of morphisms $\{(\alpha_i,\beta_i): (U_i,T_i)\rightarrow (U,T)\}_{i\in I}$ in the category $\XSinf$ s.t. the collections $\{\alpha_i\}_{i\in I}$ and $\{\beta_i\}_{i\in I}$ are open coverings.
 \end{definition}

The coverings on $\XSinf$ describe a topology. We denote $\Xinf$ the (small) log infinitesimal topos, i.e. the category of sheaves of sets on $\XSinf$.

\subsubsection{Description of Pre-Sheaves and Sheaves in the log infinitesimal topos}

In this section we give a description of the (small) infinitesimal topos. 

We denote an object $g=(U,T)\in  \XSinf$ as the morphism $g:U\rightarrow T$ attached to it, moreover we denote $U=:s(g)$ and $T=:t(g)$ where $s$ stands for source and $t$ stands for target. For a morphism $\gamma: g_1\rightarrow g_2$ we denote
\[
\gamma_s: s(g_1)\rightarrow s(g_2)\quad \text{and}\quad \gamma_t: t(g_1)\rightarrow t(g_2)
\]
the morphisms between sources and between targets.

\begin{lemma}\label{lemmaPShinf}
	The data of a presheaf $F$ over $\XSinf$ is equivalent to:
	\begin{itemize}
		\item a collection of presheaves $\{F_{g}\}_{g\in \XSinf}$, where $F_{g}$ is a presheaf over $t(g)$ for each $g\in \XSinf$;
		\item a collection of morphisms $\{\phi_{\gamma}: \gamma_t^{-1} F_{g_2}\rightarrow F_{g_1}\}_{\gamma: g_1\rightarrow g_2}$, where $\gamma$ varies between all the morphisms in $\XSinf$.
	\end{itemize}
	with the following properties:
	\begin{enumerate}
		\item if $\gamma=\gamma_2\circ\gamma_1$, then $\phi_{\gamma_1}\circ \left(\gamma_{1,t}^{-1} \phi_{\gamma_2} \right)= \phi_{\gamma}$, i.e.
		\[
		\begin{tikzcd}
			\gamma_{1,t}^{-1} \left(\gamma_{2,t}^{-1} F_{g_3}\right) \ar[d, "\cong"]\ar[r, "\gamma_{1,t}^{-1} \phi_{\gamma_2}"]& \gamma_{1,t}^{-1} F_{g_2} \ar[r, "\phi_{\gamma_1}"] & F_{g_1}\\
			(\gamma_t)^{-1} F_{g_3}\ar[rru, "\phi_{\gamma}"]& &
		\end{tikzcd}
		\]
		is a commutative diagram, where $\gamma_1: g_1\rightarrow g_2$ and $\gamma_2: g_2\rightarrow g_3$;
		\item $\phi_{\text{id}_g}=\text{id}_{F_g}$.
	\end{enumerate}
	
	Moreover $F$ is a sheaf on the log infinitesimal site iff each $F_g$ is a sheaf on $t(g)$ and if $\gamma_t$ is an open immersion, then $\phi_\gamma$ is an isomorphism.
\end{lemma}

\subsubsection{Envelopes}\label{SectionEnvelop}
 
 We observe that, in general, the category $\XSinf$  does not have a final object, in such cases it is usual that one defines global sections  of a sheaf $F$ by
\[
\Gamma\left(\XSinf, F\right):=\text{Hom}_{\Xinf}\left(\uno, F\right)
\]
where $\uno$ is the final object in $\Xinf$, i.e. it is the sheafification of the presheaf $\uno^{pre}$, where
\[
\uno^{pre}(U,T)= \{\ast\}\quad \text{for each }(U,T)\in \XSinf.
\]
By $\{\ast\}$ we denoted a set with one element, $\ast$.
In this section we define the envelope of a closed immersion and we describe the basic properties of the envelope. The envelope construction plays a central role in our theory since it provides a hyper-covering of the sheaf $\uno$, hence a way to compute global sections of any given sheaf.

Let us denote by $\cC$ the category with objects closed immersions of log adic spaces over $S$ (\cite{diao_lan_liu_zhu}, Definition 2.2.23.) $X\hookrightarrow Y$, with morphisms pairs $(\alpha, \beta):(X\hookrightarrow Y)\to (X'\hookrightarrow Y')$, with
$\alpha: X\lra X'$ and $\beta: Y\lra Y'$ such that the obvious square is commutative. We also have the full subcategory $\cC^{\rm ex}$ of $\cC$ whose objects consist of exact closed immersions of log adic spaces over $S$ (\cite{diao_lan_liu_zhu}, Definition 2.2.23 and Definition 2.2.2.) We have

\begin{proposition}
\label{prop:exactification}
The inclusion functor $\iota:\cC^{\rm ex}\subset \cC$ has a right adjoint $(X\hookrightarrow Y)\to (X\hookrightarrow Y)^{\rm ex}\cong (X\hookrightarrow Y^{\rm ex})$.
\end{proposition}

\begin{proof}
Let $\iota:X\to Y$ be a closed immersion of log adic spaces over $S$. Thanks to \cite{diao_lan_liu_zhu}, Proposition 2.3.19 and Proposition 2.3.22, locally on $X$ and $Y$ we have a chart $h:P\to Q$
of $\iota$ with $P,Q$ finite saturated monoids. Let $h^{\rm gp}: P^{\rm gp}\to Q^{\rm gp}$ be the natural group homomorphism and set $P^{\rm ex}:=(h^{\rm gp})^{-1}(Q)\to Q$.
We define $\uY^{\rm ex}:=\uY\times_{S\langle P\rangle}S\langle P^{\rm ex}\rangle$, with the log-structure defined by the  pre-log structure $P^{\rm ex}\to \cO_{\uY^{\rm ex}}$. See \cite{diao_lan_liu_zhu}, Example 2.2.19, for the notation $S\langle P\rangle$ and $S\langle P^{\rm ex}\rangle$ and Proposition 2.3.27 shows that the fiber product above exists in the category of log adic spaces over $S$.  By construction the morphism $\iota$ defines an exact closed immersion 
$$
\iota^{\rm ex}: X\cong X\times_{S\langle Q\rangle}S\langle Q\rangle\to Y^{\rm ex}:=Y\times_{S\langle P\rangle}S\langle P^{\rm ex}\rangle.
$$
 with chart $P^{\rm ex}\to Q$. In general, these local constructions glue and define $\iota^{\rm ex}:X\to Y^{\rm ex}$. 
 
\end{proof}

Let $j:X\rightarrow Y$ be a closed immersion of log adic spaces over $S$, not necessarily exact or nilpotent; we apply Proposition \ref{prop:exactification} and consider the exact closed immersion $j^{\rm ex}:X \rightarrow Y^{\rm ex}$, then  there is a coherent ideal $J\subset \cO_{\uY^{\rm ex}}$ s.t. $(j^{\rm ex})^{-1}\cO_{\uY^{\rm ex}}/ J \cong \cO_X$ via the morphism $j^{\rm ex}$. Let $P^n_{X/Y}$ be the adic space having as topological space the topological space of $X$ and with ring of definition $\cO_{\uY^{\rm ex}}/ J^{n+1}$ and log structure induced from the log structure of $Y^{\rm ex}$. Observe that $X$ is isomorphic (via $j$) to $P^0_{X/Y}$ as log adic spaces, moreover the morphism $X\rightarrow P^n_{X/Y}$ is an exact closed immersion with nilpotent ideal of order $n+1$, hence $(X, P^n_{X/Y})$ is an object in the site $\XSinf$.
\begin{definition}
	The object 
	\[
	D_X^n(Y):=(X, P^n_{X/Y})\in \XSinf
	\] is called the \textbf{$n$-th log envelope} of $X$ in $Y$.
\end{definition}
We have already seen that $P^0_{X/Y}=X$, moreover if $X=Y$ and $j=\text{id}_X$, then $P^n_{X/Y}=X$ for any $n\in \N$ and $D^n_X(X)=(X,X)\in \XSinf$. An interesting example is given by taking the log $n$-th envelope of the diagonal morphism
\[
\Delta_k: X\longrightarrow X^{\times^{k+1}_S}.
\]
We recall that $X$ is log smooth and separated over $S$, hence $\Delta_k$ is a closed immersion for all $k\in \N$. We denote $P^n_X(k):=P^n_{X/X^{\times^{k+1}_S}}$ and $D^n_X(k):=\left(X, P^n_X(k)\right)$. Moreover $I(k)\subset \cO_{\underline{P}^n_X(k)}$ will denote the ideal defined by the closed immersion of $X\subset P^n_X(k)$ and
$\mathcal{P}^n_X(k):= \Delta_k^{-1} \left(\cO_{\underline{P}^n_X(k)}\right)$ will
denote the structural sheaf of rings of $P^n_{X}$ (we see it over $X$ since the two topological spaces are equal). For each $i\in \N_{\leqslant k}$ let us denote with $p^n_i(k)$ the morphism
\[
\begin{tikzcd}
	P^n_{X}(k) \ar[r, "	p^n_i(k)"]\ar[d, ]& P^n_{X}(0)\ar[d, equal]\\
	X^{\times_S^{k+1}}\ar[r, "p_i"]& X
\end{tikzcd}
\]
induced by the $i$-th projection $p_i$.
If $k=1$ we simplify the notation:
\[
I:= I(1),\quad\quad P^n_{X}:=P^n_{X}(1)\quad\text{and}\quad \mathcal{P}^n_{X}:=\mathcal{P}^n_{X}(1).
\]
For $i,j\in \N_{\leqslant 2}$ we denote $p_{i,j}^n(2)$ the morphism
\[
\begin{tikzcd}
	P^n_{X}(2) \ar[r, "	p_{i,j}^n(2)"]\ar[d, ]& P^n_X\ar[d, hook]\\
	X\times_S X\times_S X\ar[r, "p_{i,j}"]& X\times_S X
\end{tikzcd}
\]
induced via the projection on the $i-th$ and $j-th$ coordinates (where the order matters). 

\medskip

At this point it is worth trying to understand $\cP^n_X$ and especially $\cP_X:=\lim_n \cP^n_X$, for $f:X:\to S$ a log smooth and separated morphism of log adic spaces, with $S={\rm Spa}(B,B^+)$  (in this Lemma we do not assume that the log structure of $S$ is trivial, only that both $X$ and $S$ have finite separated (for short fs) log structures.) We first have 

\begin{lemma} 
\label{lemma:logsmooth}
Given  $f:X\rightarrow S$ as above, there exist

\item[(i)] an affinoid open covering $(U_i)_{i\in I}$ of $\uX$,

\item[(ii)] fine and saturated monoids $M_i$ with the property that the order of the torsion subgroup of $M_i^{\rm gp}$ is invertible in $U_i$, for all $i\in I$, 

\item[(iii)] morphisms of log adic spaces $U_i\to S\langle M_i\rangle $ (see Example 2.2.19 of \cite{diao_lan_liu_zhu} for the notation) whose underlying morphism of adic spaces is \'etale, for all $i\in I$.

\bigskip
Then, for every $n\ge 0$ we have $\cP_{U_i}^n\cong  \cO_{U_i}[M_i^{\rm gp}]/J_i^{n+1}$ with $J_i$ the augmentation ideal, i.e., the ideal generated by $m_i-1$ for varying $m_i\in M_i^{\rm gp}$. In particular,   
$\omega^1_{U_i/S} \cong \cO_{U_i}\otimes_{\Z} M_i^{\rm gp}\cong J_i/J_i^2$.  
\end{lemma}
\begin{proof}
	The first claim follows from Definition 3.1.1 of  \cite{diao_lan_liu_zhu}. By definition $\cP_{U_i}^n$ is the structure sheaf of the $n+1$-infinitesimal neighbourhood of $U_i\subset (U_i\times_S U_i)^{\rm ex}$. It follows from the proof of Proposition \ref{prop:exactification}, that working locally on $X/S$ we have: $U_i\to S$ is a log smooth morphism of affinoid log adic spaces, where
$U_i$ has a fine chart $M_i\to A_i$, and $U_i={\rm Spa}(A_i,A_i^+)$. 

Let $i\in I$ and from now on, during this proof, we denote $U:=U_i\to S$ and $A:=A_i$, $S:={\rm Spa}(B,B^+)$$M:=M_i\to A$ and $N\to B$ be the log structures and a morphism of monoids $N\to M$ providing a chart for $U\to S$. Let $\delta\colon U\to S\langle M\rangle$ be the induced \'etale map.  Without loss of generality we may assume that $\omega^1_{U/S}$ if a free $A$-module.

We have, $(M\oplus_N M)^{\rm ex}=M\oplus_N (M)^{\rm gp}$. Using $\delta$ we get  \'etale maps
$\gamma= 1\times \delta\colon U\times_S U\to U\times_{S\langle N\rangle} S\langle M\rangle$ and 
$
\gamma^{\rm ex}\colon (U\times_S U)^{\rm ex}\longrightarrow U \times_{S\langle N\rangle} S\langle M^{\rm gp} \rangle$, given by identifying $(U\times_S U)^{\rm ex}\cong (U\times_SU)\times_{S\langle M\oplus_N M\rangle}S\langle M\oplus_N M^{\rm gp}\rangle$, identifying the latter with $U \times_S \bigl(U \times_{S\langle N\rangle} S\langle M^{\rm gp} \rangle \bigr)$  amd projecting onto $ U \times_{S\langle N\rangle} S\langle M^{\rm gp} \rangle$. 
We have a natural, commutative diagram

$$
\begin{array}{cccccccc}
U & \stackrel{\Delta}{\longrightarrow}   & U\times_S U\cr
\Vert & & \big\downarrow            1\times \gamma\cr
U &\stackrel{\Delta'}{\longrightarrow} &  U\times_{S\langle N\rangle} S\langle M\rangle
\end{array}
$$
where $\Delta$ is the diagonal and $\Delta'$ the augmentation map.
This induces a commutative diagram
$$
\begin{array}{cccccccc}
U & \stackrel{\Delta^{\rm ex}}{\longrightarrow} &   (U\times_S U)^{\rm ex}\cr 
\Vert & & \big\downarrow \gamma^{\rm ex}\cr
U & \stackrel{\Delta^{',\rm gp}}{\longrightarrow} &   U \times_{S\langle N\rangle} S\langle M^{\rm gp} \rangle,\cr
\end{array}
$$
where $\Delta^{',\rm gp}$ is the augmentation map. Consider the following diagram, with squares that are cartesian by the above discussion, possibly except for the most left one:
$$\begin{matrix}
U & \longrightarrow & U' & \longrightarrow &(U\times_SU)^{\rm ex} & \longrightarrow & U\times_S U & \stackrel{\pi_2}{\longrightarrow}  & U \cr \Vert & &
\alpha\big\downarrow &&\gamma^{\rm ex}\big\downarrow  & & \gamma \big\downarrow & & \big\downarrow\delta\cr U& =& 
U &\stackrel{\Delta^{',\rm gp}}{\longrightarrow}  &U \times_{S\langle N\rangle} S\langle M^{\rm gp} \rangle & \longrightarrow &  U \times_{S\langle N\rangle} S\langle M \rangle & \stackrel{\pi_2}{\longrightarrow}  & S\langle M \rangle ,\cr 
\end{matrix}$$where $\pi_2$ is the second projection. The map $\alpha$ is \'etale and $\Delta^{\rm ex}$ factors through $U'$ so that $U'$ decomposes as $U'=U\amalg U''$ with $U''$ and $U$ open and closed in $U'$.

For every positive integer $n$ consider the $n$-th infinitesimal neighborhoods 

\begin{itemize}
\item[i.] $U'\subset {\cal P}^n_{U' \subset (U\times_SU)^{\rm ex} }$ of $U'$ in $(U\times_SU)^{\rm ex} $, 

\item[ii.] $U\subset {\cal P}^n_{U\subset (U\times_SU)^{\rm ex} }$  of $U$ in $(U\times_SU)^{\rm ex} $ 

\item[iii.] $U\subset {\cal P}^n_{U \subset U \times_{S\langle N\rangle} S\langle M^{\rm gp} \rangle}$ of $U$ in $U \times_{S\langle N\rangle} S\langle M^{\rm gp} \rangle$.

\item[iv.] $U''\subset{\cal P}^n_{U'' \subset (U\times_SU)^{\rm ex} }$ of $U''$ in $(U\times_SU)^{\rm ex} $,

\end{itemize}

{\bf Lemma:} The map ${\cal P}^n_{U \subset (U\times_SU)^{\rm ex} }\to {\cal P}^n_{U \subset U \times_{S\langle N\rangle} S\langle M^{\rm gp} \rangle}$ is an isomorphism.

\smallskip

\begin{proof} As $U'\cong U \amalg U''$ we have a decomposition  $${\cal P}^n_{U' \subset (U\times_SU)^{\rm ex} }\cong {\cal P}^n_{U \subset (U\times_SU)^{\rm ex} }\amalg {\cal P}^n_{U'' \subset (U\times_SU)^{\rm ex} }.$$The morphism ${\cal P}^n_{U' \subset (U\times_SU)^{\rm ex} }\to {\cal P}^n_{U \subset U \times_{S\langle N\rangle} S\langle M^{\rm gp} \rangle}$ is \'etale by the cartesian diagram above. Hence the map  ${\cal P}^n_{U \subset (U\times_SU)^{\rm ex} }\to {\cal P}^n_{U \subset U \times_{S\langle N\rangle} S\langle M^{\rm gp} \rangle}$ is \'etale as it is the composite of an open immersion and an \'etale map. Since it induces the identity on $U$, it is an isomorphism.

\end{proof}

 \

The conclusion of Lemma \ref{lemma:logsmooth} follows.




	
\end{proof}

We have the following important consequence of Lemma \ref{lemma:logsmooth}.

\begin{corollary}\label{rmk:coordinates}
We use the notations of Lemma \ref{lemma:logsmooth}, i.e. let $f:X\to S$ be a log smooth and separated morphism and let $(U_i)_{i\in I}$ be the affinoid open cover there.
Let $M_i$ be the fine, saturated monoids such that the torsion subgroup of $M_i^{\rm gp}$ is finite, invertible in $U_i$, for every $i\in I$. Let $m_{i,1}, m_{i,2},...,m_{i,d_i}$ be a basis of the 
torsion free quotient of $M_i^{\rm gp}$. Then $\xi_{i,j}:=m_{i,j}-1\in J_i$ for $1\le j\le d_i$ have the properties:

a) $\xi_{i,j}({\rm mod}\ J_i^2)={\rm dlog}(m_{i,j})$ for $1\le j\le d_i$ is an $O_{U_i}$-basis of $\omega_{U_i/S}^1$, for all $i\in I$.

b) $\cO_{U_i}[M_i^{\rm gp}]/J_i^{n+1}\cong \cO_{U_i}[(M_i^{\rm gp})^{\rm tor}][\xi_{i,1},...,\xi_{i,d_i}]^{{\rm deg}\le n}$, and $\cO_{U_i}\subset \cO_{U_i}[(M_i^{\rm gp})^{\rm tors}]$ is a finite \'etale extension.

Moreover $\cP_{U_i}:=\lim_n\cP_{U_i/S}^n\cong \cO_{U_i}[[\xi_{i,1},...,\xi_{i, d_i}]]$.

In other words, $\xi_{i,1},...,\xi_{i,d_i}$ are local coordinates for the morphism $f$.
\end{corollary}

\begin{proof}
By Lemma \ref{lemma:logsmooth} it is enough to show b). Let, as in the proof of Lemma \ref{lemma:logsmooth}, $i\in I$ and denote
$U:=U_i:={\rm Spa}(A,A^+)$, $M:=M_i^{\rm gp}$, $J:=J_i$, $m_1,...,m_d$ a basis of $M/M^{\rm tor}$ and $\xi_1:=m_1-1,...,\xi_d:=m_d-1$.

Then $A[M]/J^{n+1}\cong A[M^{\rm tor}][m_1,...,m_d, m_1^{-1},...,m_d^{-1}]/J^{n+1}=A[M^{\rm tor}][\xi_1,...,\xi_d]/J^{n+1}$ as
 for all $1\le j\le d$, $m_j=1+\xi_j$ and $\xi_j\in J$ which is nilpotent in our ring, therefore $m_j$ is a unit in $A[M]/J^{n+1}$. 

As $M^{\rm tor}$ is finite of order invertible in $A$ (this in fact is automatic as $A$ has characteristic $0$), $A[M^{\rm tor}]/A$ is etale, 
therefore $\xi_j({\rm mod }J^2)=:{\rm dlog}(m_j)$, $j=1,...,d$ is a basis of $\omega^1_{U/S}$. This proves the corollary.

\end{proof}

Due to Corollary \ref{rmk:coordinates}, the theory now follows as in \cite{berthelot_ogus}.

Coming back to a general closed immersion $j:X\rightarrow Y$, defined by the ideal $J\subset \cO_{\uY^{\rm ex}}$, then for any $n\leqslant m$ there is a map (identity between topological spaces and projection by $J^{n+1}$ as morphism between the sheaves)
\[
P^n_{X/Y}\longrightarrow P^m_{X/Y}.
\]
We obtain an inductive system $\{D^n_X(Y)\}_{n\in\N}$; we would like to define the log envelope of $X$ in $Y$ as the colimit of the $D^n_X(Y)$, but in general this colimit does not exist in our site. Then we consider the Yoneda embedding and we compute the colimit in the log infinitesimal topos.
\begin{definition}
	The \textbf{log envelope} of $X$ in $Y$ is
	\[
	D_X(Y):=\colim_{n\in \N}^{\Xinf}  h_{D^n_X(Y)}\  \in \Xinf
	\]
	where $h_{D^n_X(Y)}:= \text{Hom}_{\XSinf} (- , D^n_X(Y))$.
\end{definition}
The overscript in the colimit recalls the category in which the colimit is computed. 
We provide  another description of the envelop $D_X(Y)$.

Let $g:U\rightarrow T$ be an object in the log infinitesimal site, we define
\[
\text{Hom}_{\XSinf}((U,T), (X,Y)):=\{\beta\in\text{Hom}_{S}(T,Y)\mid  
\begin{tikzcd}
	U\ar[r, hook]\ar[d, "g"]& X\ar[d, "j"]\\
	T\ar[r, "\beta"]& Y
\end{tikzcd} \text{ is commutative }
\},
\]

One can verify that this functor (with obvious restriction maps) is a sheaf in the log infinitesimal topos, we denote this sheaf by
\[
\text{Hom}_{\XSinf}(\ -\ , (X,Y)).
\]
We remark that the pair $(X,Y)$ is not in general in the site, as we do not require that the morphism $j$ be nilpotent or exact. The following Lemma is proved as Lemma 2.2.5 in \cite{guo}.

\begin{lemma} \label{lemmaD_X(Y)=Hom (-,Y)}
	If $j: X\rightarrow Y$ is a closed immersion, then 
	\[
	D_X(Y)\cong {\rm Hom}_{\XSinf}(\ -\ , (X,Y))
	\]
\end{lemma}

\subsubsection{The global sections of a log infinitesimal sheaf.}

Let $\Delta:X\to X\times_SX$ be the diagonal closed immersion of log adic spaces over $S$, and we use the notations of the previous subsection.
We observe that the two projection maps $p^{n}_0(1), p^{n}_1(1): P^n_{X}(1)\rightrightarrows X$ make the following diagram commute, for any $n_1\leqslant n_2$ and $i=0,1$:
\[
\begin{tikzcd}
	P_{X}^{n_1} \ar[r, "p^{n_1}_i(1)"] \ar[d]& X\ar[d] \\
	P_{X}^{n_2} \ar[r, "p^{n_2}_i(1)"] & X \\
\end{tikzcd},
\]
hence we get two induced projection maps
\[
D_X(1):=D_X(X\times_S X) \rightrightarrows (X,X).
\]
\begin{lemma}\label{lemma1isCoEq}
	The canonical morphism
	\[
	D_X(0)={\rm Hom}_{\XSinf}\left(-, (X,X)\right)\rightarrow \uno
	\]
	is an epimorphism of sheaves, in particular
	\[
	\uno=\CoEq^{\Xinf}\left(D_X(1)\rightrightarrows D_X(0)\right),
	\]
	where the two morphisms are the two projections.
\end{lemma}
\begin{proof}
	Showing that a morphism of sheaves is an epimorphism requires to show (by definition) that for any object $(U,T)\in \XSinf$ there is a covering $\{(U_i,T_i)\}_{i\in I}$ s.t. the morphism between the $(U_i,T_i)$-sections is surjective for all $i\in I$.
	
	Let $(U,T)$ be an object in $\XSinf$, then there is a covering $\{(U_i,T_i)\}_{i\in I}$ s.t. for each $i\in I$ the map $U_i\rightarrow T_i$ is nilpotent of finite order $n_i$ and we have a diagram
	\[
	\begin{tikzcd}
		U_i\ar[r]\ar[d]& X\ar[d, equal]\\
		T_i& X
	\end{tikzcd}.
	\]
	By log smoothness of $X$ over $S$ (\cite{diao_lan_liu_zhu}, Proposition 3.3.16), we get that there is a map
	\[
	g_i:T_i\longrightarrow X
	\]
	s.t. the diagram is still commutative, Lemma  \ref{lemmaD_X(Y)=Hom (-,Y)} implies that we have an element $g_i\in D_X(0)(U_i,T_i)$, therefore the unique morphism $\uno(U_i,T_i)=\{\ast\}$ is surjective for each $i\in I$. We get that $D_X(Y)\rightarrow\uno$ is an epimorphism of sheaves.
	
	Observe that by lemma \ref{lemmaD_X(Y)=Hom (-,Y)}
	\begin{multline*}
		D_X(1):=D_X(X\times_S X)\cong \text{Hom}_{\XSinf}\left( \ -\ ,X\times_S X\right)\\
		\cong \text{Hom}_{\XSinf}\left( \ -\ ,X\right)^2
		\cong  D_X(0)\times^{\Xinf}_{\uno} D_X(0),
	\end{multline*}
	where we used that the product in a category with a final object is the fibered product in this category along the final object. Since $D_X(0)\rightarrow \uno$ is an epimorphism, then
	\[
	\uno \cong \CoEq^{\Xinf} \left(D_X(1)\rightrightarrows D_X(0)\right)
	\]
	and we conclude the proof.
\end{proof}

As we anticipated at the beginning of this subsection we are now able to define global sections of a sheaf and to give a description of these global sections in terms of envelopes.

\begin{definition}
	Given a sheaf $F$ in the topos $\Xinf$, we define the \textbf{global sections of }$F$ as the set
	\[
	\Gamma\left( \XSinf, F\right):=\text{Hom}_{\Xinf}\left(\uno,F{}\right)
	\]
\end{definition}

\begin{proposition}\label{PropGlobalSections}
	Let $F$ be a sheaf in the topos $\Xinf$, there is a canonical isomorphism
	\[
	\Gamma\left( \XSinf, F\right)\cong \lim_{n\in \N} \Eq^{Set}\left( F(X,X)\rightrightarrows F(X, P^n_X)\right)
	\]
	where the morphisms in the equalizer are induced by the two projections maps.
	
\end{proposition}
\begin{proof}
	We have the following chain of natural isomorphisms
	\begin{multline*}
		\Gamma\left( \Xinf, F\right)=\text{Hom}_{\Xinf}\left(\uno,F\right)\\
		\cong \text{Hom}_{\Xinf}\left(\CoEq^{\Xinf}\left(D_X(X\times_S X)\rightrightarrows D_X(0) \right),F\right)\\
		\cong 
		\text{Hom}_{\Xinf}\left(\colim_{n\in \N}\CoEq^{\Xinf}\left(h_{(X, P^n_X)}\rightrightarrows h_{(X,X)}\right),F\right)\\
		\cong \lim_{n\in \N}\Eq^{Set} \left(
		\text{Hom}_{\Xinf}\left(h_{(X, X)}, F \right)\rightrightarrows
		\text{Hom}_{\Xinf}\left(h_{(X, P^n_X)}, F\right)\right)\\
		\cong 
		\lim_{n\in \N} \Eq^{Set}\left( F(X,X)\rightrightarrows F(X, P^n_X)\right)
	\end{multline*}
	where the two morphisms in the equalizers are the two restriction morphisms of $F$ induced by the two projection maps \[
	(\text{id}_X,p^n_0(1)), \ (\text{id}_X,p^n_1(1)):D_X^n(1)= \left(X,P^n_{X/Y}\right)\rightrightarrows \left(X,X\right)=D_X^n(0)
	\]
\end{proof}

\subsection{Quasi-coherent crystals.}\label{SectionCrystal}
In this section we introduce the notion of a quasi-coherent crystal in the log infinitesimal topos, but first we need to define quasi-coherent 
$\cO_X$-modules, for $X$ a log adic space over $S$.

Let $f:X\rightarrow S$ be a morphism of log adic spaces, which is log smooth and separated, with $\uS={\rm Spa}(B, B^+)$ and $S$ has trivial log structure. 
We will work with the category of quasi-coherent $\cO_X$-modules ${\rm QCoh}(X)$ (see \cite{conrad}, section 2.1) defined as follows. We point out that the article \cite{conrad} is written for rigid spaces $X$ over a $p$-adic field 
$K$, but the definitions and the results of section 2.1 of \cite{conrad} which we need, hold in our setting. 

\begin{definition}\label{def:quasicoh} An $\cO_X$-module $F$ is quasi-coherent if it is locally, a colimit of coherent modules, i.e. there is an open covering 
$(U_i)_{i\in I}$ of $\uX$ and for each $i\in I$ we have $F|_{U_i}\cong \colim_{\alpha\in A_i} G_{i,\alpha}$, where $G_{i,\alpha}$ are coherent $\cO_{U_i}$-modules for every $\alpha\in A_i$, with $A_i$ an inductive ordered set, for every $i\in I$.
\end{definition}

\medskip

\begin{definition}
	Let $\cO_{X_{\rm inf}}$ be the \textbf{structure sheaf} of the topos $\Xinf$ defined by $\cO_{X_{\rm inf}, (U,T)}:= \cO_T$
	together with the restriction morphisms 
	\[
	\phi^\cO_{(\alpha,\beta)}:\beta^{-1}\cO_{T_2}\rightarrow\cO_{T_1}
	\]
	for any $(\alpha,\beta):(U_1,T_1)\rightarrow(U_2,T_2)\in \XSinf$.
\end{definition}

\begin{definition}
	An $\cO_{X_{\rm inf}}$\textbf{-module} is an abelian sheaf $F\in\Xinf$ s.t. for each $(U,T)\in \XSinf$, the sheaf $F_{(U,T)}$ is an $\cO_T$-module and for each $(\alpha,\beta):(U_1,T_1)\rightarrow(U_2,T_2)\in \XSinf$, the morphism
	\[
	\phi_{(\alpha,\beta)}:\beta^{-1}F_{(U_2,T_2)}\rightarrow F_{(U_1,T_1)}
	\]
	is $\beta^{-1}\cO_{T_2}-$linear.
	
	
\end{definition}

We observe that for an $\cO_{X_{\rm inf}}$-module $F$, the restriction morphisms induce morphisms
\[
\begin{tikzcd}[row sep=small]
	\varphi^F_{(\alpha,\beta)}:\beta^*F_{(U_2,T_2)}= \beta^{-1}F_{(U_2,T_2)}\otimes_{\beta^{-1}\cO_{T_2}} \cO_{T_1} \ar[r]& F_{(U_1,T_1)}\\
	(x\otimes a)\ar[r,maps to]& a\phi_{(\alpha,\beta)}(x)
\end{tikzcd}
\]
of $\cO_{T_1}$-modules.

\begin{definition} We say that an $\cO_{X_{\rm inf}}$-module  $F$ is {\bf coherent} if $F_{(U,T)}$ is a coherent $\cO_T$-module for every object $(U,T)$  of $\XSinf$.

We say that an $\cO_{X_{\rm inf}}$-module $F$ is {\bf quasi-coherent}, if $F_{(U,T)}$ is a quasi-coherent $\cO_T$-module for every object $(U,T)$  of $\XSinf$ (see definition  \ref{def:quasicoh}).

A \textbf{quasi-coherent crystal} in $\Xinf$ is a quasi-coherent $\cO_{X_{\rm inf}}-$module $F$ such that for every morphism $(\alpha,\beta)$ in $\XSinf$ the morphisms $\varphi^F_{(\alpha,\beta)}$ above are isomorphisms.
\end{definition}

\bigskip 

Let $F\in\Xinf$ be a quasi-coherent crystal, then the morphisms (for $i=0,1$)
\[
(\text{id}_X, p_i^n(1)): D^n_X(1)\rightrightarrows  D^n_X(0)=(X,X)
\]
induce $\mathcal{P}^n_X-$linear isomorphisms (by definition of a crystal)
\[
\epsilon_n^F:=	\left(\varphi^F_{p_{1,n}}\right)^{-1} \circ \varphi^F_{p_{0,n}}: p_0^n(1)^* F_{(X,X)} \overset{\cong}{\longrightarrow} 
F_{D_X^n(1)}\overset{\cong}{\longrightarrow} p_{1}^n(1)^* F_{(X,X)} .
\]
We remark that, as $\mathcal{P}^n_X$-modules, we have
\[
p_0^n(1)^*F_{(X,X)}=F_{(X,X)}\otimes_{\cO_X} \mathcal{P}^n_X\quad\text{and}\quad p_1^n(1)^*F_{(X,X)}= \mathcal{P}^n_X \otimes_{\cO_X}F_{(X,X)}.
\]
\begin{lemma}\label{RemarksectionPrimaDellaProp}
	If $F\in \Xinf$ is a quasi-coherent crystal and $X$ is smooth over $S$, then
	\[
	\Gamma\left(\XSinf, F\right)=\{x\in F(X,X)\mid \epsilon_n^F(p_0^n(1)^*x)= p_1^n(1)^* x\text{ for all }n\in \N\}.
	\]
\end{lemma}

\begin{proof}The remark follows from Proposition \ref{PropGlobalSections}, indeed for any $n\in \N_{>0}$
\begin{multline*}
	\{x\in F(X,X)\mid \epsilon_n^F(p_0^n(1)^*x)= p_1^n(1)^* x\} \subset \{x\in F(X,X)\mid \epsilon_{n-1}^F(p_0^{n-1}(1)^*x)= p_1^{n-1}(1)^* x\},
\end{multline*}
because $\epsilon_{n-1}^F = \epsilon_n^F$ mod $I^n$. \end{proof}

Now we would like to prove that for a quasi-coherent crystal $F$ over $X$ the global sections are given by
\begin{equation}\label{GlobalSectionFormula}
	\Gamma\left(\XSinf, F\right)=\{x\in F(X,X)\mid \epsilon^1_F(p_0^1(1)^*x)= p_1^1(1)^* x\}.
\end{equation}
i.e. if $x\in F(X,X)$ satisfies $\epsilon_n^F(p_0^n(1)^*x)= p_1^n(1)^* x$ with $n=1$, then it satisfies the condition for all $n\in \N$. In order to prove the formula above we will define differential operators of a certain (finite) order and investigate some of their properties,
in the next section.

\subsubsection{Log differential operators.}\label{SectionDiffOp}
Recall that $f:X\rightarrow S$ denotes a log smooth and separated morphism of log adic spaces. 
We will work with the category of quasi-coherent $\cO_X$-modules ${\rm QCoh}(X)$, as in definition \ref{def:quasicoh}.

In this section we recall the definitions of stratifications $\{\epsilon_n^\cF\}_{n\in \N}$ of a quasi-coherent $\cO_X$-module $\cF$ and the module of differential operators of finite orders between quasi-coherent $\cO_X$-modules $\cF,\cG$, denoted ${\rm Diff}^\bullet(\cF, \cG)$. Moreover, given a stratification we will define a ring homomorphism
\[
{\rm Diff}^\bullet\left(\cO_X,\cO_X\right)\longrightarrow {\rm Diff}^\bullet\left(\cF,\cF\right).
\]
Our main source for these results is section \S 3.2 of \cite{shiho}.  In this article the author works both with schemes with log structures and coherent sheaves or formal schemes over a complete DVR of mixed characteristics and finite residue field, but all the results hold due to our description of the log infinitesimal neighbourhoods of the diagonal of a log smooth morphism and the definition of quasi-coherent sheaves.

\begin{lemma} [Lemma 3.2.3, \cite{shiho}]
Let $X\to S$ be a log smooth and separated morphism of fs log scheme with $S$ affinoid with trivial log structure.
For every $n$, let $P^n_{X/S}$ denote the adic space with topological space the topological space of $X$ and 
structure sheaves $(\cP^n_{X/S}, \cP^{n,+}_{X/S})$, with $\cP^{n,+}_{X/S}$ the integral closure of $\cO_X^+$ in 
$\cP^n_{X/S}$. Let $p_i$, $i=1,2$ denote the composition $\displaystyle P^n_{X/S}\to X\times_SX\stackrel{\pi_i}{\lra} X$, where $\pi_i$ is the $i$-th projection. Then:

1) The log structure of $P^n_{X/S}\times_XP^n_{X/S}$ is fine.

2) The morphism $X\to P^n_{X/S}\times_XP^n_{X/S}$ is an exact closed immersion.

\end{lemma}

\begin{proof} Given the description of $\cP^n_{X/S}$ in Lemma \ref{lemma:logsmooth}, the prooof follows the proof of Lemma 3.2.3 in \cite{shiho}.
\end{proof}

\begin{remark} In fact, Lemma 3.2.3 in \cite{shiho} holds also for formal schemes, so the result applies to log adic spaces over $p$-adic fields associated to log formal schemes.

We also remark that Corollary \ref{rmk:coordinates} gives a refinement of Proposition 3.2.5 in \cite{shiho} in the case $X\to S$ is log smooth, in the sense that the local coordinates are related to the generators of $M^{\rm gp}$, with $(U,M)$ a local chart of $X\to S$. 

 \end{remark}

\begin{definition}[Definition 3.2.6, \cite{shiho}]
Let $X\to S$ be a log smooth and separated morphism of adic spaces and let $E,F$ be quasi-coherent $\cO_X$-modules. We define $\cD{\rm iff}^n(E,F):=\cH{\rm om}_{\cO_X}(\cP^n_{X/S}\otimes_{\cO_X}E, F)$
and $\cD{\rm iff}^\bullet(E,F):=\colim_n \cD{\rm iff}^n(E,F)$.

We define the $\cP^n_{X/S}$-structure of $\cD{\rm iff}^n(E,F)$ by $(af)(x):=f(ax)$, for $a$ section of $\cP^n_{X/S}$,
$f$ a section of $\cD{\rm iff}^n(E,F)$ and $x$ section of $\cP^n_{X/S}$. This induces a  $\cP_{X/S}:=\lim_n\cP^n_{X/S}$-module structure of  $\cD{\rm iff}^\bullet(E,F)$.

\end{definition}

\begin{lemma}
	For each $n,m\in \N$ we have a map of $\cO_X$-algebras (with respect to both left or right $\cO_X$-module structures)  
	\[
	\delta^{n,m}:  \mathcal{P}_X^{n+m}\longrightarrow \mathcal{P}_X^n\otimes_{{\cO_X}} \mathcal{P}_X^m, \ 
	a\otimes b\lra (a\otimes 1)\otimes (1\otimes b)
	\]
	
\end{lemma}

Thanks to this lemma we can say that a composition of two differential operators $D_1, D_2$ of order $n_1$, respectively $n_2$ is a differential operator of order $n_1+n_2$.

\begin{lemma}\label{lemma:comp}
	If $\cF,\cG,\cH$ are three $\cO_X-$modules, for each $n,m\in \N$ the composition map
	\[
	\cD{\rm iff}^n(\cF,\cG)\times \cD{\rm iff}^m(\cG,\mathcal{H}) \rightarrow \cD{\rm iff}^{n+m}(\cF,\mathcal{H})
	\]
	is defined as follows: let $D_1$ and $D_2$ be sections of $\cD{\rm iff}^n(\cF, \cG)$ respectively $\cD{\rm iff}^m(\cG,\cH)$, then $D_2\circ D_1$ is  the composition
	\[
	\begin{tikzcd}
		\mathcal{P}_X^{n+m}\otimes_{\cO_X} \cF\  \ar[r, "\delta^{n,m}\otimes \text{id}_\cF"]&\ \mathcal{P}_X^n\otimes_{\cO_X}\mathcal{P}_X^m\otimes_{\cO_X} \cF\ \ar[r, "\text{id}_{\mathcal{P}_X^n}\otimes D_1"]&\ \mathcal{P}_X^n \otimes_{\cO_X} \cG \ar[r, "D_2"] & \mathcal{H}.
	\end{tikzcd}
	\]
	in other words, if $D_1,D_2$ are defined on an open $U\subset X$, then \[
	D_2\circ D_1=
	D_2\circ \left(
	\text{id}_{{\mathcal{P}_X^n}_{|_U}}\otimes D_1\right)\circ
	\left(\delta_U^{n,m}\otimes \text{id}_{\cF_{|U}}\right).\] 
\end{lemma}

We recall the following definition.

\begin{definition}
	A \textbf{stratification} on an $\cO_X$-module $\cF$ is a collection $\{\epsilon_n\}_{n\in\N}$ of $\cP_X$-linear isomorphisms
	\[
	\epsilon_n: p_1^{n}(1)^*\cF=\mathcal{P}_X^n\otimes_{\cO_X} \cF \longrightarrow \cF\otimes_{\cO_X} \mathcal{P}_X^n=p_{0}^{n}(1)^*\cF 
	\]
	with the following properties:
	\begin{itemize}
		\item (compatibility) for all $n\leqslant m$ the following diagram commutes
		\[
		\begin{tikzcd}
			\mathcal{P}_X^m\otimes_{\cO_X} \cF \ar[r, "\epsilon_m"] \ar[d]& \cF\otimes_{\cO_X} \mathcal{P}_X^m\ar[d]\\
			\mathcal{P}_X^n\otimes_{\cO_X} \cF \ar[r, "\epsilon_n"]& \cF\otimes_{\cO_X} \mathcal{P}_X^n\\
		\end{tikzcd},
		\] where the vertical maps are induced by the projections;
		\item (cocycle condition) for all $n\in \N$
		\begin{equation}\label{CocycleConditionDefStratification}
			\left(p_{0,1}^{n}(2)^*\epsilon_n\right) \circ \left((p_{1,2}^{n}(2)^* \epsilon_n\right)= p_{0,2}^{n}(2)^*\epsilon_n.
		\end{equation}
		\item (identity) $\epsilon_0$ is the identity.
	\end{itemize}
\end{definition}
The cocycle condition is well defined because
\[
p_{1,2}^{n}(2)^*p_{0}^{n}(1)^* \cF= \left(p_{0}^{n}(1)\circ p_{1,2}^{n}(2)\right)^* \cF=  \left(p_{1}^{n}(1)\circ p_{0,1}^{n}(2)\right)^* \cF= p_{0,1}^{n}(2)^*p_{1}^{n}(1)^* \cF,
\]
and similarly
\[
p^n_{1,2}(2)^* p_1^n(1)^*=p^n_{0,2}(2)^* p_1^n(1)^*\quad 
\text{and}\quad   p^n_{0,1}(2)^* p_0^n(1)^*=p^n_{0,2}(2)^* p_0^n(1)^*.
\]

Given a stratification $\{\epsilon_n\}_{n\in \N}$ on $\cF$, we define a morphism of sheaves of rings
\[
\cD{\rm iff}^\bullet(\cO_X,\cO_X)\longrightarrow \cD{\rm iff}^\bullet(\cF,\cF).
\]
Let $\partial$ be a section of $ \cD{\rm iff}^n(\cO_X,\cO_X)$, then we can consider the map
\[
\begin{tikzcd}
	\cF\ar[d, "t_\cF"]&\cF \otimes_{\cO_X} \cO_X=\cF\\
	\mathcal{P}_X^n\otimes_{\cO_X} \cF\ar[r, "\epsilon_n"]& \cF\otimes_{\cO_X} \mathcal{P}_X^n\ar[u, " \text{id}_\cF\otimes\overline{\partial}^n"]
\end{tikzcd}.
\]
By definition, $\epsilon_n$ is $\cP_X$-linear, then
\[
\nabla^n(\partial):=\left( \text{id}_\cF\otimes \partial^n\right)\circ \epsilon_n
\]
is $\cO_X$-linear with respect to the left structure, then
\[
\nabla^n(\partial):= \nabla^n(\partial)\circ t_\cF\in \cD{\rm iff}^n(\cF,\cF).
\]
We defined a morphism
\[
\nabla^n: \cD{\rm iff}^n(\cO_X,\cO_X)\longrightarrow \cD{\rm iff}^n(\cF,\cF).
\]
Via the compatibility of $\{\epsilon_n\}_{n\in \N}$, the $\nabla^n$'s agree and there is a morphism
\[
\nabla: \cD{\rm iff}^\bullet(\cO_X,\cO_X)\longrightarrow \cD{\rm iff}^\bullet(\cF,\cF).
\]

\begin{proposition}\label{PropositionStratificationThenDiffAlgebraMap}
	Given a stratification $\{\epsilon_n\}_{n\in \N}$ of an $\cO_X$-module $\cF$, the morphism
	\[
	\nabla: \cD{\rm iff}^\bullet(\cO_X,\cO_X)\longrightarrow \cD{\rm iff}^\bullet(\cF,\cF).
	\] defined above is a morphism of $\cO_X$-algebras. 
\end{proposition}

\begin{lemma}\label{LemmaDiffGenInDegreeOne}
	The $\cO_X$-algebra $\cD{\rm iff}^\bullet(\cO_X,\cO_X)$ is locally generated by the elements of degree $1$.
\end{lemma}
\begin{proof}
	Since the statement is local, let us assume, using Lemma \ref{lemma:logsmooth} that $\uX$ is affinoid with a log chart given by the finite saturated monoid $M$,
	such that  $\cP_X^n=\cO_X\langle M^{\rm gp}\rangle/J^{n+1},$ and the torsion degree of $M^{\rm gp}$ is invertible in $X$. Therefore we have 
	\[
	\mathcal{P}_X^n\cong \bigoplus_{k=0}^n {\rm Sym}^k(J/J^2).
	\]
	
	We recall that
	\[
	\cD{\rm iff}^n(\cO_X,\cO_X)\cong \cH{\rm om}_{\cO_X}(\mathcal{P}_X^n, \cO_X)
	\]
	as $\cO_X$-modules. Let us denote by $x_1,...,x_d$ a basis of $M^{\rm gp, tf}:=M^{\rm gp}/(M^{\rm gp})^{\rm tor}$, the torsion free quotient of $M^{\rm gp}$.
	Then $\zeta_i:=x_1-1$, $1\le i\le d$ give a system of generators of the augmentation ideal $J$.
	For a $q\in \N^{d}$ let's denote by $\overline{D}_q\in \cH{\rm om}_{\cO_X}(\mathcal{P}_X^n, \cO_X)$ the dual morphism to $\zeta^q:=
	(\zeta_1)^{q_1}\dots (\zeta_d)^{q_d}\in \mathcal{P}_X^n$ and by $\partial_{x_i}$ the dual morphism to ${\rm dlog} (x_i)$.
	
	Via the isomorphism above $\cD{\rm iff}^n(\cO_X,\cO_X)$ is generated as module (hence as algebra) by all the $D_q$ with
	\[
	|q|_1:=\sum_{i=1}^d q_i\leqslant n.
	\]
	Define $q!:= \prod_{i=1}^d q_i!$. One can check via the composition formula in lemma  \ref{lemma:comp}, as in 
	\cite{berthelot_ogus}, that
	\[
	D_q\circ D_{q'} ({\rm dlog}( x^{q''}))= \frac{(q')! (q'+q)!}{q!} \ D_{q+q'} ({\rm dlog} (x^{q''})).
	\]
	So
	\[
	D_q\circ D_{q'}= \binom{q+q'}{q} D_{q+q'}.
	\]
	Then the elements $\partial_{x_i}$ commute and by induction one gets that for any $q\in\N^d$
	\[
	D_q= \frac{1}{q!} \prod_{i=1}^n \partial_{x_i}^{q^i},
	\]
	and $\cD{\rm iff}^1\left(\cO_X,\cO_X\right)$ is generated by these sections, as stated.
\end{proof}

\subsubsection{The global sections of a quasi-coherent crystal.}

We continue to use the notations of section \ref{SectionDiffOp}.

\begin{proposition}\label{PropGlobalSectionsCrystal}
	Let $F\in \Xinf$ be a quasi-coherent crystal and suppose $f:X\rightarrow S$ is log smooth and separated morphism, then
	\[
	\Gamma\left((X/S)_{\rm inf}^{\rm log}, F\right)=\{x\in F(X,X)\mid \epsilon^F_1(p_0^1(1)^*x)= p_1^1(1)^* x\}.
	\]
\end{proposition}
\begin{proof}
	We wish to show that for any $x\in F(X,X)$, if $\epsilon^F_1(p_0^1(1)^*x)= p_1^1(1)^* x$, then $\epsilon_n^F(p_0^n(1)^*x)= p_1^n(1)^*x$ for every $n\in \N_{>1}$.
	
	\textbf{I step.} Observe that $\{\epsilon_n\}_{n\in \N}$ is a stratification on $F_{(X,X)}$, i.e.
	\begin{itemize}
		\item (identity) $\epsilon_0$ is the restriction associated to $\text{id}_{(X,X)}$, then it must be the identity of $F_{(X,X)}$;
		\item (compatibility) holds via the cocycle condition of the restriction morphisms of $F$ applied to the following morphisms for $i=0,1$.
		\[
		\begin{tikzcd}[row sep=large]
			X\ar[d]\ar[r,equal]&X\ar[d]\ar[r,equal]& X\ar[d, equal]\\
			P^n_X\ar[r]\ar[rr, bend right, "p_i^n(1)"]&P^m_X\ar[r, "p_i^m(1)"]&X
		\end{tikzcd}
		\]
		\item (cocycle condition): via the cocycle condition of the restriction morphisms it follows that the diagram 
		\[
		\begin{tikzcd}
			p^n_2(2)^* F_{(X,X)}\ar[r, "\cong"]\ar[rd]&
			F_{D_X^n(2)}&\ar[l, "\cong"]
			p^n_0(2)^* F_{(X,X)}\\
			&p^n_1(2)^* F_{(X,X)}\ar[ru] \ar[u, "\cong"]&
		\end{tikzcd}
		\]
		is commutative and the cocycle condition on $\{\epsilon_n\}_{n\in \N}$ is a consequence of the following equalities
		\begin{multline*}
			p^n_2(2)^*=p^n_{1,2}(2)^* p_1^n(1)^*=p^n_{0,2}(2)^* p_1^n(1)^*; \quad p^n_1(2)^*=p^n_{1,2}(2)^* p_0^n(1)^*=p^n_{0,1}(2)^* p_1^n(1)^*\\
			\text{and}\quad  p^n_0(2)^*=p^n_{0,1}(2)^* p_0^n(1)^*=p^n_{0,2}(2)^* p_0^n(1)^*.
		\end{multline*}
	\end{itemize}
	
	Since $\{\epsilon_n\}_{n\in \N}$ is a stratification on $F_{(X,X)}$, via the Proposition \ref{PropositionStratificationThenDiffAlgebraMap}, there is a homomorphism of sheaves of $\cO_X$-algebras
	\[
	\nabla: {\rm Diff}^\bullet\left(\cO_X,\cO_X\right)\longrightarrow {\rm Diff}^\bullet\left(F_{(X,X)}, F_{(X,X)}\right).
	\]
	
	\textbf{II step.} For each open $U\subset X$, the following equality holds:
	\[
	\{x\in F_{(X,X)}(U)\mid \epsilon^F_1(p_0^1(1)^*x)=p_1^1(1)^* x\}= \bigcap_{\partial\in D^1(U)} \{x\in F_{(X,X)}(U)\mid \left(\nabla_U\partial\right)(x)=0\}
	\]
	where $D^1(U):={\rm Diff}^1(\cO_X,\cO_X)(U)\setminus {\rm Diff}^0(\cO_X,\cO_X)(U)$.
	
	Since the two sets (letting $U$ vary) are defined by local conditions, they define sub-sheaves of $F$ and we can check the equality locally.  By applying Lemma \ref{lemma:logsmooth}, we may assume that $X$ is affinoid,
	with log structure given by the finite saturated monoid $M$, and fix a basis $y_1,...,y_d$ of the torsion free quotient of $M^{\rm gp}$. We denoted $J$ the augmentaion ideal of $\cO_X\langle M^{\rm gp}\rangle$.
	Since ${\rm dlog}( y_i), 1\otimes 1$ is a basis of $\mathcal{P}_X^1$, we can write
	\[
	\epsilon_1(p_0^1(1)^* x)=x_0+\sum_{i=1}^d x_i \otimes {\rm dlog}(y_i) .
	\]
	Observe that $x=\epsilon_0(x)=x_0$ since $\epsilon_0$ is the identity on $F$ and $\epsilon_1\equiv \epsilon_0$ modulo $J$. Moreover
	\[
	(\nabla\partial_{y_i})  (x)= \overline{\partial_{y_j}}^1  \left(x+\sum_{i=1}^d x_i \otimes {\rm dlog}(y_i) \right)= x_j.
	\]
	Then $(\nabla\partial)(x)=0$ for each $\partial\in D^1(X)$ iff $x_i=0$ for each $i=1,\dots,d$ and the $2$nd step follows.,
	
	\textbf{III step.} For each $n\in \N_{>0}$ and any open $U\subset X$
	\begin{multline*}
		\{x\in F_{(X,X)}(U)\mid \epsilon_n(p^n_0(1)^*x)= p^n_1(1)^*x\}=\{x\in F_{(X,X)}(U)\mid \epsilon_1(p_0^1(1)^*x)= p_1^1(1)^* x\}.
	\end{multline*}
	One inclusion is true in general (by the compatibility condition).
	
	For the other inclusion, as in the previous step, we can work locally where we have coordinates $y_1,\dots, y_d$, i.e. a basis of the torsion free quotient of $M^{\rm gp}$ (see Lemma \ref{lemma:logsmooth}). We'll prove this by induction on $n\in \N_{>1}$. Fix an $x\in F_{(X,X)}(X)$ with $\epsilon_1(p^1_0(1)^*x)= p^1_1(1)^*x$. Denote
	\[
	\epsilon_n(p_0^1(1)^* x)=x+\sum_{ |q|_1 = n} x_q \otimes ({\rm dlog}(y))^q .
	\]
	where $q\in \N^d$ and $({\rm dlog}(y))^q:= ({\rm dlog}( y_1))^{q_1}\dots ({\rm dlog}(y_d))^{q_d}$ via the identification
	\[
	\mathcal{P}_X^n\cong \cO_X[{\rm dlog}(y_1),\dots,{\rm dlog}(y_d)]^{deg\leqslant n}.
	\]
	But now observe that for each $q_*\in \N^d$ with $|q_*|_1=n$
	\begin{multline*}
		x_{q_{*}}= D_{q_{*}} \left(x+\sum_{|q|_1 = n} x_q \otimes ({\rm dlog}(y))^q\right)=(\nabla D_{q_*}) (x)\\
		=(\nabla \frac{1}{q_*!}\prod_{i=1}^d \partial_{y_i}^{q_{*,i}}) (x)= \frac{1}{q_*!} \left(\prod_{i=1}^d (\nabla\partial_{y_i})^{q_{*,i}}\right)(x)=0
	\end{multline*}
	where we used that: a) $D_{q_*}({\rm dlog}(y))^q= \delta_{q}^{q_*}\cdot  ({\rm dlog}(y))^{q_*}$ for each $q\in \N^d$ with $|q|_1=|q_*|_1$; b) the fact that $\cD{\rm iff}^n\left(\cO_X,\cO_X\right)$ is (locally) generated as $\cO_X$-algebra in degree $1$ (we wrote the explicit formula); c) the fact that $\nabla$ is a ring homomorphism and the $2$nd step.
\end{proof}

\subsection{Linearization and De-linearization}
\label{sec:linedelin}

A. Grothendieck, in \cite{grothendieck}, defined the linearization of de Rham complexes for schemes over fields of characteristic zero; Berthelot-Ogus, in \cite{berthelot_ogus}, generalized linearization to schemes over arbitrary characteristic via divided power structures. We follow their constructions in order to associate to a quasi-coherent $\cO_X$-module $\cF$ on $X$ over $S$, with log integrable connection, a quasi-coherent crystal $\cF_{\rm inf}$, and the linearized de Rham complex. We will show that in our setting there is also a ``de-linearization" functor $u_*$ that preserves cohomology.

\subsubsection{Linearization}\label{sec:linearization}

We use the notations of section \ref{SectionDiffOp}, i.e. $X\to S$ is a log smooth morphism of fs log adic space, where $S={\rm Spa}(B,B^+)$. Let $\cF$ be a quasi-coherent $\cO_X$-module. As before $I\subset (\cO_X\ho_{f^{-1}\cO_S}\cO_X)^{\rm ex}$ is the ideal corresponding to the log diagonal immersion.

\begin{definition}
	The \textbf{linearization of} $\cF$ is the $\cP_X$-module
	\[
	L(\cF):=\lim_{n\in \N} L_n(\cF),
	\]
	where the projective system is $\{L_n(\cF):= \mathcal{P}^{n}_X\otimes_{\cO_X} \cF \}_{n\in \N}$ with morphisms
	\[
	L_{n+1}(\cF)= (\cO_X\ho_{f^{-1}\cO_S}\cO_X)^{\rm ex}/I^{n+2}\otimes_{\cO_X} \cF\longrightarrow (\cO_X\ho_{f^{-1}\cO_S}\cO_X)^{\rm ex}/I^{n+1}\otimes_{\cO_X} \cF= L_n(\cF).
	\]
\end{definition}

\begin{remark}\label{RemarkPolynomials}
	As we observed, locally, where $x_1,\dots, x_d$ are coordinates of $X$, i.e. generators of the group $M^{\rm gp}/(M^{\rm gp})^{\rm tor}$, for $(U,M)$ a local chart of $X$, we denote $G:=(M^{|rm gp})^{\rm tor}$, which is a finite abelian group and we can write 
	$$
	\mathcal{P}^n_X= \cO_X[G][\zeta_1,\dots,\zeta_d]^{deg\leqslant n},
	$$
	where $\zeta_i=x_i-1(\mbox{mod } I^{n+1})$ and the map 
	\[
	\mathcal{P}^{n+s}_X\rightarrow \mathcal{P}^{n}_X
	\]
	corresponds to the projection of polynomials of degree $\leqslant n+s$ into polynomials of degree $\leqslant n$. Hence
	\[
	\mathcal{P}_X:=\lim_{n\in \N} \mathcal{P}^n_X= \cO_X[[\zeta_1,\dots,\zeta_d]].
	\]
	
	We observe that $L(\cF)= \mathcal{P}_X\ho_{\cO_X}\cF$, where the completion of the tensor product is taken with respect to the $(\zeta_1,\dots , \zeta_d)$-topology. We remark that for every $n\ge 0$, $L_n(\cF)$ is a quasi-coherent $\cO_X$-module, but $L(\cF)$ is not quasi-coherent. We can think of it as a pro-quasi-coherent sheaf, i.e. a projective system/limit of quasi-coherent ones.
\end{remark}

\begin{lemma}
	If $D:\mathcal{F}\rightarrow \mathcal{G}$ is a log differential operator of degree $n\in \N$, then there is a $\cP_X$-linear map
	\[
	L(D): L(\mathcal{F})\rightarrow L(\mathcal{G})
	\] 
	limit of the morphisms $L_k(D)$ for $k\in \N$, defined by
	\[
	\left(\text{id}_{\mathcal{P}^k_X}\otimes D^n\right)\circ\left( \delta^{k,n}\otimes \text{id}_{F}\right):\mathcal{P}_X^{k+n} \otimes_{\cO_X} \mathcal{F} \rightarrow \mathcal{P}_X^{k}\otimes_{\cO_X} \mathcal{P}_X^{n}\otimes_{\cO_X} \mathcal{F} \rightarrow \mathcal{P}_X^{k} \otimes_{\cO_X}\mathcal{G}.
	\]
\end{lemma}
\begin{proof}
	Since
	\[
	\cH{\rm om}_{P}\left(L(\mathcal{F}), \lim_{k\in \N} \mathcal{P}_X^{k} \otimes_{\cO_X} \mathcal{G}\right)= \lim_{k\in \N} \cH{\rm om}_{P}\left(L(\mathcal{F}), \mathcal{P}_X^{k} \otimes_{\cO_X} \mathcal{G}\right)
	\]
	we have to show that there are compatible $\cP_X$-linear maps $L(\mathcal{F})\rightarrow \mathcal{P}_X^{k} \otimes_{\cO_X} \mathcal{G}.$ We consider the $\cP_X$-linear morphisms as in the statement:

	\[
	L(\mathcal{F})\rightarrow \mathcal{P}_X^{k+n} \otimes_{\cO_X} \mathcal{F}\rightarrow\mathcal{P}^{k}_X \otimes_{\cO_X} \mathcal{G}.
	\]
	Let $k,s\in \N$, then the diagram
	\[
	\begin{tikzcd}
		L(\mathcal{F})\ar[r]\ar[rd]&\mathcal{P}_X^{k+s+n} \otimes_{\cO_X} \mathcal{F}\ar[d]\\
		&\mathcal{P}_X^{k+n} \otimes_{\cO_X} \mathcal{F}
	\end{tikzcd}
	\]
	commutes by definition of limit. Let's consider the following diagrams
	\[
	\begin{tikzcd}
		\mathcal{P}_X^{k+s+n}\ar[d]\ar[r, "\delta^{k+s, n}"]& \mathcal{P}_X^{k+s}\otimes_{\cO_X} \mathcal{P}_X^n\ar[d]\\
		\mathcal{P}_X^{k+n}\ar[r, "\delta^{k, n}"]& \mathcal{P}_X^{k}\otimes_{\cO_X} \mathcal{P}_X^n
	\end{tikzcd}\ \ 
	\begin{tikzcd}
		\mathcal{P}_X^{k+s}\otimes_{\cO_X} \mathcal{P}_X^{n}\otimes_{\cO_X} \mathcal{F} \ar[d] \ar[r, "\text{id}_{\mathcal{P}^{k+s}}\otimes D^n"] \ \ &\ \ 
		\mathcal{P}_X^{k+s} \otimes_{\cO_X}\mathcal{G}\ar[d]\\
		\mathcal{P}_X^{k}\otimes_{\cO_X} \mathcal{P}_X^{n}\otimes_{\cO_X} \mathcal{F} \ar[r, "\text{id}_{\mathcal{P}^{k}}\otimes D^n"]&
		\mathcal{P}_X^{k} \otimes_{\cO_X}\mathcal{G}
	\end{tikzcd}
	\]
	where the vertical maps are induced by projections. Observe that the two diagrams are commutative and this system gives a $\cP_X$-linear map
	\[
	L(\mathcal{F})\longrightarrow L(\mathcal{G}).
	\]
\end{proof}
\begin{remark}
	We observe that if $D\in \cD{\rm iff}^n(\cF, \mathcal{G})$, then we can see $D$ as a section of $\cD{\rm iff}^{n+k}(\cF, \mathcal{G})$ for any $k\in \N$. One can check that $D$ and its images in $\cD{\rm iff}^{n+k}(\cF, \cG)$ give the same map $L(D):L(\cF)\rightarrow L(\mathcal{G})$. 
\end{remark}
\begin{lemma}
	If $\mathcal{F},\mathcal{G},\mathcal{H}$ are $\cO_X$-modules, $D_1\in \cD{\rm iff}^{n_1}(\mathcal{F},\mathcal{G})$ and $D_2\in \cD{\rm iff}^{n_2}(\mathcal{G},\mathcal{H})$, then 
	\[
	L(D_2\circ D_1)=L(D_2)\circ L(D_1): L(\mathcal{F})\longrightarrow L(\mathcal{H}).
	\]
\end{lemma}
Let $\nabla: \cF\rightarrow \cF\otimes_{\cO_X} \omega^1_{X/S}$ be an integrable log connection.
\begin{corollary}
	Let $\cF$ be a quasi-coherent $\cO_X$-module with an integrable log connection $\nabla$, then there is a complex
	\[
	\begin{tikzcd}
		L(\cF)\ar[r, "L(\nabla)"]& L(\cF\otimes_{\cO_X} \omega^1_{X/S})\ar[r, "L(\nabla^1)"] & L(\cF\otimes_{\cO_X} \omega^2_{X/S})\ar[r, "L(\nabla^2)"]&\dots 	
	\end{tikzcd}
	\]
\end{corollary}
We claim that this complex is exact. Let us consider the canonical derivation

\[
\de: \cO_X\rightarrow \omega_{X/S}^1,
\]
which we see as an integrable log connection on $\cO_X$.

\begin{lemma}
	For any $k\in \N$, locally on $X$ where there are coordinates $x_1,\dots , x_d$, the map
	\[
	L(\de): \mathcal{P}_X \ho_{\cO_X}\omega^k_{X/S}\rightarrow \mathcal{P}_X\ho_{\cO_X}\omega^{k+1}_{X/S}
	\]
	is the completion of the map
	\[
	\begin{tikzcd}[row sep=small]
		\cO_X[[\zeta_1,\dots , \zeta_d]]\otimes_{\cO_X}\omega_{X/S}^k\ar[r]& \cO_X[[\zeta_1,\dots , \zeta_d]]\otimes_{\cO_X}\omega_{X/S}^{k+1}\\
		f \otimes \omega\ar[r, maps to]& f\otimes\de^k\omega+\sum_{i=1}^d  \left(\frac{\partial }{\partial \zeta_i}f\right) \otimes {\rm dlog}(x_i)\wedge\omega
	\end{tikzcd}
	\]
	where $\zeta_i=x_i-1$ for $i=1,..,d$.
\end{lemma}

\begin{proof}

The proof is a simple computation which we omit.

\end{proof}

\begin{proposition}\label{PropositionLinearizedDifferentialIsExact}
	The complex
	\[
	\begin{tikzcd}
		L(\cO_X)\ar[r, "L(\de)"]& L( \omega^1_{X/S})\ar[r, "L(\de^1)"] & L( \omega^2_{X/S})\ar[r, "L(\de^2)"]&\dots 	\end{tikzcd}
	\]
	is exact, moreover
	\[
	{\rm Ker}(L(\de))=\cO_X.
	\]
\end{proposition}
\begin{proof}
	The statement is local, hence we may suppose that $x_1,\dots, x_d $ are coordinates of $X$. 
	
	Now the proof follows the arguments in  section 6.12 in \cite{berthelot_ogus}, or one can use explicit computations with local coordinates.

\end{proof}
This Proposition implies that the complex with $\cO_X$ in degree $0$ and all other terms $0$, is quasi-isomorphic to the complex 
\[
\begin{tikzcd}
	L( \omega^1_{X/S})\ar[r, "L(\de^1)"] & L( \omega^2_{X/S})\ar[r, "L(\de^2)"]&\dots 	\end{tikzcd}.
\]
This is true for any flat quasi-coherent $\cO_X$-module with integrable connection.
\begin{proposition}\label{PropExactLinearizedComplex}
	Let $\cF$ be a quasi-coherent $\cO_X$-module with integrable connection $\nabla$, then the complex
	\[
	\begin{tikzcd}
		L(\cF)\ar[r, "L(\nabla)"]& L(\cF\otimes_{\cO_X} \omega^1_{X/S})\ar[r, "L(\nabla^1)"] & L(\cF\otimes_{\cO_X} \omega^2_{X/S})\ar[r, "L(\nabla^2)"]&\dots 	\end{tikzcd}
	\]
	is isomorphic to the complex
	\[
	\begin{tikzcd}
		\cF \ho_{\cO_X}L(\cO_X)\ar[r, "\text{id}_\cF\otimes L(\de)"]& \cF \ho_{\cO_X}L( \omega^1_{X/S})\ar[r, "\text{id}_\cF\otimes L(\de^1)"] & \cF \ho_{\cO_X}L( \omega^2_{X/S})\ar[r, " \text{id}_\cF\otimes L(\de^2)"]&\dots 	\end{tikzcd}.
	\]
	In particular if $\cF$ is flat the two complexes are exact with kernel
	\[
	{\rm Ker}(L(\nabla))\cong \cF.
	\]
\end{proposition}
\begin{proof}
	Let $n\in \N$, there is a stratification $\{\epsilon_n\}_{n\in \N}$ on $\cF$ attached to the connection $\nabla$, i.e. a compatible system of $P$-linear isomorphisms where
	\[
	\epsilon_n: \mathcal{P}^n_X\otimes_{\cO_X}\cF\longrightarrow \cF\otimes_{\cO_X} \mathcal{P}^n_X
	\]
	with $\epsilon_0=\text{id}_\cF$ and cocycle condiction. The map $\epsilon_1$ is defined as
	\[
	\begin{tikzcd}[row sep=small]
		\epsilon_1:\mathcal{P}^1_X\otimes_{\cO_X}\cF\ar[r]& \cF\otimes_{\cO_X}\mathcal{P}^1_X\\
		(a\otimes b)\otimes x\ar[r,maps to]& (a\otimes b)\cdot \left(\nabla(x)+ x\otimes 1\otimes 1\right)
	\end{tikzcd}.
	\]
	Now we want to show that the maps $\epsilon_n$ induce a map between the two complexes, i.e. that
	\[
	\begin{tikzcd}[row sep=large, column sep=large]
		\mathcal{P}^{n+1}_X\otimes_{\cO_X}\cF\otimes_{\cO_X}\Omega^k_{X/S}\ar[r, "L_n(\nabla^k)"]\ar[d, "\epsilon_{n+1}\otimes \text{id}_{\Omega_{X/S}^{k}}"]& \mathcal{P}^n_X\otimes_{\cO_X}\cF\otimes_{\cO_X}\Omega^{k+1}_{X/S}
		\ar[d, "\epsilon_{n} \otimes  \text{id}_{\Omega_{X/S}^{k+1}} "]\\
		\cF \otimes_{\cO_X} \mathcal{P}^{n+1}_X\otimes_{\cO_X}\Omega_{X/S}^k\ar[r, "\text{id}_\cF\otimes L_n(\de^k)"]& \cF \otimes_{\cO_X} \mathcal{P}^n_X\otimes_{\cO_X}\Omega_{X/S}^{k+1}
	\end{tikzcd}
	\]
	is commutative for each $n,k\in \N$. Once we show the commutativity, since $\epsilon_n$ is an isomorphism we conclude the first statement of the Lemma.
	
	The statement now is local, hence we may suppose that $X$ has coordinates $x_1,\dots, x_d$. Let $y:= \uno\otimes x\otimes \omega$ where  $x\in \cF$ and $\omega={\rm dlog}( x_{i_1})\wedge\dots {\rm dlog} (x_{i_k})$. Then 
	\begin{multline*}
		\left(\epsilon_{n} \otimes  \text{id}_{\Omega_{X/S}^{k+1}}\right)\circ L_n(\nabla^k)(y)=\left(\epsilon_{n} \otimes  \text{id}_{\omega_{X/S}^{k+1}}\right)\left( \sum_{i=1}^d\uno \otimes  (\nabla_{\partial_{x_i}}x) \ {\rm dlog}( x_i) \wedge \omega\right)\\
		= \sum_{i=1}^d\sum_{0\leqslant |q|_1\leqslant n} \frac{1}{q!} \left(\prod_{j=1}^d   \nabla_{\partial_{x_j}}^{q_j+ \delta_i^j}\right) (x)\ \zeta^q \otimes \ {\rm dlog}(x_i) \wedge \omega
	\end{multline*}
	and
	\begin{multline*}
		\left(
		\text{id}_\cF\otimes L_n(\de^k)\right)\circ\left(\epsilon_{n+1}\otimes \text{id}_{\Omega^k_{X/S}}
		\right)
		(y)
		= 	
		\left(
		\text{id}_\cF\otimes L_n(\de^k)
		\right)
		\left(
		\sum_{0\leqslant |q|_1\leqslant n+1}\left(\prod_{j=1}^d \nabla_{\partial_{x_j}}^{q_j}\right)(x)\ \zeta^{q} \otimes \omega	
		\right)
		\\
		=\sum_{1\leqslant |q|_1\leqslant n+1}\sum_{i=1}^d \frac{1}{q!}\left(\prod_{j=1}^d \nabla_{\partial_{x_j}}^{q_j}\right)(x)\  \frac{\partial \zeta^{q}}{\partial_{\zeta_{q_i}}} \otimes{\rm dlog}(x_i)\ \wedge\ \omega.
	\end{multline*}
	The two expressions are equal (via a change of the variable $q$ in the summation).
	
	For the last two assertions observe that if $\cF$ is flat, hence  locally the short exact sequence
	\begin{multline*}
		\cF \otimes_{\cO_X} \cO_X[G][\zeta_1,\dots , \zeta_d]^{\leqslant n+2}\otimes_{\cO_X} \omega_{X/S}^k
		\overset{id_\cF\otimes L_{n+1}(d^k) }{\longrightarrow }
		\cF \otimes_{\cO_X} \cO_X[G][\zeta_1,\dots , \zeta_d]^{\leqslant n+1}\otimes_{\cO_X} \omega_{X/S}^{k+1} 
		\\
		\overset{id_\cF\otimes L_n(d^{k+1})}{\longrightarrow }
		\cF \otimes_{\cO_X} \cO_X[G][\zeta_1,\dots , \zeta_d]^{\leqslant n}\otimes_{\cO_X} \omega_{X/S}^{k+2}
	\end{multline*}
	is exact for each $n,k\in \N$. Hence
	\[
	0\rightarrow {\rm Ker}(L_{n+1}(d^{k}) )\hookrightarrow \cF \otimes_{\cO_X} \cO_X[G][\zeta_1,\dots , \zeta_d]^{\leqslant n+2}\otimes_{\cO_X} \omega_{X/S}^k \rightarrow {\rm Ker}(L_n(d^{k+1}))\rightarrow 0
	\]
	is exact for each $n,k\in \N$, since ${\rm Ker}(L_{n+1}(d^{k}) )$ satisfies the Mittag-Lefler condition, the kernel and the limit commute, and we conclude. We recall that $G:=(M^{\rm gp})^{\rm tor}$ is a finite abelian group.
\end{proof}

\subsubsection{The quasi-coherent crystal associated to a quasi-coherent sheaf with stratification.}

In this section we describe how to associate a quasi-coherent crystal on $\XSinf$ to the data of a quasi-coherent $\cO_X$-module $F$ with a stratification. We will then apply this to associate to every quasi-coherent $\cO_X$-module $F$, with integrable connection $\nabla$, a family of quasi-coherent crystals $\bigl((L_n(F))_{\rm inf}\bigr)_n$, and in the end a pro-quasi-coherent crystal
$L(F)_{\rm inf}:=\lim_n \bigl(L_n(F)\bigr)_{\rm inf}.$

Let $F$ be a quasi-coherent $\cO_X$-module with a stratification $\{\epsilon_n\}_{n\in \N}$. We have seen in lemma \ref{lemma1isCoEq} that $(X,X)$ covers $\uno$, i.e. that for any object $(U,T)\in \XSinf$ there is a cover $\{(U_i, T_i)\}_{i\in I}$ with maps $g_i:T_i\rightarrow X$ s.t. the following diagram commutes
\[
\begin{tikzcd}
	U_i\ar[r]\ar[d, "g_i"]&X\\
	T_i\ar[ur, "\beta_i"]&
\end{tikzcd}
\]
for any $i\in I$. Hence locally we can define the sheaf $F_{(U_i,T_i)}:=\beta_i^* F=O_{T_i}\otimes_{\beta_i^{-1}(\cO_X)}\beta_i^{-1}(F)$, which is a quasi-coherent $\cO_{T_i}$-module; now we use the stratification in order to glue these sheaves. We may suppose that the map $g_i$ is nilpotent of order $n_i\in \N$. Let $i,j\in I$, $U_{i,j}:=U_i\cap U_j$, $T_{i,j}:= T_i\cap T_j$ and $n_i, n_j\leqslant n\in \N$, then there is a diagram
\[
\begin{tikzcd}
	U_{i,j}\ar[r, "\alpha_{i,j}  "]& T_{i,j}\ar[d, "g_{i,j}"]\ar[rd, "g_j"]\ar[dl, "g_i"]& \\
	X&X\times_S X\ar[r, "p_1"]\ar[l, "p_0"]& X
\end{tikzcd}
\]
the map $g_{i,j}$ factors through a map $h_{i,j}: T_{i,j}\rightarrow P^n_X$, the proof of this fact is similar to the proof of the lemma \ref{lemmaD_X(Y)=Hom (-,Y)}: using that $g_i\circ \alpha_{i,j}= g_j\circ \alpha_{i,j}$, the elements of $I$ are sent to elements in $\cO_{T_{i,j}}$ that differ by $n$-nilpotent elements, hence the two morphisms coincide modulo $I^n$. For any $n_i,n_j\leqslant n\in \N$ we get a diagram
\[
\begin{tikzcd}[column sep=large]
	U_{i,j}\ar[r, "\alpha_{i,j}  "]& T_{i,j}\ar[d, "g_{i,j}^n"]\ar[rd, "g_j"]\ar[dl, "g_i"]& \\
	X&P^n_X \ar[r, "p_1^n(1)"]\ar[l, "p_0^n(1)"]& X
\end{tikzcd}.
\]
We can glue the sheaf $F_{(U_i,T_i)}$ over $T_i$ and the sheaf $F_{(U_j,T_j)}$ over $T_j$ via the isomorphism given by the stratification:
\[
\begin{tikzcd}[column sep=large]
	\psi_{i,j}: g_i^* F=
	(g_{i,j}^{n})^*  p_0^n(1)^* F\ar[r,"(g_{i,j}^n)^* \epsilon_n"]&
	(g_{i,j}^{n})^* p_1^n(1)^* F=
	g_j^* F.
\end{tikzcd}
\]
The compatibility condition ensures that this procedure works with any $n\in \N$ bigger than $n_i,n_j$, the cocycle condition implies that the composition of two gluing is the right gluing, the identity condition (toghether with the cocycle condition) implies that $\psi_{i,j}= \psi_{j,i}^{-1}$. Hence we can glue these sheaves to a unique sheaf $F_{(U,T)}$.  For any map $(\alpha,\beta):(U_1,T_1)\rightarrow (U_2,T_2)$ there is a canonical isomorphism $\beta^*F_{U_1,T_1}\cong F_{(U_2,T_2)}$: if locally on $T_2$ we have the maps $g_i^{(2)}: T_{2,i}\rightarrow X$, then we get maps $g_i^{(1)}:= g_i^{(2)}\circ \beta_i$ and $F_{(U_{1,i}, T_{1,i})}\cong (g_i^{(1)})^* F\cong  \beta_i^* F_{(U_{2,i}, T_{2,i})} $ by definition; where $\beta_i$ is the right restriction of $\beta$, $\{(U_{2,i}, T_{2,i})\}$ is a cover as above of $(U_2,T_2)$ and $\{(U_{1,i}, T_{1,i}):= \left(\alpha^{-1}(U_{2,i}), \beta^{-1}(T_{2,i})\right) \}$.

We define the restriction morphism as the composition
\[
\beta^{-1} F_{(U_2,T_2)}\rightarrow  \cO_{T_2}\otimes_{\beta^{-1}\cO_{T_1}}  F_{(U_2,T_2)}\cong \beta^* F_{(U_2,T_2)}.
\]
One can check that the conditions of lemma \ref{lemmaPShinf} are satisfied (via the identity, cocycle condition of the stratification). 

Hence we have proved:
\begin{lemma}\label{LemmaCrystallization}
	If $F$ is a quasi-coherent $\cO_X$-module with a stratification $\{\epsilon_{n}\}_{n\in \N}$, then there is an associated quasi-coherent crystal $F_{\rm inf}$ on the infinitesimal site $(X/S)_{\rm inf}^{\rm log}$, associated to it.
\end{lemma}

Let us consider the category of quasi-coherent $\cO_X$-modules $\cF$ with differential operators of finite order as morphisms, we denote this category as ${\rm Mod}^{\rm diff}$.
\begin{lemma}
	There is a linearization functor
	\[
	L(-)_{\rm inf}: {\rm Mod}^{\rm diff}\longrightarrow {\rm Mod}_{\cO_{\Xinf}}
	\]
\end{lemma}
\begin{proof}
	For each $n\in \N$ we define $L_n(\cF)_{\rm inf}\in \cO_{\Xinf}$, then we will define
	\[
	L(\cF)_{\rm inf}:= \lim_{n\in \N} L_n(\cF)_{\rm inf}.
	\] 
	In order to define the sheaves $L_n(\cF)_{\rm inf}$ we use the construction of lemma \ref{LemmaCrystallization}. Let $\cF$ be a quasi-coherent  $\cO_X$-module, consider its $n$-th linearization $L_n(\cF):=\cP^n_X \tensor_{\cO_X}\cF$, it has a canonical stratification given by
	\[
	\begin{tikzcd}[row sep=small]
		\epsilon_{m,n}^{\cF,can}:\cP^m_X\tensor_{\cO_X} L_n(\cF)\ar[r, "\cong"]&
		L_n(\cF)\tensor_{\cO_X} \cP^m_X\\
		(a\otimes b)\otimes (c\otimes d) x\ar[r, maps to]&(a\tensor  dx) \tensor (1\tensor bc)=(1\tensor  dx) \tensor (a\tensor bc)
	\end{tikzcd}
	\]
	Via this stratification we obtain the crystal $L_n(\cF)_{\rm inf}$. Moreover let $D:\cF\rightarrow \cG$ be a differential operator of order $n$ between quasi-coherent $\cO_X$-modules, we get a linearized map $L_n(D): L_n(\cF)\rightarrow L_n(\cG)$. For any thickening $(U,T)$ with a section $g: T\rightarrow X$ we define 
	\[
	L_n(D)_{\rm inf, (U,T)}:= g^* L_n(D): L_n(\cF)_{\rm inf, (U,T)}= g^*L_n(\cF)\longrightarrow L_n(\cG)_{\rm inf, (U,T)}= g^*L_n(\cG)
	\]moreover the following diagram commutes
	\[
	\begin{tikzcd}
		p_0^m(1)^* L_n(\cF) \ar[r, "\epsilon_{m,n}^{\cF,can}"]\ar[d, "p_0^m(1)^* L_n(D)"]& p_1^m(1)^* L_n(\cF)\ar[d, "p_1^m(1)^* L_n(D)"] \\
		p_0^m(1)^* L_n(\cG) \ar[r, "\epsilon_{m,n}^{\cG,can}"]& p_1^m(1)^* L_n(\cG)
	\end{tikzcd},
	\] indeed following the upper path we get
	\[
	(a\tensor bx) \tensor (c\tensor d)\mapsto (a \tensor D(bx))\tensor (c\tensor d)\mapsto (ac \tensor d) \tensor (1\tensor D(bx))
	\]
	and following the lower path we get
	\[
	(a\tensor bx) \tensor (c\tensor d)\mapsto (ac \tensor d)\tensor (1\tensor bx)\mapsto (ac \tensor d)\tensor (1 \tensor D(bx)).
	\]
	Hence the morphisms $L_n(D)_{\rm inf}$ glue on intersections and we can define $L(D)_{\rm inf}$ as the limit of the maps
	\[
	L(\cF)_{\rm inf} {\longrightarrow} L_n(\cF)_{\rm inf}\longrightarrow L_n(\cG)_{\rm inf}.
	\]
\end{proof}

\begin{definition}
	If $\cF$ is a quasi-coherent $\cO_X$-module with an integrable connection $\nabla$, the \textbf{linearization of } $\cF$ in the infinitesimal topos is the sheaf
	\[
	\cF_{\rm inf}\ho L(\cO_X):= \lim_{n\in \N} \left(\cF\otimes_{\cO_X} L_n(\cO_X)\right)_{\rm inf}
	\]
	where the stratification is given by the isomorphisms
	\[
	\cP_X^m \otimes_{\cO_X} \left(\cF \otimes_{\cO_X} L_n(\cO_X) \right) \overset{\epsilon_m\otimes \text{id}_{\cP^n_X}}{\longrightarrow} \cF \otimes (\cP^m_X\otimes L_n(\cO_X)) \overset{\text{id}_\cF \otimes \epsilon_{m,n}^{ can}}{\longrightarrow} (\cF \otimes L_n(\cO_X))\otimes \cP^m_X
	\]
	and the limit is computed in $\Xinf$.
\end{definition}
In analogue way one can define for any $k\in \N$ the sheaf 
\[
\cF_{\rm inf}\ho L(\omega_{X/S}^k)_{\rm inf}:= \lim_{n\in \N} \left( \cF \otimes_{\cO_X} L(\omega^k_{X/S})_n\right)_{\rm inf}
\] where the stratification isomorphisms (as before) are given by the composition of the stratification on $\cF$, given by the connection, and the canonical stratification on $L(\omega^k)_n$. 
\begin{proposition}
	For any quasi-coherent $\cO_X$-module $\cF$ with integrable connection $\nabla$, the two complexes of sheaves $\cF_{\rm inf}\ho L(\omega^\bullet_{X/S})_{\rm inf}$ and $L(\cF\otimes \omega^\bullet_{X/S})_{\rm inf}$ are canonically isomorphic.
\end{proposition}
\begin{proof}
	The first complex is given via the (limit) maps of the following morphisms
	\[
	\cF_{\rm inf}\ho L(\omega^k_{X/S})_{\rm inf} \rightarrow \left(\cF \otimes L_{n+1}(\omega^k_{X/S})\right)_{\rm inf} \overset{\left(id_\cF\otimes L_n(d^k)\right)_{\rm inf}}{\longrightarrow} \left(\cF \otimes L_{n+1}(\omega^k_{X/S})\right)_{\rm inf}.
	\]
	In order to show that this morphism is well defined we have to check that it glues, i.e. that the maps "commute with the stratification", i.e. that the following diagram commutes
	\[
	\begin{tikzcd}
		\cP^m_X\tensor \cF \tensor L_{n+1}(\omega^k_{X/S})\ar[r, "\epsilon_m\otimes \text{id}"]\ar[d, "\text{id}\tensor \delta^{n,1}\tensor \text{id}"]
		&
		\cF \tensor \cP^n_X\tensor L_{n+1}(\omega^k_{X/S})\ar[r, "\text{id}_\cF\tensor \epsilon_{m,n}^{can}"]
		&
		\left(\cF\tensor L_{n+1}(\omega^k_{X/S})\right) \otimes \cP^m_X
		\ar[d, "\text{id}\tensor \delta^{n,1}\tensor \text{id}"]
		\\
		\cP^m_X\tensor \cF\tensor \cP_X^n\tensor \cP^1_X\tensor \omega_{X/S}^k \ar[d, "\text{id}\tensor L(d^k)_1"]
		&
		&
		\left(\cF \tensor \cP^n_X\tensor \cP^1_X\tensor \omega^k_{X/S}\right)\tensor \cP^m_X
		\ar[d, "\text{id}\tensor L_1(d^k)\tensor\text{id}"]
		\\
		\cP^m_X\tensor \cF\tensor \cP^n_X\tensor \omega^{k+1}_{X/S}
		\ar[r, "\epsilon_m\tensor \text{id}"]
		&
		\cF\tensor \cP^m_X\tensor L_n(\omega^{k+1}_{X/S} )
		\ar[r, "\text{id}\tensor \epsilon_{m,n}^{can}"]
		&
		\left(\cF\tensor L_n(\omega^{k+1}_{X/S}) \right)\tensor \cP^m_X
	\end{tikzcd}
	\]
	where $\{\epsilon_m\}_{m\in \N}$ is the stratification attached to $\cF$ relative to $\nabla$. The commutativity can be checked locally, using coordinates.

	
	The complex $L(\cF\otimes \omega_{X/S}^{\bullet})_{\rm inf}$ is also well defined via the functoriality of $L(-)_{\rm inf}$ and the fact that $\nabla^k$ are differential morphisms (of degree $1$).
	
	Now we want to check that the two complexes are isomorphic. We recall that
	\begin{multline*}
		L(\cF\otimes \omega_{X/S}^{k})_{\rm inf}= \lim_{n\in \N} L_n(\cF\otimes \omega_{X/S}^{k})_{\rm inf},
		\cF_{\rm inf}\ho L(\omega^k_{X/S})_{\rm inf}=\lim_{n\in \N} \cF_{\rm inf}\otimes L_n(\omega^k_{X/S})_{\rm inf},
	\end{multline*}
	where the two pro-quasi-coherent crystals in the limits are built via their stratification as in lemma \ref{LemmaCrystallization}. Then it suffices to show that for each $n\in \N$ there is an isomorphism 
	\[
	L_n(\cF\otimes_{\cO_X} \omega_{X/S}^{k})\cong \cF\otimes_{\cO_X}L_n(\omega^k_{X/S})
	\]
	commuting with the two stratifications and then check that this isomorphism commutes with the complex maps. The isomorphism is the following:
	\[
	L_n(\cF\otimes_{\cO_X} \omega_{X/S}^{k})=\cP^n_{X}\tensor\cF \tensor \Omega^k_{X/S)}\overset{\epsilon_n\tensor \text{id}_{\omega^k}}{\longrightarrow} \cF \tensor\cP^n_{X}\tensor \omega^k_{X/S)}= \cF\tensor L_n(\omega^k_{X/S})
	\]
	In order to check the compatibility of the stratification with the isomorphism we have to show that the following diagram is commutative.
	\[
	\begin{tikzcd}
		\cP^n_X\tensor L_m\left( \cF\tensor\omega^k_{X/S}\right)\ar[r, "\epsilon_{m,n}^{can}"]\ar[d, "\text{id}_{\cP^m_X}\tensor \epsilon_m\tensor \text{id}_{\omega^k_{X/S}}"]&  L_m\left( \cF\tensor\omega^k_{X/S}\right)\tensor \cP^n_X\ar[rd, "\epsilon_m\tensor\text{id}_{\omega^k_{X/S}\tensor \cP^n_X}"]&\\
		\cP^n_X \tensor \cF \tensor \left(\cP^m_X\tensor \omega^k_{X/S}\right)\ar[r, "\epsilon_n\tensor \text{id}"]&\cF\tensor \cP^n_X \tensor \cP^m_X\tensor \omega^{k}_{X/S}\ar[r, "\text{id}_{\cF}\tensor \epsilon^{can}"]&\left(\cF\tensor \cP^m_X\tensor \omega^k_{X/S}\right)\tensor \cP^n_X
	\end{tikzcd}
	\]
	The commutativity can be checked locally and it is enough to check it in the case $n=m$ via the compatibility condition on the stratifications and the surjectivity of the projection morphisms. We leave this calculation to the reader.

The last thing that we have to show is that the isomorphisms just built commute with the two complexes. This could be checked locally and only on $(X,X)$, the diagram over a thickening is (locally) the pullback of the diagram over $(X,X)$. The diagram is the following:
	\[
	\begin{tikzcd}[column sep=large, row sep=large]
		\cP^{n+1}_X\tensor \cF \tensor \omega^k_{X/S} \ar[r, "L(\nabla^k)_{m+1}"]\ar[d, "\epsilon_m\tensor \text{id}_{\omega^k_{X/S} }"]&
		\cP^{m}\tensor \cF \tensor \omega^{k+1}_{X/S}\ar[d, "\epsilon_m\tensor \text{id}_{\omega^{k+1}_{X/S}}"]
		\\
		\cF\tensor \cP^{m+1}_X\tensor\omega^{k}_{X/S}\ar[r, "\text{id}_{\cF}\tensor L(d^i)_m"]&
		\cF\tensor \cP^m_X\tensor \omega^{k+1}_{X/S}
	\end{tikzcd}.
	\]
	The commutativity of this diagram follows by the computations done in proposition \ref{PropExactLinearizedComplex}.
\end{proof}

We have the complex $L(\omega_{X/S}^{\bullet})_{\rm inf}$; as in proposition \ref{PropExactLinearizedComplex} we get an exact sequence.

\begin{theorem}\label{TheoremFinfExactSequence}
	The sequence
	\[
	0\rightarrow\cO_{\Xinf} \rightarrow L(\omega_{X/S}^{\bullet})_{\rm inf}
	\]
	is exact, moreover if $\cF$ is a flat quasi-coherent $\cO_X$-module with integrable connection $\nabla$, the sequences
	\[
	0\rightarrow \cF_{\rm inf} \rightarrow 	\cF_{\rm inf}\ho L(\omega_{X/S}^\bullet)_{inf}\quad \text{and}\quad 0\rightarrow \cF_{\rm inf} \rightarrow  L(\cF\tensor_{\cO_X}\omega_{X/S}^\bullet)_{\rm inf}
	\]
	are exact.
\end{theorem}
\begin{proof}
	A sequence of sheaves on the infinitesimal site is exact iff the sequence of sheaves associated to any thickening $(U,T)$ is exact in the topos $T_{\rm an}$. We can work locally, where $(U,T)$ has a section $g:T\rightarrow X$ and the sequence becomes
	\[
	0\rightarrow g^*\cO_{X} \rightarrow g^*L(\cO_{X/S}) \overset{g^* L(\de)}{\rightarrow } g^* L(\omega_{X/S}^1) \overset{g^* L(\de^1)}{\rightarrow }  g^* L(\omega_{X/S}^2) \rightarrow \ \cdots\ \ .
	\]
	Locally, where we have coordinates $x_1,\dots , x_d$ of $X$, the sequence is
	\[
	0\rightarrow \cO_{T} \rightarrow \cO_T[[\zeta_1,\dots , \zeta_d]] \rightarrow \cO_T[[\zeta_1,\dots , \zeta_d]]	\ho_{\cO_X} \omega^{1}_{X/S} \rightarrow  \cdots
	\]
	via the explicit computations in proposition \ref{PropositionLinearizedDifferentialIsExact}, it is exact.
	
	In order to see that 
	\[
	0\rightarrow \cF_{\rm inf} \rightarrow 	\cF_{\rm inf}\ho L(\omega_{X/S}^\bullet)_{\rm inf}
	\]
	is exact one works locally and as in proposition \ref{PropExactLinearizedComplex}, this sequence is limit of exact sequences and the Mittag-Lefler condition is (locally) satisfied.
\end{proof}

\subsubsection{De-linearization and the comparison theorem.}
\label{sec:delin}

The main goal of this section is to introduce a morphism of topoi
\[
u_{X/S}: \Xinf\longrightarrow X_{\rm an}
\]
where we denoted $X_{\rm an}$ the analytic topos of $X$, i.e. the category of sheaves of sets on $X$. We wish to show that for a flat, quasi-coherent $\cO_X$-module $\cF$ with an integrable connection $\nabla$, there is a canonical isomorphism 
\[
{\rm H}^i\left(\Xinf, \cF_{\rm inf}\right)\cong {\rm H}^i_{\rm dR}\left(X, (\cF, \nabla) \right).
\]

We denote  $\Xinfx= \Sh\left((X/S)_{\rm inf, {(X,X)}}\right)$, where $\left((X/S)_{\rm inf, {(X,X)}}\right)$ is the site with underlying category having as objects the morphisms $g:(U,T)\rightarrow (X,X)\in \XSinf$. A covering of $g$ is a collection of morphisms $\{g_i: (U_i,T_i)\rightarrow (X,X)\}_{i\in I}$ s.t. ${(U_i,T_i)}_{i\in I}$ is an admissible covering of $(U, T)$ in $\XSinf$. A sheaf $\cG\in \Xinfx$ is a (compatible) collection of sheaves $\cG_{g}$ on $T$, where $T$ and $g$ vary between all the morphisms of thickenings $g: (U,T)\rightarrow(X,X)$

\begin{definition}
	Let $\varphi:\Xinfx\rightarrow X_{\rm an}$ be the morphism of topoi given by
	\[
	\begin{tikzcd}[column sep=small, row sep=small]
		\varphi^{-1}: X_{an}  \ar[r] &\Xinfx&&	\varphi_*: \Xinfx \ar[r] & X_{an}\\
		\cE\ar[r, maps to]& \{\beta^{-1} \cE\}_{(\alpha, \beta): (U,T)\rightarrow (X,X)}&&\cG=\{\cG_{g}\}_{g:(U,T)\rightarrow (X,X)}\ar[r, maps to]& \cG_{\text{id}_{(X,X)}}
	\end{tikzcd}.
	\]
	Let $j: \Xinfx\rightarrow \Xinf$ be the morphism of topoi defined by
	\[
	\begin{tikzcd}[column sep=small, row sep=small]
		j^*: \Xinf  \ar[r] &\Xinfx\\
		F=\{F_{(U,T)}\}_{(U,T)}\ar[r, maps to]& \{F_{(U,T)}\}_{(U,T)\rightarrow (X,X)}
	\end{tikzcd}
	\]
	and the sheafification of
	\[
	\begin{tikzcd}[column sep=small, row sep=small]
		\tilde{j}_*': \Xinfx \ar[r] & \Psh\XSinf\\
		\cG=\{\cG_{g}\}_{g:(U,T)\rightarrow (X,X)}\ar[r, maps to]& \{\displaystyle\prod_{g:(U,T)\rightarrow (X,X)}\cG_{g}\}_{(U,T)}
	\end{tikzcd}.
	\]
	We denote the sheafification of $\tilde{j}_*'$ as $j_*'$.
	Let $u: \Xinf\rightarrow X_{\rm an}$ the morphism of topoi given by
	\[
	\begin{tikzcd}[column sep=small, row sep=small]
		u^{*}: X_{\rm an}  \ar[r] &\Xinf && 	u_*: \Xinf \ar[r] & X_{\rm an}\\
		\cE\ar[r, maps to]& \{ t_* \cE_{|_U}\}_{t: U\rightarrow T} && F=\{F_{(U,T)}\}_{(U,T)}\ar[r, maps to]& \{\Gamma(\Uinf, F_{|_{\Uinf} })\}_{U} 
	\end{tikzcd}.
	\]
\end{definition}
\begin{remark}
	The pairs in the above definitions define morphisms of topoi, i.e. one needs to check adjunctions between these pairs of functors.
	This is formal and follows as in section \S 5 of \cite{berthelot_ogus}.
\end{remark}

We get a diagram
\[
\begin{tikzcd}
	\Xinfx \ar[r, "\varphi"]\ar[d, "j"]& X_{ad}\\
	\Xinf\ar[ru, "u"]&
\end{tikzcd}.
\]
We have to check that $u^*, j^*, \varphi^{-1}$ commute with finite limits, 

For each nilpotent thickening $t: U\rightarrow T$ there is an arrow 
\[
		{\rm Hom}_{U_{\rm inf}}\left(\uno, F|_{U_{\rm inf}}\right)\rightarrow {\rm Hom}_{U_{\rm inf}}\left((h_{(U,T)})_{|_{U_{\rm inf}}}, F|_{U_{\rm inf}}\right)
		\]
		of sheaves in $U_{\rm an}$, where we recall that $h_{(U,T)}$ denotes the sheaf associated to $(U,T)$ via the Yoneda embedding. Via the push forward morphism $t_*$, we get a morphism
		\[
		s_{(U,T)}:(u^{-1}\cE)_{(U,T)}=t_* \cE_{|_U}\overset{t_*f_U}{\longrightarrow} t_* {\rm Hom}_{U_{\rm inf}}\left(\uno, F|_{U_{\rm inf}}\right)\rightarrow F_{(U,T)}
		\]
		This association is bijective since $t$ is a nilpotent thickening, hence $t_*:U_{\rm an}\rightarrow T_{\rm an}$ is an equivalence of categories. The inverse is given by taking a morphism $s=\{s_{(U,T)}\}$. For any section $x\in \cE(U)$ we get an element $x_{(U,T)}:= s_{(U,T)}(x)\in F(U,T)$ for each $(V,T)$ and $V\subset U$, this family $x_{(U,T)}$ gives a morphism $f_U(x)\in {\rm Hom}_{U_{inf}}\left(\uno, F_{|_{U_{inf}}}\right)$. Varying $U\subset X$, $x\in \cE(U)$ we get the morphism
		$
		f: \cE \rightarrow u_* F
		$ associated to $s$.

Then, as $t_*$ is an equivalence of categories and $\beta^{-1}$ commutes with finite limits it follows that $u^*, j^*, \varphi^{-1}$ commute with finite limits.

\begin{lemma}
	The diagram above is commutative.
\end{lemma}
\begin{proof}
	

The argument is formal and follows as in \cite{berthelot_ogus}.
	
\end{proof}

\begin{definition}
	Let $j_{!}^{\rm Ab}: \Xinfx^{\rm Ab}\rightarrow \Xinf^{\rm Ab}$ be the functor that associates to an abelian sheaf $\cG\in \Xinfx^{\rm Ab}$ the abelian sheaf defined by
	\[
	(j_{!}^{\rm Ab}\cG)_{(U,T)}:=\displaystyle\bigoplus_{g:(U,T)\rightarrow (X,X)} \cG_g\ .
	\]
	Let $j_{!}: \Xinfx\rightarrow \Xinf$ be the functor that associated to a sheaf $\cG\in \Xinfx$ the sheaf defined by
	\[
	(j_{!}\ \cG)_{(U,T)}:=\displaystyle\coprod_{g:(U,T)\rightarrow (X,X)} \cG_g\ .
	\]
\end{definition}
Observe that the functors $j_*$ and $j^*$ restrict to a pair of functors
\[
j^*: \Xinf^{\rm Ab}\leftrightarrows  \Xinfx^{\rm Ab}: j_*'
\]
that are adjoint.
Moreover $j_!^{\rm Ab}\leftrightarrows j^*$ are adjoint. Indeed a morphism $f\in{\rm  Hom}_{\Xinf^{\rm Ab}}\left(j_!^{\rm Ab} \cG, F\right)$ is a collection of morphisms $f_{(U,T)}: \bigoplus_g \cG_g\rightarrow F_{(U,T)}$ and it corresponds to a collection of morphisms $f_g: \cG_g\rightarrow F_{(U,T)}= (j^{*}F)_g$.

In analogue way $j_!\leftrightarrows j^{*}$.
\begin{remark}\label{Remarkj^*Exact}
	$j_!^{\rm Ab}\leftrightarrows j^{*}\leftrightarrows j_*'$, hence the functor $j^{*}$ commutes with limits and colimits, then $j^{*}$ is exact.
\end{remark}
\begin{lemma}
	If $\cE\in \Xinf^{\rm Ab}$ is an abelian sheaf, then for each $i\in \N$
	\[
	{\rm H}^i\left(\Xinfx, j^{*} \cE\right)\cong {\rm H}^i\left((X,X), \cE\right).
	\] 
\end{lemma}
\begin{proof}
	For $i=0$ the proof is an easy computation:
	\begin{multline*}
		\Gamma\left( (X,X), \cE\right)={\rm Hom}_{\Xinf}\left(h_{(X,X)}, \cE\right)= {\rm Hom}_{\Xinf}\left(j_! \uno, \cE\right)\\
		\cong {\rm Hom}_{\Xinfx}\left(\uno, j^{*}\cE \right)=\Gamma\left( \Xinfx, j^{*}\cE\right).
	\end{multline*}

	For $i>0$ observe that $j_!^{\rm Ab}$ is exact, hence his right adjoint $j^{*}$ takes injectives to injectives (lemma 12.29.1 \cite{stack_proj})and it is exact by the remark \ref{Remarkj^*Exact}, hence we can take an injective resolution $\cI^\bullet$ of $\cE$ and
	\[
	{\rm H}^i\left((X,X),\cE\right)={\rm H}^i\left( \cI^\bullet(X,X)\right)
	\]
	moreover (by the previous observations) $j^{*} \cI^\bullet$ is an injective resolution of $j^*\cE$ and 
	\[
	{\rm H}^i\left( \Xinfx, j^{*}\cE\right)= {\rm H}^i\left((j^{*}\cI^\bullet)(id_{(X,X)})\right)= {\rm H}^i\left(\cI^\bullet(X,X)\right).
	\]
	In the last equality we used the computation with $i=0$.
\end{proof}

Observe that $\varphi_*, \varphi^{-1}$ restrict to a pair of functors between abelian sheaves (that we also denote with $\varphi_*$ and $\varphi^{-1}$). Moreover also the pair
\[
\varphi^{-1}: X_{\rm ad}^{\rm Ab}\leftrightarrows \Xinfx^{\rm Ab}: \varphi_*
\]
is a pair of adjunction.

\begin{proposition}
	For any $F\in \Xinf$ we have that $\varphi_* j^* F\cong F_{(X,X)}$, moreover $\varphi_*$ is exact and for any abelian sheaf $F\in \Xinf^{Ab}$
	\[
	{\rm H}^i\left((X,X), F\right)\cong {\rm H}^i\left(X_{ad}, F_{(X,X)}\right).
	\]
\end{proposition}
\begin{proof}
	For $F\in \Xinf$, then
	\[
	\varphi_* j^* F=(j^*F)_{id_{(X,X)}}=F_{(X,X)}
	\]
	$\varphi_*$ preserves limits since it is right adjoint to the functor $\varphi^{-1}$. Moreover $\varphi_*$ preserves epimorphisms: $f: F_1\rightarrow F_2$ is an epimorphism of abelian sheaves $F_1,F_2\in \Xinfx^{Ab}$ iff for each $g:(U,T)\rightarrow (X,X)$ the map $f_g$ is an epimorphism of sheaves in $T_{\rm an}$, hence if $f$ is an epimorphism, then $\varphi_* f=f_{id_{(X,X)}}$ is an epimorphism too. Hence $\varphi_*$ preserves limits and epimorphisms, then it preserves all exact sequences and it is exact.
	
	For the last part of the statement let $F\in \Xinf^{\rm Ab}$, observe that $\varphi^{-1}$ commutes with colimits (since it's left adjoint) and finite limits, hence it is exact. Hence $\varphi_*$ takes injectives to injectives and it is exact, then the spectral sequence
	\[
	R^p \Gamma\left( X_{\rm ad}, -\right) \circ R^q \varphi_* \Rightarrow R^{p+q} \left( \varphi_* \circ\Gamma\left( X_{\rm ad}, -\right)\right)= R^{p+q} \left( \Gamma\left( \Xinfx, -\right)\right).
	\]
	degenerates to
	\[
	\left(R^i \Gamma\left( X_{\rm ad}, -\right) \right)\circ \varphi_*\cong R^{i} \left( \Gamma\left( \Xinfx, -\right)\right).
	\]
	Hence for any $F\in \Xinf^{\rm Ab}$
	\[
	{\rm H}^i((X,X), F)\cong H^i(\Xinfx, j^{-1} F)\cong {\rm H}^i(X_{\rm an}, \varphi_*j^{-1}F).
	\]
\end{proof}
\begin{corollary}
	For any $\cG\in \Xinfx^{Ab}$ the abelian sheaf $j_*\cG$ is acyclic for $u_*$.
\end{corollary}
\begin{proof}
	The functor $j^{*}$ commutes with colimits and finite limits, then it is exact, hence $j_*$ maps injective objects to injective objects and there is a spectral sequence
	\[
	R^p u_* R^q j_* \Rightarrow R^{p+q} (u\circ j)_*\cong R^{p+q} \varphi_*.
	\]
	But $j_*$ is also exact, via its explicit description and the fact that locally any $(U,T)$ has a section $(U,T)\rightarrow (X,X)$. Moreover $\varphi_*$ is exact, hence the spectral sequence degenerates to
	\[
	\left(R^i u_*\right)  j_* \cG= 0
	\]
	for any $i>0$ and $\cG\in \Xinfx$.
\end{proof}

Now we give another description of the functor $j_*'$.
\begin{definition}
	Let $j_*: \Xinfx\longrightarrow \Xinf$ be the functor that sends a sheaf $\cG\in \Xinfx$ to the sheaf
	\[
	j_*\cG:=\{\lim_{n\in \N} \left(p^n_T\right)_*\cG_{p^n_X}\}
	\]
	where $p^n_T: P^n_{U/U\times_S T}\rightarrow T$ and $p^n_X: D^n_{U}(U\times_S T)\rightarrow (X,X)$ are the two projections.
\end{definition}
\begin{lemma}
	There is a natural isomorphism $j_*\cong j_*'$; in particular for any $\cG\in \Xinfx^{Ab}$, the sheaf $j_*\cG$ is acyclic for $u_*$.
\end{lemma}
\begin{proof}
	First we want to describe, for a thickening $(U,T)$, the sheaf $j^{*}h_{(U,T)}$. Let $g:(U',T')\rightarrow (X,X)\in \XSinfx$, then the set
	\[
	\left(j^{*}h_{(U,T)}\right)(U',T')=Hom_{\XSinf}\left((U',T'), (U,T)\right)
	\]
	corresponds to the set of commutative diagrams between $(U',T')\rightarrow (X,X)$ and $(U, U\times_S T)\rightarrow (X,X)$, thanks to the Lemma \ref{lemmaD_X(Y)=Hom (-,Y)} we get
	\[
	\left(j^{*}h_{(U,T)}\right)(U',T')\cong\displaystyle\colim_{n\in \N}  {\rm Hom}_{\XSinfx}\left( (U',T')\rightarrow (X,X), D_U^n(U\times_S T)\rightarrow (X,X) \right).
	\]
	Then $j^{*}h_{(U,T)}= \colim_{n\in \N}  {\rm Hom}_{\XSinfx}\left( -\ , D_U^n(U\times_S T)\rightarrow (X,X) \right)$ is colimit of representable objects.
	
	Let $\cG\in \Xinfx$, then for any thickening $(U,T)$ we get
	\begin{multline*}
		(j_{*}'\cG){(U,T)}\cong {\rm Hom}_{\Xinf}\left(h_{(U,T)}, j_{*} \cG \right)\cong {\rm Hom}_{\Xinfx}( j^{-1} h_{(U,T)} , \cG)\\
		= \lim_{n\in \N} \cG\left( D^n_{U}(U\times_S T)\rightarrow (X,X)\right)= (j_*\cG)(U,T).
	\end{multline*}
	
\end{proof}

Observe that  $\varphi_*$ restricts to a functor
\[
\varphi_*: {\rm QCoh}\bigl(\cO_{\Xinfx}\bigr) \longrightarrow {\rm QCoh}\bigl(\cO_{X_{\rm ad}}\bigr)\ ,
\]
we can define $\varphi^*$, an analogue of $\varphi^{-1}$.
\begin{definition}
	Let $\varphi^*: {\rm QCoh}\bigl(\cO_{X_{\rm an}}\bigr)\longrightarrow {\rm QCoh}\bigl(\cO_{\Xinfx}\bigr)$ be the functor that associates to a quasi-coherent $\cO_{X}-$module $\cE$ the $\cO_{\Xinfx}-$module defined by
	\[
	(\varphi^*\cE)_{g}:= \beta^*\cE
	\]
	where $g=(\alpha,\beta):(U,T)\rightarrow (X,X)$.
\end{definition}
One can prove that $\varphi^*\leftrightarrows \varphi_*$ following the proof of the adjunction $\varphi^{-1}\leftrightarrows \varphi_{*}$.

\begin{lemma}
\label{lemma:cross}
Let $(U,{\cal M}_U,\alpha_U)$ and $(V,{\cal M}_V,\alpha_V)$ be affinoid adic spaces with fs log structures and with $U={\rm Spa}(A,A^+)$ and $V={\rm Spa}(B,B^+)$. Let $f\colon (U,{\cal M}_U,\alpha_U)\to(V,{\cal M}_V,\alpha_V)$ be a morphism of log adic spaces.

Let $J\subset A$ be an ideal such that $J^s=0$ and $p\colon A \rightarrow C:=A/J$. Let $W:={\rm Spa}(C,C^+)$ be the quotient adic space  with induced log structure $({\cal M}_W,\alpha_W)$ so that $\iota\colon (W,{\cal M}_W,\alpha_W)\to (U,{\cal M}_U,\alpha_U)$ is an exact closed immersion. Suppose that there is a section  $j\colon (U,{\cal M}_U,\alpha_U)\to (W,{\cal M}_W,\alpha_W) $ of $p$, defined over $(V,{\cal M}_V,\alpha_V)$.
	
	Let
	$I_1:=Ker\bigl((C\hat{\otimes}_BC)^{\rm ex} \rightarrow C\bigr)$ and $I_2:=Ker\bigl((A\hat{\otimes}_BC)^{\rm ex} \rightarrow C\bigr)$.
	Then, for every $n\ge 0$ we have
	$$
		A\hat{\otimes}_{i,C} \frac{(C\hat{\otimes}_BC)^{\rm ex}}{I_1^n}\cong 	 \frac{(A\hat{\otimes}_BC)^{\rm ex}}{I_2^n}.
	$$
  
  \end{lemma}

  \begin{proof} By \cite{diao_lan_liu_zhu} Prop. 2.3.21 and 2.3.22 we can take charts $M \to A$, $N \to B$, $Q\to C$ with maps of fs monoids $N\to M\to Q$ and a splitting $Q\to M$, as monoids over $N$, providing charts for $\iota$ and $j$. Since $\iota$ is exact we may take $M$ such that $M\cong Q\times_{Q^{\rm gp}} M^{\rm gp}$. The splitting $Q\to M$ provides a direct summand decomposition  $M^{\rm gp}=Q^{\rm gp}\oplus R$ for some abelian group $R$. Hence, $M\cong Q\oplus R $ as $Q\to Q^{\rm gp}$ is injective by assumption.  In conlcusion, we can take $M=Q$ as chart. We compute $(M\oplus_N Q)^{\rm ex}=(Q\oplus_N Q)^{\rm ex}=Q\oplus_N Q^{\rm gp}$.  Since $$(W\times_V W)^{\rm ex} \cong  (W\times_V W)\times_{V\langle Q\oplus_N Q\rangle}V\langle(Q\oplus_N Q)^{\rm ex}\rangle \cong (W\times_V W)\times_{V\langle Q\oplus_N Q\rangle}V\langle Q\oplus_N Q^{\rm gp}\rangle$$ and  $$(U\times_V W)^{\rm ex} \cong  (U\times_V W)\times_{V\langle M\oplus_N Q\rangle}V\langle (M\oplus_N Q)^{\rm ex}\rangle\cong    (U\times_V W)\times_{V\langle Q\oplus_N Q\rangle}V\langle Q\oplus_N Q^{\rm gp}\rangle,$$ By \cite{diao_lan_liu_zhu}  Ex. 2.3.26, we conclude that $(U\times_V W)^{\rm ex} \cong U\times_V (W\times_V W)^{\rm ex}$. The claim follows. 
  
\end{proof}

\

\begin{theorem}\label{TheoremLinearizationIsAyclicForU}
	For any flat $\cO_X-$module $\cE\in X_{\rm ad}$ there is an isomorphism
	\[
	L(\cE)_{inf}\cong j_*\varphi^* \cE,
	\]
	in particular $L(\cE)_{inf}$ is acyclic for $u_*$.
\end{theorem}
\begin{proof}
	For any object $g=(\alpha,\beta):(U,T)\rightarrow (X,X)\in \XSinfx$ one can define the morphism
	\[
	\left(j^* L(\cE)_{inf}\right)_{g}=\lim_{n\in \N} \beta^* L(\cE)_n \rightarrow \beta^* L(\cE)_0= \beta^*\cE = \left(\varphi^* \cE\right)_g.
	\]
	Since
	\[
	{\rm Hom}_{\Xinf}\left( L(\cE)_{inf},  j_*\varphi^* \cE  \right)\cong 	{\rm Hom}_{\Xinfx}\left( j^* L(\cE)_{inf},  \varphi^* \cE  \right),
	\]
	the previous morphism corresponds to a morphism $\mu: L(\cE)_{inf}\rightarrow j_*\varphi^* \cE$, we check (locally) that this morphism is an isomorphism. Let $(U,T)\in \XSinf$ with a section $\beta: T\rightarrow X$, then 
	\[
	(j_*\varphi^* \cE)_{(U,T)}= \lim_{n\in \N} (p^n_T)_* \left(\varphi^* \cE\right)_{p^n_X}= \lim_{n\in \N}(p^n_T)_* (p^n_X)^* \cE=\lim_{n\in \N}\left(\frac{(\cO_T\hat{\tensor}_{\cO_S} \cO_U)^{\rm ex}}{I^{n+1}_{U/U\times T}} \otimes_{\cO_X} \cE_{|_U}\right)
	\]
	where $I^{n+1}_{U/U\times T}= Ker\left(\cO_T\hat{\tensor}_{\cO_S} \cO_U\rightarrow \cO_U\right)$. On the other side
	\[
	L(\cE)_{inf,(U,T)}=\lim_{n\in \N} \left(\frac{ \left(\cO_T\hat{\tensor}_{\cO_S} \cO_U\right)}{I^{n+1}_{U/U\times U}} \otimes_{\cO_X} \cE_{|_U}\right),
	\]
	where $I^{n+1}_{U/U\times U}= Ker\left((\cO_U\hat{\tensor}_{\cO_S} \cO_U)^{\rm ex}\rightarrow \cO_U\right)$.
	
	The morphism $\mu$
	\[
	\lim_{n\in \N}\left(\frac{ \left(\cO_T\hat{\tensor}_{\cO_S} \cO_U\right)^{\rm ex}}{I^{n+1}_{U/U\times T}} \otimes_{\cO_X} \cE_{|_U}\right) \longrightarrow  \lim_{n\in \N} \left( \cO_T\ho_{\beta,\cO_U}\frac{ \left(\cO_U\hat{\tensor}_{\cO_S} \cO_U\right)^{\rm ex}}{I^{n+1}_{U/U\times U}} \otimes_{\cO_X} \cE_{|_U}\right)
	\]
	corresponds to the natural projection map, that via the Lemma \ref{lemma:cross} is an isomorphism.
\end{proof}


\begin{proposition}
	Given a flat $\cO_X$-module $\cF$ on $X$ with an integrable connection $\nabla$
	\[
	u_*(L(\cF)_{\rm inf})\cong \cF.
	\]
\end{proposition}
\begin{proof}
	The proof is a consequence of propositions \ref{PropGlobalSectionsCrystal} and \ref{PropExactLinearizedComplex}. More in detail, denote $F:= L(\cF)$ with the canonical stratification by proposition \ref{PropGlobalSectionsCrystal} and by definition of $F_{inf}$, for any open $U\subset X$, there are canonical identifications
	\begin{multline*}
		\Gamma\left(U_{\rm inf}, {(F_{\rm inf})}_{|_{U_{\rm inf}}}\right)\cong \{x\in {(F_{\rm inf})}_{|_{U_{\rm inf}}}(U,U)\mid \epsilon^1_{F_{\rm inf}}(p_0^1(1)^*x)= p_1^1(1)^* x\}\\
		=\{x\in F_{\rm inf}(U,U)\mid \epsilon^1_{F_{\rm inf}}(p_0^1(1)^*x)= p_1^1(1)^* x\}\\
		=\{x\in F(U,U)\mid \epsilon^1_F(p_0^1(1)^*x)= p_1^1(1)^* x\}.
	\end{multline*}
	But by definition of $\epsilon^1_F$ we have that
	\[
	\Gamma\left(U_{\rm inf}, {(F_{\rm inf})}_{|_{U_{\rm inf}}}\right)\cong {\rm Ker}(L(\nabla))(U).
	\]
	This isomorphism is functorial on $U$ and, via proposition \ref{PropExactLinearizedComplex}, we get
	\[
	u_*(F_{\rm inf})\cong {\rm Ker}(L(\nabla))\cong \cF.
	\]
\end{proof}
In the proof we used that the stratification $\{\epsilon^n_{F_{\rm inf}}\}_{n\in \N}$ attached to the sheaf $(F_{\rm inf})_{(X,X)}=F$ defined in \S\ref{SectionCrystal} is equal to the stratification $\{\epsilon^n_{F}\}_{n\in \N}$ that we used to define $F_{\rm inf}$. 

\begin{theorem}
	If $\cF$ is a flat $\cO_X$-module with integrable connection $\nabla$, then there is a canonical isomorphism
	\[
	{\rm H}^i\left(\Xinf, \cF_{\rm inf} \right)\cong {\rm H}^i_{dR}\left(X, (\cF, \nabla)\right)
	\]
\end{theorem}
\begin{proof}
	The sequence
	\[
	0\rightarrow \cF_{inf}\rightarrow L\left(\cF\tensor \Omega_{X/S}^\bullet\right)
	\]
	is exact via the Theorem \ref{TheoremFinfExactSequence}, then $\cF_{inf}$ is isomorphic to $L\left(\cF\tensor \Omega_{X/S}^\bullet\right)$ in the derived category. Moreover via the Theorem \ref{TheoremLinearizationIsAyclicForU} $L(\cF\tensor_{\cO_X} \Omega_{X/S}^\bullet)_{inf}$ is done by acyclic objects for $u_*$ and the Lerray spectral sequence
	\[
	R^p\Gamma\left(X,-\right) R^q u_*\ \left( L\left(\cF\tensor \Omega_{X/S}^\bullet\right) \right)\Rightarrow R^{p+q} \Gamma\left(\Xinf,-\right)  \left(L\left(\cF\tensor \Omega_{X/S}^\bullet\right)\right)
	\]
	degenerates, then
	\begin{multline*}
		R^i\Gamma\left(\Xinf,-\right)\left(\cF_{inf}\right)\cong  R^i \Gamma\left(\Xinf,-\right) \left(L\left(\cF\tensor \Omega_{X/S}^\bullet\right)\right)
		\\
		\cong  R^i \Gamma\left(X,-\right)\circ u_* \left(L\left(\cF\tensor \Omega_{X/S}^\bullet\right) \right)
		\cong R^i\Gamma\left(X,-\right)  \left(\cF\tensor \Omega_{X/S}^\bullet\right).
	\end{multline*}
\end{proof}

The functor $u_\ast$ has a left adjoint and it is not exact, therefore it does not commute with general colimits. Nevertheless, we have the following:
 
\begin{lemma}\label{lemma:uastcolim}
	Let $\{F_N\}_{N\in \N}$ be a direct system of coherent and locally free $\cO_X$-modules such that $F_N\rightarrow F_{N+1}$ is a monomorphism for each $N\in \N$. Suppose that there is an integrable connection $\nabla$ on $F:= \colim_{N}F_N$. Then
	\[
		u_\ast \left(\colim_N L( F_N )_{\inf }\right)\cong  u_\ast  L( \colim_N F_N)_{\inf}	\cong \colim_N F_N.
	\]
\end{lemma}
\begin{proof}
	Firstly we prove that $\phi: \left(\colim _NL( F_N )_{\inf}\right)(U,U)\rightarrow L( \colim_N F_N)_{\inf} (U,U)$ is injective for each $(U,U)\in \XSinf$.
	
	We can cover each $(U,U)\in\XSinf$ with objects of the same shape, hence we can work locally. Let $F:=\colim_N F_N$, then
	\[
	\left(\colim_N L( F_N )_{\inf}\right) (U,U)=\colim_N \lim_n L_n(F_N)(U,U)=\colim_N \lim_n \cP^n_U\tensor F_N(U)
	\] 
	and
	\[
	L( \colim_N F_N)_{\inf} (U,U)=\lim_n L_n\left(\colim_N F_N\right)_{\inf}(U,U)
	=\lim_n \left(\cP^n_U\tensor F(U)\right).
	\]
	Let $x\in \colim_N \lim_n \cP^n_U\tensor F(U)$, then $x=(x_n)_{n\in \N}\in \lim_n \cP^n_U\tensor F_M(U)$ for an $M\in \N$. Since $F_{M}\hookrightarrow F$ via the hypothesis and each $\cP^n_U$ is flat we get that 
	\[
		\phi_n:\cP^n_U \otimes  F_N(U)\hookrightarrow \cP^n_U \otimes  F(U)
	\]
	is injective for each $n\in \N$.
	If $\phi(x)=0,$ then $\phi_n(x_n)=0\in \cP^n_U \otimes  F(U)$ for each $n\in \N$ and we get that $x=0$; hence $\phi$ is injective.
	
	Let's consider the following diagram induced via the two projections of $X\times X$ into $X$
	\[
		\begin{tikzcd}
			\left(\colim_N L(F_N)_{\inf}\right)(U,U) \ar[r,shift right=1.25]\ar[r,shift left=1.25]\ar[d, hook, "\phi"]&\left(\colim_N L(F_N)\right)(U, \cP^n_U)\ar[d]\\
			L(F)_{\inf}(U,U)\ar[r,shift right=1.25]\ar[r,shift left=1.25]&L(F)_{\inf}(U, \cP^n_U)
		\end{tikzcd}.
	\]
	Since each $F_N$ is locally free and coherent we get that $F$ is a flat $\cO_{\Xinf}$-module with an integrable connection; then by the previous Theorem and the explicit computation of the global sections for a sheaf in $\Xinf$
	\[
		{\rm Eq}\left(	L(F)_{\inf}(U,U)\rightrightarrows L(F)_{\inf}(U, \cP^n_U)\right)\cong \left(u_\ast L(F)_{\inf} \right)(U)\cong F(U).
	\] 
	Since $\phi$ is injective we get that
	\begin{align*}
		F(U)\subseteq u_\ast \left(\colim_N L( F_N )_{\inf }\right)(U)\cong {\rm Eq}\left(	\left(\colim_N L(F_N)_{\inf}\right)(U,U)\rightrightarrows\left(\colim_N L(F_N)\right)(U, \cP^n_U)\right)\\
		\subseteq {\rm Eq}\left(	L(F)_{\inf}(U,U)\rightrightarrows L(F)_{\inf}(U, \cP^n_U)\right)\cong F(U).
	\end{align*}

\end{proof}


\bigskip

  \begin{thebibliography}{99}

 

\bibitem[AI]{andreatta_iovita}  F.~Andreatta,  A.~Iovita:
\emph{  Triple product $p$-adic L-functions associated to finite
slope $p$-adic families of modular forms}, Duke Math.~J.~{\bf 170},  1989--2083 (2021).

\bibitem[AIPHS]{halo_spectral} F.~Andreatta, A.~Iovita, V.~Pilloni: \emph{  Le Halo Spectral}, Ann.~Sci.~ENS, {\bf 51}, 603-656 (2018).

\bibitem[ICM18]{ICM} F.~Andreatta, A.~Iovita, V.~Pilloni: \emph{$p$-Adic variation of automorphic sheaves.} ICM. Rio de Janeiro, 2018.

\bibitem[AIS]{andreatta_iovita_stevens}  F.~Andreatta,  A.~Iovita, G.~Stevens: \emph{Overconvergent modular sheaves and modular forms for ${\rm GL}_{2/F} $}, Israel J. of Mathematics {\bf 201}, 299--359 (2014).

\bibitem[Ay]{andreychev} G.~Andreychev: \emph{Pseudocoherent and Perfect Complexes and Vector Bundles on Analytic Adic Spaces}, preprint, arXiv:2311.04394 (2021).

\bibitem[BGG]{bernstein_gelfand_gelfand} I. N. Bernstein, I. M. Gelfand, S. I. Gelfand, \emph{Differential operators on the base affine space and a study of $\fg$-modules}, in Lie groups and their Representation theory, A first course, Budapest (1971), London (1975).

 \bibitem[BO]{berthelot_ogus} P.~Berthelot, A.~Ogus: \emph{Notes on crystalline cohomology}, Princeton University Press, Princeton, N.J., (1978).

\bibitem[BCGP]{boxer_calegari_gee_pilloni} G.~Boxer, F.~Calegari, T.~Gee, V.~Pilloni: \emph{Modularity theorems for abelian surfaces}, preprint, arXiv:2502.20645 (2025).

\bibitem[CF]{chiarellotto_fornasiero} B. Chiarellotto, M. Fornasiero: \emph{Logarithmic de Rham, Infinitessimal and Betti cohomologies}, J. Math. Sci. Univ. Tokyo, 13, 205-257, (2006).

 \bibitem[CS]{clausen_scholze} D.~Clausen, P.~Scholze: \emph{Condensed Mathematics}, preprint, (2019)


\bibitem[Co]{ColemanPrimitive} R.~Coleman: \emph{ A $p$-adic Shimura isomorphism and $p$-adic periods of modular forms}, in ``$p$-adic monodromy and the Birch and Swinnerton-Dyer conjecture" (Boston, MA, 1991), 21–51, Contemp. Math. {\bf 165}, Amer. Math. Soc., Providence, RI, (1994). 

\bibitem[Cd]{conrad} B.~Conrad: \emph{Relative ampleness in rigid geometry}, Ann.~Inst.~Fourier {\bf 54} 1049--1126 (2006).

\bibitem[D]{diximier} J. Diximier: \emph{Alg\`ebgre enveloppantes}, Gauthier-Villars, Paris (1974); reprint of English translation, \emph{Envelopping Algebras}, Amer. Math. Soc., Providence, RI, (1996).
 
 \bibitem[DLLZ]{diao_lan_liu_zhu} Diao, Hansheng; Lan, Kai-Wen; Liu, Ruochuan; Zhu, Xinwen: \emph{Logarithmic adic spaces: some foundational results.}, $p$-adic Hodge theory, 
 singular varieties, and non-abelian aspects, 65–182. Simons Symp.Springer, Cham, (2023).
 
 
\bibitem[F]{faltings} G.~Faltings: \emph{On the cohomology of locally symmetric hermitian spaces}, LNM, vol 1029, Paris (1982), Springer-Verlang Berlin, Heidelberg, New York, (1983).

\bibitem[F-C]{faltings_chai}  G.~Faltings, C.-L.~Chai: \emph{Degeneration of abelian varieties}, Ergebnisse der Mathematik und ihrer Grenzgebiete {\bf 22}, Springer-Verlang, Berlin, Heidelberg, New York, (1991).


\bibitem[GL]{garland_lepowsky} H.~Garland, J.~Lepowsky: \emph{Lie algebra homology and the Macdonald-Kac formulas}, Invent~ Math.~{\bf 34}, 37--76 (1976).


\bibitem[Gr]{grothendieck} A.~Grothendieck: \emph{Crystals and de Rham cohomology of schemes} Notes by I. Coates and O. Jussila. Adv.~Stud.~Pure Math.~{\bf 3}, Dix expos\'es sur la cohomologie des sch\'emas, 306--358, North-Holland, Amsterdam, 1968.


\bibitem[G]{guo} H.~Guo: \emph{Crystalline cohomology of rigid analytic spaces}, to appear in Bull. Soc. Math. France.

\bibitem[Hu]{humphreys} J. E. Humphreys: \emph{Representations of Semisimple Lie Algebras in the BGG Category $\cO$}, GTM {\bf 94}, (1991)

\bibitem[H]{Huber} R.~Huber: \emph{  \'Etale Cohomology of Rigid Analytic Varieties and Adic Spaces}, Aspects Math., {\bf E30} Friedr. Vieweg \& Sohn, Braunschweig, 1996, x+450 pp.

\bibitem[K1]{katzpadic} N.~Katz: \emph{$p$-adic properties of modular schemes and modular forms}, In ``Modular functions of one variable III" LNM \textbf{350},  69--190, Springer, Berlin, 1973.

\bibitem[KWL-P]{kwlan_polo} K.-W.~Lan, P.~Polo: \emph{Dual BGG complexes for automorphic bundles} Math.~Res.~Lett.~{\bf 25}, 85--141 (2018).

\bibitem[Le]{lepowsky} J.~Lepowsky: \emph{A generalization of the Bernstein-Gelfand-Gelfand resolution}, Journal of Algebra~{\bf 49}, 496--511 (1977).


\bibitem[LP1]{lue_pan} L.~Pan: \emph{On locally analytic vectors of the completed cohomology of modular curves I}, Forum of Mathematics, Pi {\bf 10}  (2022).



\bibitem[PT]{polo_tilouine} P.~Polo, J.~Tilouine, \emph{Bernestein-Gelfand-Gelfand complex and cohomology of nilpotent groups over $\Z_{(p)}$ for representations with $p$-small weights}, in Cohomology of Siegel varieties, Asterisque {\bf 280,} Soci\'et\'e Math\'ematique de France, Paris (2002).

\bibitem[RC]{camargo} J.~E.~Rodriguez Camargo: \emph{Geometric Sen Theory over rigid analytic spaces.}, preprint, (2022)

\bibitem[Sh]{shiho} A.~Shiho: \emph{Crystalline Fundamental Groups I, - Isocrystals on Log Crystalline Site and Log Convergent Site}, J. Math. Sci. Tokyo 7, 509-656, (2000). 

\bibitem[S]{snoor} A.~Snoor: \emph{Quasicoherent sheaves for dagger analytic geometry}, preprint, (2024)

\bibitem[Stack-proj]{stack_proj} The Stacks Project Authors: \emph{Derived categories}, Stacks Project, section 21, http://stacks.math.columbia.edu (2013).


\bibitem[T]{tilouine} J.~Tiluouine, \emph{Formes compagnons et complexes BGG pour ${\rm GSp}_4$.}, Ann.~Inst.~Fourier {\bf 62}, 1383--1436, (2012).
 
 \bibitem[Ur]{UNO} E.~Urban: \emph{  Nearly overconvergent modular forms}, Iwasawa theory 2012, 401-441, Contrib. Math. Comput. Sci. {\bf 7}, Springer, Heidelberg, (2014).

\bibitem[Z]{zink} T.~Zink: \emph{ Cartiertheorie commutativer formaler Gruppen}, Teubner Texte zar Mathematik, {\bf 68}, Leipzig: Teubner (1984).  

 

\end{thebibliography}
\end{document}